\documentclass[10pt]{amsart}
\usepackage{amsmath}
\usepackage{amsfonts}
\usepackage{amsxtra}

\usepackage{amscd}

\usepackage{amsthm}
\usepackage{amssymb}
\usepackage{eucal}
\usepackage[all]{xy}
\usepackage{graphicx}

\usepackage{xcolor}

\usepackage{enumitem}

\usepackage[hidelinks]{hyperref}

\makeatletter
\def\amsbb{\use@mathgroup \M@U \symAMSb}
\makeatother

\usepackage{bbold}
\let\mbb\mathbb
\renewcommand{\mathbb}{\amsbb}

\usepackage[capitalise]{cleveref}
\crefformat{enumi}{#2\textup{(#1)}#3}
\crefformat{equation}{(#2#1#3)}

\numberwithin{equation}{section}

\newcommand{\nc}{\newcommand}
\nc{\renc}{\renewcommand}
\nc{\ssec}{\subsection}
\nc{\sssec}{\subsubsection}
\nc{\on}{\operatorname}

\newtheorem{cor}[subsubsection]{Corollary}
\newtheorem{lem}[subsubsection]{Lemma}
\newtheorem{prop}[subsubsection]{Proposition}

\newtheorem{thm}[subsubsection]{Theorem}

\theoremstyle{plain}

\newtheorem{Thm*}{Theorem}[section]
\newtheorem{Thm'}[Thm]{"Theorem"}

\newcommand{\thmref}[1]{Theorem~\ref{#1}}

\newcommand{\secref}[1]{Sect.~\ref{#1}}
\newcommand{\lemref}[1]{Lemma~\ref{#1}}
\newcommand{\propref}[1]{Proposition~\ref{#1}}
\newcommand{\corref}[1]{Corollary~\ref{#1}}

\theoremstyle{definition}

\newcommand{\wt}{\widetilde}

\newcommand{\Tr}{\operatorname{Tr}}

\newcommand{\Corr}{\operatorname{Corr}}

\theoremstyle{definition}

\theoremstyle{remark}
\newtheorem{rem}[subsubsection]{Remark}

\nc{\ZZ}{{\mathbb Z}}
\nc{\NN}{{\mathbb N}}
\nc{\OO}{{\mathbb O}}
\renc{\SS}{{\mathbb S}}
\nc{\DD}{{\mathbb D}}
\nc{\GG}{{\mathbb G}}

\nc{\Fq}{{\mathbb F}_q}
\nc{\Fqb}{\ol{{\mathbb F}_q}}
\nc{\Ql}{\ol{{\mathbb Q}_\ell}}

\nc{\id}{\text{id}}

\nc{\BA}{{\mathbb{A}}}
\nc{\BC}{{\mathbb{C}}}
\nc{\BE}{{\mathbb{E}}}
\nc{\BF}{{\mathbb{F}}}
\nc{\BG}{{\mathbb{G}}}
\nc{\BM}{{\mathbb{M}}}
\nc{\BO}{{\mathbb{O}}}
\nc{\BD}{{\mathbb{D}}}
\nc{\BL}{{\mathbb{L}}}
\nc{\BN}{{\mathbb{N}}}
\nc{\BP}{{\mathbb{P}}}
\nc{\BQ}{{\mathbb{Q}}}
\nc{\BR}{{\mathbb{R}}}
\nc{\BV}{{\mathbb{V}}}
\nc{\BW}{{\mathbb{W}}}
\nc{\BZ}{{\mathbb{Z}}}
\nc{\BS}{{\mathbb{S}}}

\nc{\CA}{{\mathcal{A}}}
\nc{\CB}{{\mathcal{B}}}

\nc{\CE}{{\mathcal{E}}}
\nc{\CF}{{\mathcal{F}}}
\nc{\CG}{{\mathcal{G}}}
\nc{\CH}{{\mathcal{H}}}

\nc{\CL}{{\mathcal{L}}}
\nc{\CC}{{\mathcal{C}}}
\nc{\CM}{{\mathcal{M}}}
\nc{\CN}{{\mathcal{N}}}
\nc{\CK}{{\mathcal{K}}}
\nc{\CO}{{\mathcal{O}}}
\nc{\CP}{{\mathcal{P}}}
\nc{\CQ}{{\mathcal{Q}}}
\nc{\CR}{{\mathcal{R}}}
\nc{\CS}{{\mathcal{S}}}
\nc{\CT}{{\mathcal{T}}}
\nc{\CU}{{\mathcal{U}}}
\nc{\CV}{{\mathcal{V}}}
\nc{\CW}{{\mathcal{W}}}
\nc{\CX}{{\mathcal{X}}}
\nc{\CY}{{\mathcal{Y}}}
\nc{\CZ}{{\mathcal{Z}}}
\nc{\CI}{{\mathcal{I}}}
\nc{\CJ}{{\mathcal{J}}}

\nc{\fa}{{\mathfrak{a}}}
\nc{\fb}{{\mathfrak{b}}}
\nc{\fd}{{\mathfrak{d}}}
\nc{\ff}{{\mathfrak{f}}}
\nc{\fg}{{\mathfrak{g}}}
\nc{\fgl}{{\mathfrak{gl}}}
\nc{\fh}{{\mathfrak{h}}}
\nc{\fj}{{\mathfrak{j}}}
\nc{\fl}{{\mathfrak{l}}}
\nc{\fm}{{\mathfrak{m}}}
\nc{\fn}{{\mathfrak{n}}}
\nc{\fu}{{\mathfrak{u}}}
\nc{\fp}{{\mathfrak{p}}}
\nc{\fr}{{\mathfrak{r}}}
\nc{\fs}{{\mathfrak{s}}}
\nc{\ft}{{\mathfrak{t}}}
\nc{\fz}{{\mathfrak{z}}}
\nc{\fsl}{{\mathfrak{sl}}}
\nc{\hsl}{{\widehat{\mathfrak{sl}}}}
\nc{\hgl}{{\widehat{\mathfrak{gl}}}}
\nc{\hg}{{\widehat{\mathfrak{g}}}}
\nc{\htt}{{\widehat{\mathfrak{t}}}}
\nc{\chg}{{\widehat{\mathfrak{g}}}{}^\vee}
\nc{\hn}{{\widehat{\mathfrak{n}}}}
\nc{\chn}{{\widehat{\mathfrak{n}}}{}^\vee}

\nc{\fA}{{\mathfrak{A}}}
\nc{\fB}{{\mathfrak{B}}}
\nc{\fD}{{\mathfrak{D}}}
\nc{\fE}{{\mathfrak{E}}}
\nc{\fF}{{\mathfrak{F}}}
\nc{\fG}{{\mathfrak{G}}}
\nc{\fK}{{\mathfrak{K}}}
\nc{\fL}{{\mathfrak{L}}}
\nc{\fM}{{\mathfrak{M}}}
\nc{\fN}{{\mathfrak{N}}}
\nc{\fP}{{\mathfrak{P}}}
\nc{\fU}{{\mathfrak{U}}}
\nc{\fV}{{\mathfrak{V}}}
\nc{\fX}{{\mathfrak{X}}}
\nc{\fY}{{\mathfrak{Y}}}
\nc{\fZ}{{\mathfrak{Z}}}

\nc{\ba}{{\mathbf{a}}}
\nc{\bc}{{\mathbf{c}}}
\nc{\bd}{{\mathbf{d}}}
\nc{\bbf}{{\mathbf{f}}}
\nc{\be}{{\mathbf{e}}}
\nc{\bi}{{\mathbf{i}}}
\nc{\bj}{{\mathbf{j}}}
\nc{\bn}{{\mathbf{n}}}
\nc{\bo}{{\mathbf{o}}}
\nc{\bp}{{\mathbf{p}}}
\nc{\bq}{{\mathbf{q}}}
\nc{\bu}{{\mathbf{u}}}
\nc{\bv}{{\mathbf{v}}}
\nc{\bx}{{\mathbf{x}}}
\nc{\bs}{{\mathbf{s}}}
\nc{\by}{{\mathbf{y}}}
\nc{\bw}{{\mathbf{w}}}
\nc{\bA}{{\mathbf{A}}}
\nc{\bK}{{\mathbf{K}}}
\nc{\bB}{{\mathbf{B}}}
\nc{\bF}{{\mathbf{F}}}
\nc{\bC}{{\mathbf{C}}}
\nc{\bG}{{\mathbf{G}}}
\nc{\bD}{{\mathbf{D}}}
\nc{\bE}{{\mathbf{E}}}
\nc{\bH}{{\mathbf{H}}}
\nc{\bI}{{\mathbf{I}}}
\nc{\bL}{{\mathbf{L}}}
\nc{\bM}{{\mathbf{M}}}
\nc{\bN}{{\mathbf{N}}}
\nc{\bO}{{\mathbf{O}}}
\nc{\bV}{{\mathbf{V}}}
\nc{\bW}{{\mathbf{W}}}
\nc{\bX}{{\mathbf{X}}}
\nc{\bZ}{{\mathbf{Z}}}
\nc{\bS}{{\mathbf{S}}}

\nc{\sA}{{\mathsf{A}}}
\nc{\sB}{{\mathsf{B}}}
\nc{\sC}{{\mathsf{C}}}
\nc{\sD}{{\mathsf{D}}}
\nc{\sF}{{\mathsf{F}}}
\nc{\sK}{{\mathsf{K}}}
\nc{\sM}{{\mathsf{M}}}
\nc{\sO}{{\mathsf{O}}}
\nc{\sW}{{\mathsf{W}}}
\nc{\sQ}{{\mathsf{Q}}}
\nc{\sP}{{\mathsf{P}}}
\nc{\sZ}{{\mathsf{Z}}}
\nc{\sr}{{\mathsf{r}}}
\nc{\bk}{{\mathsf{k}}}
\nc{\sg}{{\mathsf{g}}}
\nc{\sff}{{\mathsf{f}}}
\nc{\sfe}{{\mathsf{e}}}
\nc{\sfj}{{\mathsf{j}}}
\nc{\sfb}{{\mathsf{b}}}
\nc{\sfc}{{\mathsf{c}}}
\nc{\sd}{{\mathsf{d}}}
\nc{\sv}{{\mathsf{v}}}

\nc{\Coh}{{\mathsf{Coh}}}
\nc{\QCoh}{{\on{QCoh}}}
\nc{\IndCoh}{{\on{IndCoh}}}

\nc{\hattimes}{\wh\otimes}

\nc{\bh}{{\bar{h}}}

\nc{\comult}{{co\text{-}mult}}
\nc{\counit}{{co\text{-}unit}}
\nc{\dgSch}{\on{Sch}}
\nc{\Sch}{\on{Sch}}
\nc{\affdgSch}{\on{Sch}^{\on{aff}}}
\nc{\affSch}{\on{Sch}^{\on{aff}}}
\nc{\Groupoids}{\on{Grpd}}
\nc{\inftygroup}{\on{Spc}}
\nc{\inftyCat}{\infty\on{-Cat}}
\nc{\StinftyCat}{\inftyCat^{\on{St}}}
\nc{\MoninftyCat}{\infty\on{-Cat}^{\on{Mon}}}
\nc{\SymMoninftyCat}{\infty\on{-Cat}^{\on{SymMon}}}
\nc{\SymMonStinftyCat}{\on{DGCat}^{\on{SymMon}}}
\nc{\MonStinftyCat}{\on{DGCat}^{\on{Mon}}}
\nc{\inftystack}{\on{Stk}}
\nc{\inftystackalg}{Stk^{1\text{-}alg}}
\nc{\inftyprestack}{\on{PreStk}}
\nc{\inftydgnearstack}{\on{NearStk}}
\nc{\inftydgstack}{\on{Stk}}
\nc{\inftydgstackalg}{DGStk^{1\text{-}alg}}
\nc{\inftydgprestack}{\on{PreStk}}
\nc{\dgindSch}{\on{indSch}}
\nc{\indSch}{{}^{\on{cl}}\!\on{indSch}}
\nc{\infSch}{\on{infSch}}
\nc{\dr}{{\on{dR}}}

\nc{\mmod}{{\on{-}\!{\mathbf{mod}}}}

\nc{\starr}{\text{\dh}}
\nc{\Spectra}{\on{Spectra}}
\nc{\Crys}{\on{Crys}}
\nc{\oblv}{{\mathbf{oblv}}}
\nc{\ind}{{\mathbf{ind}}}
\nc{\coind}{{\mathbf{coind}}}
\nc{\inv}{{\mathbf{inv}}}
\nc{\triv}{{\mathbf{triv}}}
\nc{\CMaps}{{\mathcal Maps}}
\nc{\Maps}{\on{Maps}}
\nc{\bMaps}{\mathbf{Maps}}
\nc{\BMaps}{\ul{\on{Maps}}}
\nc{\Ran}{{\on{Ran}}}
\nc{\Vect}{\on{Vect}}
\nc{\Shv}{\on{Shv}}
\nc{\Spc}{{\on{Spc}}}
\nc{\ppart}{(\!(t)\!)}
\nc{\qqart}{[\![t]\!]}
\nc{\Dmod}{\on{D-mod}}
\nc{\DGCat}{\on{DGCat}}
\renc{\det}{\on{det}}
\nc{\Conf}{\on{Conf}}

\nc{\uno}{\mbb{1}}
\nc{\Cat}{\on{Cat}}

\nc{\uShv}{\underline{\on{Shv}}}
\nc{\uhom}{\underline{\on{hom}}}
\nc{\uHom}{\underline{\on{Hom}}}
\nc{\AGCat}{\on{AGCat}}
\nc{\ul}{\underline}
\nc{\enhOV}{\on{enh}(\bO,\CV)}
\nc{\EnhOV}{\on{Enh}(\bO,\CV)}
\nc{\wtF}{\ul{F}}
\nc{\ModtoEnr}{\on{Mod-to-Enr}}
\nc{\bP}{\mathbf{P}}
\nc{\Quiv}{\on{Quiv}}
\renc{\mod}{\on{-mod}}
\nc{\Hom}{\on{Hom}}
\nc{\CHom}{\CH\on{om}}
\nc{\Spec}{\on{Spec}}
\nc{\PreStk}{\on{PreStk}}
\nc{\ol}{\overline}
\nc{\sotimes}{\overset{!}\otimes}
\nc{\Frob}{{\on{Frob}}}
\nc{\sFunct}{\mathsf{Funct}}
\nc{\Funct}{\on{Funct}}
\nc{\sfunct}{\mathsf{funct}}
\nc{\cG}{{\check{G}}}
\nc{\LS}{\on{LS}}
\nc{\sft}{\mathsf{t}}

\begin{document}

\begin{abstract}
In this paper we record the formalism of \emph{algebro-geometric} DG categories
(in short $\AGCat$) following a suggestion of V.~Drinfeld.
This formalism will be applied to ``real-world" problems in papers sequel to this one,
\cite{GRV2} and \cite{GRV3}.
\end{abstract}

\date{\today}

\title[Applications of (higher) categorical trace I]{Applications of (higher) categorical trace I: \\
the definition of $\on{AGCat}$}
\author{Dennis Gaitsgory}
\author{Nick Rozenblyum}
\author{Yakov Varshavsky}


\maketitle

\tableofcontents

\section*{Introduction}

\ssec{What is this paper about?}

This paper does not contain deep \emph{original} results. Our goal has been to write down a certain construction,
proposed by V.~Drinfeld some 20 years ago, that provides an ``artificial" fix to the failure of the
\emph{categorical Künneth formula}.

\sssec{}

For a scheme $X$ of finite type over a ground field $k$, let $\Shv(X)$ denote the (ind-complete version of the) DG category
of $\ell$-adic sheaves on $X$, defined as in \cite[Sect. 1.1]{AGKRRV1}. For a pair of schemes $X_1$ and $X_2$,
the operation of \emph{external tensor product} gives rise to a functor
\begin{equation} \label{e:boxtimes intro}
\Shv(X_1)\otimes \Shv(X_2)\to \Shv(X_1\times X_2),
\end{equation}
where the $\otimes$ is Lurie's tensor product of DG categories.

\medskip

The point of departure of Drinfeld's proposal and of this paper is that the functor \eqref{e:boxtimes intro} is \emph{not}
an equivalence (unless one of the schemes $X_1$ or $X_2$ is a disjoint union of points). However, there are multiple
reasons (see \secref{ss:motivations}) for which one would want to ``turn" \eqref{e:boxtimes intro} into an equivalence.
as explained below.

\medskip

The device that turns \eqref{e:boxtimes intro} into an equivalence (see below) is basically a formal manipulation,
which may seem too naive to produce
anything really useful. But it happens to be deeper than it looks. At least, it helps one solve problems that exist
outside of this formalism (see Sects. \ref{sss:useful1} and \ref{sss:useful2}).

\sssec{}

The formal structure of turning \eqref{e:boxtimes intro} into an equivalence will be as follows: we will define a new
symmetric monoidal $(\infty,2)$-category, to be denoted $\AGCat$, equipped with the following data:

\begin{itemize}

\item A symmetric monoidal functor $\bi:\DGCat\to \AGCat$;

\medskip

\item A symmetric monoidal functor $\ul\Shv: \on{Corr}(\Sch)\to \AGCat$, where
$\on{Corr}(\Sch)$ is the \emph{category of correspondences} of schemes, see \secref{ss:corr};

\medskip

\item A symmetric monoidal transformation $$\alpha:\bi\circ \Shv(-)\to \ul\Shv,$$
where $\Shv(-): \on{Corr}(\Sch)\to \DGCat$
is the original functor $X\mapsto \Shv(X)$ equipped with the \emph{right-lax} symmetric monoidal structure
encoded by \eqref{e:boxtimes intro}.

\end{itemize}

\medskip

The category $\AGCat$ is constructed so that it is universal with respect to the possessing the structure of
$(\bi,\ul\Shv(-),\alpha)$ above.

\medskip

Moreover, it will follow from the construction that the natural transformation
$$\Shv(-)\to \bi^R\circ \ul\Shv(-),$$
arising from $\alpha$, is an isomorphism.

\sssec{}

The fact that $\ul\Shv(-)$ is \emph{strictly} symmetric monoidal means that we have \emph{isomorphisms}
\begin{equation} \label{e:boxtimes intro u}
\uShv(X_1)\otimes \uShv(X_2)\to \uShv(X_1\times X_2)
\end{equation}
in $\AGCat$, which replace the non-isomorphisms \eqref{e:boxtimes intro}.

\sssec{}

One can describe objects of $\AGCat$ explicitly. Namely, an object $\ul\bC\in \AGCat$
is an assignment
$$(X\in \Sch) \mapsto \ul\bC(X)\in \DGCat,$$
such that:

\medskip

\begin{itemize}

\item The category $\ul\bC(X)$ is a module over $\Shv(X)$ (where the latter is equipped with a monoidal structure
given by the $\sotimes$ tensor product);

\medskip

\item For $f:X_1\to X_2$ we are given functors $f_*:\ul\bC(X_1)\to \ul\bC(X_2)$ and $f^!:\ul\bC(X_2)\to \ul\bC(X_1)$,
both $\Shv(X_2)$-linear, where $\Shv(X_2)$ acts on $\ul\bC(X_1)$ via $f^!:\Shv(X_2)\to \Shv(X_1)$;

\medskip

\item For a cartesian square
$$
\CD
X_1 @>{f_X}>> X_2 \\
@V{g_1}VV @VV{g_2}V \\
Y_1 @>{f_Y}>> Y_2
\endCD
$$
we are given an isomorphism
$$f_Y^!\circ (g_2)_*\simeq (g_1)_*\circ f_X^!$$
as functors $\ul\bC(X_2)\to \ul\bC(Y_1)$;

\medskip

\item Natural compatibilities between the above data.

\end{itemize}

\medskip

From this point of view, the category $\AGCat$ is a natural substitute for DG categories in
algebraic geometry: a ``true" algebro-geometric DG category should be a family of such
for every scheme, which is exactly what our $\ul\bC$'s are. Hence, the name
AGCat: algebro-geometric categories\footnote{The notation $\AGCat$ was suggested by Sam Raskin; previously we
used $\DGCat_{\on{alg-geom}}$.}.

\sssec{}

In terms of the above description, the functor $\bi$ is a (fully faithful) embedding
$$\DGCat\to \AGCat$$
that sends $\bC\in \DGCat$ to $\bi(\bC)$ defined by
$$\bi(\bC)(X):=\bC\otimes \Shv(X).$$

\medskip

And the functor $\ul\Shv(-)$ sends $Y\in \Sch$ to $\ul\Shv(Y)$ with
$$\ul\Shv(Y)(X):=\Shv(Y\times X).$$

\medskip

The natural transformation $\alpha$ is given by the functors \eqref{e:boxtimes intro}.

\ssec{What is this good for?} \label{ss:motivations}

\sssec{}

Here is the general philosophy:

\medskip

In characteristic zero, the theory of D-modules on algebraic varieties (and stacks) gives a well-behaved
``geometric” sheaf theory in that geometric constructions are closely mirrored by categorical constructions on
the category of sheaves.  This, in particular, leads to a rich theory of categorical group representations in that setting.

\medskip

A substantial motivation for this work is to develop a parallel theory in the constructible setting that would
apply in the setting of $\ell$-adic sheaves in characteristic $p>0$.  The first issue one encounters in building
such a geometric theory of constuctible sheaf categories is the failure of the categorical Kunneth formula, i.e.,
the fact that \eqref{e:boxtimes intro} is not an equialence.

\medskip

The key idea in this work is that this is \emph{the only} issue. Once we
provide a ``formal fix",  everything else falls into place. We now discuss concrete applications for having such a fix.

\sssec{}  \label{sss:trace calc}

Here are two primary reasons to want a strictly monoidal functor $\ul\Shv(-)$. One has to do with
\emph{trace calculations}.

\medskip

Namely, let $\CF$ be an object of $\Shv(X\times X)$. We can regard it as a kernel defining a functor
$$[\CF]:\Shv(X)\to \Shv(X),$$
and one can consider its categorical trace
\begin{equation} \label{e:Trace Intro}
\Tr([\CF],\Shv(X))\in \Vect.
\end{equation}

The object \eqref{e:Trace Intro} is well-defined, but we do not know how to calculate (and may
not necessarily want to). However, if we calculate the trace of the corresponding endofunctor
$\ul{[\CF]}$ of $\ul\Shv(X)$ we obtain
$$\on{C}^\cdot(X,\Delta_X^!(\CF)).$$

\medskip

An important special case is when $X$ is a scheme over $\ol\BF_q$, but defined over $\BF_q$
and $\CF$ is given by the direct image (of the dualizing sheaf with respect to the graph) of the geometric Frobenius.
Then the above formalism provides a categorical framework for Grothendieck's function-sheaf correspondence,
see Sects. \ref{ss:less Frob} and \ref{ss:more Frob}.

\sssec{} \label{sss:useful1}

When instead of schemes, we allow to consider algebraic stacks, in this way we obtain cohomologies
of shtukas\footnote{But one does this for the Verdier dual version of $\AGCat$, see \secref{sss:left version}.}. Now, the trace formalism
allows one to organize cohomologies of shtukas into a quasi-coherent sheaf on the stack of Langlands
parameters; a toy model for this was worked out in \cite{GKRV}.

\medskip

We will address this in a sequel to this
paper, \cite{GRV3}.

\sssec{}

Another reason to want to have a functor $\ul\Shv(-)$ is in order to be able to define the notion of
\emph{categorical} representation of a group.

\medskip

Let us place ourselves temporarily in the context of D-modules, in which case, the categorical K\"unneth formula does hold,
i.e., the functor
$$\Dmod(X_1)\otimes \Dmod(X_2)\to \Dmod(X_1\times X_2)$$
is an equivalence.

\medskip

In this case, by an action of an algebraic group $G$ on a DG category $\bC$ we will mean a functor
$$\bC\overset{\on{co-act}}\to \bC\otimes \Dmod(G),$$
such that the diagram
$$
\CD
\bC @>{\on{co-act}}>> \bC\otimes \Dmod(G) \\
@V{\on{co-act}}VV  @VVV \\
\bC\otimes \Dmod(G) @>{\on{Id}\otimes \on{mult}^!}>> \bC\otimes \Dmod(G\times G)
\endCD
$$
commutes, where the left vertical arrow is
$$\bC\otimes \Dmod(G)  \overset{\on{co-act}\otimes \on{Id}}\longrightarrow
\bC\otimes \Dmod(G)\otimes \Dmod(G) \overset{\on{Id}\otimes (\boxtimes)}\longrightarrow \bC\otimes \Dmod(G\times G).$$

\medskip

Let us now return back to the context of $\ell$-adic sheaves. In this case, one can attempt to give a similar definition:
a DG category $\bC$ (over the field of coefficients, i.e., $\ol\BQ_\ell$) and a functor
$$\bC\overset{\on{co-act}}\to \bC\otimes \Shv(G),$$
such that the diagram
\begin{equation} \label{e:group action failed}
\CD
\bC @>{\on{co-act}}>> \bC\otimes \Shv(G) \\
@V{\on{co-act}}VV  @VV{(\on{Id}\otimes (\boxtimes))\circ (\on{co-act}\otimes \on{Id})}V \\
\bC\otimes \Shv(G) @>{\on{Id}\otimes \on{mult}^!}>> \bC\otimes \Shv(G\times G)
\endCD
\end{equation}
commutes.

\medskip

This is a valid definition, but in this way one produces only a small part of examples that
we want to consider. For example, we would not have the regular representation: the
pullback functor
$$\Shv(G)\overset{\on{mult}^!}\to \Shv(G\times G)$$
does not land in
\begin{equation} \label{e:sq of group}
\Shv(G)\otimes \Shv(G)\subset \Shv(G\times G).
\end{equation}

\begin{rem}
One could attempt to define the regular representation by setting
$$\on{co-act}:=\boxtimes^R\circ \on{mult}^!,$$
where $\boxtimes^R$ is the right adjoint of \eqref{e:sq of group}, i.e.,
$\on{co-act}$ is the \emph{dual} of the convolution functor
$$\Shv(G)\otimes \Shv(G)\overset{\boxtimes}\to  \Shv(G\times G)\overset{\on{mult}_*}\to \Shv(G).$$

However, this $\on{co-act}$ fails to make \eqref{e:group action failed} commute.

\end{rem}

\sssec{} \label{sss:useful2}

A way to fix this is to consider categorical representations of $G$ in $\AGCat$ instead of $\DGCat$.
It turns out that in this way one obtains a very sensible theory. In particular, it allows one to do the following
(when working over $\ol\BF_q$ but with our geometric objects defined over $\BF_q$):

\medskip

\noindent{(i)} One can compute the 2-categorical trace of Frobenius on $G\mmod(\AGCat)$, with the
result being $\ul\Shv(G/\on{Ad}^\Frob(G))$. For example, when $G$ is connected, the latter object is
(up to applying the fully faithful functor $\bi$) the category of representations of the finite group
$G(\BF_q).$

\medskip

\noindent{(ii)} If $G$ is reductive and $M$ its Levi subgroup, taking $\Tr(\Frob,-)$ on the
the 2-categorical induction functors $M\mmod(\AGCat)\to G\mmod(\AGCat)$, we recover
Deligne-Lusztig representations.

\medskip

\noindent{(iii)} One can compute characters of Deligne-Lusztig representations and relate them to
pointwise Frobenius traces of Springer sheaves.

\medskip

\noindent{(iv)} One can do all of the above for the loop group $\fL(G)$ instead of $G$, and recover
(an algebraic version of) the Fargues-Scholze theory.

\medskip

Points (i)-(iii) above will be developed in \cite{GRV2}, and point (iv) in future work.

\ssec{How is \texorpdfstring{$\AGCat$}{AGCat} actually constructed?}

\sssec{}

The construction of $\AGCat$ proceeds within the following general framework:

\medskip

Given a pair of (closed) symmetric monoidal categories $\bO$ and $\CV$ and a \emph{right-lax}
symmetric monoidal functor $F:\bO\to \CV$, we define a new symmetric monoidal
category $\ul\CV:=\on{Enh}(\bO,\CV)$, equipped with:

\begin{itemize}

\item A symmetric monoidal functor $\bi:\CV\to \ul\CV$;

\smallskip

\item A symmetric monoidal functor $\ul{F}:\bO\to \ul\CV$;

\smallskip

\item A symmetric monoidal transformation: $\alpha:\bi\circ F\to \ul{F}$,

\end{itemize}
so that $\ul\CV$ is universal with respect to admitting the above structure.

\medskip

The natural transformation $\alpha$ gives rise to a natural transformation
$$F \to \bi^R\circ \ul{F},$$
and an additional feature of the construction is that the latter map is an isomorphism.

\sssec{}

We construct $\AGCat$ by taking $\bO:=\on{Corr}(\Sch)$, $\CV:=\DGCat$ and $\CF:=\Shv(-)$.

\sssec{}

We construct $\on{Enh}(\bO,\CV)$ using the formalism of enriched categories in three steps:

\begin{enumerate}

\item We consider $\bO$ as enriched over itself (thanks to the condition that $\bO$ is closed);

\medskip

\item We define the $\CV$-enriched category $\on{enh}(\bO,\CV)$ by changing (=inducing)
the enrichment from $\bO$ to $\CV$ using $F$;

\smallskip

\item We define $\on{Enh}(\bO,\CV)$ as the category of $\CV$-enriched presheaves on $\on{enh}(\bO,\CV)$
(this is a general construction that starts with a $\CV$-enriched category and produces from it a $\CV$-module
category).

\end{enumerate}

\begin{rem}

In the particular case of $(\bO,\CV,F)=(\on{Corr}(\Sch),\DGCat,\Shv(-))$, the intermediate object $\on{enh}(\bO,\CV)$
is a familiar entity: it is the $\DGCat$-category, whose objects are schemes, and for $X_1,X_2$
$$\Hom^{\DGCat}(X_1,X_2)=\Shv(X_1\times X_2),$$
where $\Hom^{\DGCat}(-,-)$ is the Hom enriched in $\DGCat$.

\medskip

I.e., this is the category of schemes with ``kernels" as maps.

\end{rem}

\sssec{} \label{sss:Stefanich}

We should point out that the same procedure, namely,
$$(\bO,\CV,F) \mapsto \on{Enh}(\bO,\CV),$$
has been recently used by G.~Stefanich in another context in order to produce the 2-category $2\on{-IndCoh}(S)$,
which is an enlargement of the 2-category $2\on{-QCoh}(S):=\QCoh(S)\mmod$ for a base scheme/stack $S$.

\medskip

Namely, Stefanich starts with $\bO:=\on{Corr}(\Sch_{/S})$, $\CV:=\QCoh(S)\mmod$ and
$$F(X):=\IndCoh(X).$$

Here the right-lax symmetric monoidal structure is expressed by the functor
$$\IndCoh(X_1)\underset{\QCoh(S)}\otimes \IndCoh(X_2)\to \IndCoh(X_1\underset{S}\times X_2).$$

Note that even if $S$, $X_1$ and $X_2$ are smooth, the above functor is \emph{not} an equivalence
as long as $X_1\underset{S}\times X_2$ is not smooth.

\begin{rem}

Stefanich's construction is an improvement and generalization of a construction that produces
the same 2-category $2\on{-IndCoh}(S)$ that had been proposed earlier by D.~Arinkin.

\medskip

This construction plays a key role in formulating the 2-categorical local geometric Langlands,
where we take $S$ to be the stack of local Langlands parameters $\on{LS}^{\on{loc}}_\cG$. The need for an enlargement
$$2\on{-QCoh}(\on{LS}^{\on{loc}}_\cG)\rightsquigarrow 2\on{-IndCoh}(\on{LS}^{\on{loc}}_\cG)$$
is due to the phenomenon of \emph{non-temperedness}, and is a local counterpart of the enlargement
$$\QCoh(\LS^{\on{glob}}_\cG)\rightsquigarrow \IndCoh(\LS^{\on{glob}}_\cG)$$
in global geometric Langlands.

\end{rem}

\ssec{What else is done in this paper?}

\sssec{} \label{sss:else}

As was mentioned in \secref{sss:trace calc}, the formalism of $\AGCat$ provides a nice answer for
$$\Tr(\ul{[\CF]},\ul\Shv(X)), \quad \CF\in \Shv(X\times X)$$
when $X$ is a scheme.

\medskip

However, for practical applications we need a similar formula when $X$ is allowed to be a more general
algebro-geometric object, i.e., a \emph{prestack} (assumed locally of finite type).
In order to do this, we need to address the following questions:

\medskip

\begin{enumerate}

\item What is $\ul\Shv(\CY)\in \AGCat$ for a prestack $\CY$?

\medskip

\item Under what conditions is it true that
$$\ul\Shv(\CY_1)\otimes \ul\Shv(\CY_2)\simeq \ul\Shv(\CY_1\times \CY_2)?$$

\item What guarantees that the above object $\ul\Shv(\CY)\in \AGCat$ is dualizable
(so that the trace operation is defined)?

\medskip

\item What guarantees that $\ul\Shv(\CY)$ is self-dual?

\medskip

\item If $\ul\Shv(\CY)$ is self-dual and $\ul\Shv(\CY)\otimes \ul\Shv(\CY)\simeq \ul\Shv(\CY\times \CY)$,
is it true that
$$\Tr(\ul{[\CF]},\ul\Shv(\CY))\simeq \on{C}^\cdot(\CY,\Delta_\CY^!(\CF)),, \quad \CF\in \Shv(\CY\times \CY)?$$

\end{enumerate}

Answering these questions is what occupies the bulk of this paper.

\sssec{} \label{sss:co Intro}

First, we define $\ul\Shv(\CY)$ by
$$\ul\Shv(\CY)(X):=\Shv(\CY\times X),$$
where the category of sheaves on a prestack $\CY'$ is defined by
$$\underset{X\in (\Sch_{/\CY'})^{\on{op}}}{\on{lim}}\, \Shv(X),$$
where the limit is formed using the !-pullback functors.

\medskip

With this definition, there is a naturally defined map
\begin{equation} \label{e:boxtimes prestack Intro}
\ul\Shv(\CY_1)\otimes \ul\Shv(\CY_2)\to \ul\Shv(\CY_1\times \CY_2),
\end{equation}
but it is not an equivalence in general.

\sssec{}

That said, one can associate to a prestack $\CY'$ a different category, namely
$$\Shv(\CY')_{\on{co}}:=\underset{X\in \Sch_{/\CY'}}{\on{colim}}\, \Shv(X),$$
where the colimit is formed using the *-pushforward functors.

\medskip

We define the object $\ul\Shv(\CY)_{\on{co}}\in \AGCat$ by
$$\ul\Shv(\CY)_{\on{co}}(X)=\Shv(\CY\times X)_{\on{co}}.$$

If $\CY=X$ is a scheme, then more or less by definition, we have
$$\Shv(X)_{\on{co}}=\Shv(X) \text{ and } \ul\Shv(X)_{\on{co}}=\ul\Shv(X),$$
but for a general prestack $\CY$, the two are different. However, if $\CY$
has a schematic diagonal, there are natural maps
$$\Omega_\CY:\Shv(\CY)_{\on{co}}\to \Shv(\CY) \text{ and } \ul\Omega_\CY:\ul\Shv(\CY)_{\on{co}}\to \ul\Shv(\CY).$$

\medskip

We call a prestack \emph{tame} (resp., universally tame) if $\Omega_\CY$ (resp., $\ul\Omega_\CY$) is an equivalence.
We shall say that $\CY$ is \emph{AG tame} if both $\CY$ and $\CY\times \CY$ are universally tame.

\sssec{}

Tame (resp.,  universally tame, AG tame) prestacks are quite abundant: any ind-scheme
(or, more generally, any pseudo-scheme) is AG tame.

\medskip

And any quasi-compact algebraic stack with an affine diagonal, which is locally a quotient
of a scheme by a group, is AG tame. (In fact, conjecturally, the latter condition is not even
necessary).

\sssec{} \label{sss:AG good Intro}

It turns out that AG tame prestacks are adapted for positive answers to the questions
posed in \secref{sss:else}:

\medskip

\begin{enumerate}

\item If $\CY$ is AG tame, then for any $\CY'$, the map
$$\ul\Shv(\CY)\otimes \ul\Shv(\CY')\to \ul\Shv(\CY\times \CY')$$
is an equivalence.

\medskip

\item If $\CY$ is AG tame, then the object $\ul\Shv(\CY)\in \AGCat$ is dualizable and self-dual;

\medskip

\item If $\CY$ is AG tame, then for $\CF\in \Shv(\CY\times \CY)$,
\begin{equation} \label{e:Tr calc}
\Tr(\ul{[\CF]},\ul\Shv(\CY))\simeq \on{C}^\cdot_\blacktriangle(\CY,\Delta_\CY^!(\CF)),
\end{equation}
where $\on{C}^\cdot_\blacktriangle(\CY,-)$ is the functor of \emph{renormalized} cochains,
see \secref{sss:ren cochain}.

\end{enumerate}

\ssec{Structure of the paper}

This paper is structured as follows:

\sssec{}

In \secref{s:enh} we explain the general construction
$$(\bO,\CV,F) \mapsto \EnhOV.$$

\sssec{}

In \secref{s:AGCat} we use the construction in \secref{s:enh} to produce $\AGCat$.

\sssec{}

In \secref{s:prestacks} we explain how to evaluate $\ul\Shv(-)$ on prestacks, and develop
the theory outlined in Secs. \ref{sss:co Intro}-\ref{sss:AG good Intro}.

\sssec{}

In \secref{s:duality} we explain how the notion of monoidal dual in $\AGCat$ interacts
with Verdier duality on schemes and algebraic stacks.

\sssec{}

Finally, in \secref{s:Trace} we perform the trace calculation \eqref{e:Tr calc}. We then specialize
to the case of case of the Frobenius map (for algebraic stacks over $\ol\BF_q$) and show that
$$\Tr(\Frob_*,\ul\Shv(\CY))\simeq \sFunct(\CY(\BF_q),\ol\BQ_\ell).$$

We also explain how the latter identification relates to Grothendieck's sheaves-functionss correspondence.

\sssec{}

In \secref{s:enriched cat} we review the theory of enriched categories (over a given monoidal category
$\CA$) and relate them to $\CA$-module categories.

\medskip

The key notions here are the \emph{enriched presheaves} construction
$$(\CC\in \Cat^\CA) \mapsto (\bP^\CA(\CC)\in \CA\mmod).$$
and the enriched Yoneda functor
$$\on{Yon}^\CA_\CC:\CC\to \bP^\CA(\CC).$$

\sssec{}

In \secref{s:weighted limits} we review \emph{weighted limits and colimits},
which are generalizations of usual limits and colimits for $\CA$-module
categories.

\medskip

The weights in question are given by objects of $\bP^\CA(\CC)$ and $\bP^{\CA^{\on{rev}}}(\CC^{\on{op}})$,
respectively.

\sssec{}

Finally, in \secref{s:adj} we review the notion of adjunction in an $(\infty,2)$-category.

\ssec{Conventions and notations}

\sssec{}

We will be working in the context of higher category theory. When we say ``category"
we always mean an $\infty$-category.

\medskip

We let $\Spc$ denote the category of \emph{spaces}, a.k.a. \emph{groupoids}.

\medskip

Given a category $\bC$, and $\bc_1,\bc_2\in \bC$, we let
$$\Maps_\bC(\bc_1,\bc_2)\in \Spc$$
denote the corresponding space of objects.

\medskip

For a category $\bC$, we let $\bC^{\on{grpd}}$ denote the space of its objects
(i.e., the groupoid obtained from $\bC$ by discarding non-invertible morphisms).

\sssec{}

In this paper we will also use $(\infty,2)$-categories. We will adopt the
definition developed in \cite[Chapter 10]{GR1}.

\sssec{}

Throughout this paper we will be working with the category of (classical, separated) schemes of finite type
over a ground field $k$, assumed algebraically closed. We denote this category by $\Sch$.
We will not need derived algebraic geometry over $k$ for this paper.

\medskip

By a prestack we shall mean a (classical) prestack locally of finite type over $k$
(see \cite[Chapter 2, Sect. 1.3.6]{GR1} for what this means). We denote the corresponding category
by $\on{PreStk}$.

\sssec{}

We will do \emph{higher algebra} in vector spaces over the field of coefficients, which
in this paper we take to be $\ol\BQ_\ell$.

\medskip

We let $\DGCat$ denote the $(\infty,2)$-category of $\ol\BQ_\ell$-linear DG categories,
defined as in \cite[Chapter 1, Sect. 10]{GR1}.

\medskip

We view $\DGCat$ as equipped with the symmetric monoidal structure, given by the
Lurie tensor product. The unit object for this structure is $\Vect$,
the category of (chain complexes) $\ol\BQ_\ell$-vector spaces.

\medskip

For a DG category $\bC$ and $\bc_1,\bc_2\in \bC$, we let
$$\CHom_\bC(\bc_1,\bc_2)\in \Vect$$
denote the corresponding object. We have
$$\Maps_{\bC}(\bc_1,\bc_2)=\tau^{\leq 0}(\CHom_\bC(\bc_1,\bc_2)),$$
where we view the right-hand side as a connective spectrum.

\ssec{Acknowledgements}

This paper (obviously) owes its existence to V.~Drinfeld.

\medskip

We would also like to thank D.~Kazhdan, S.~Raskin, G.~Stefanich for valuable discussions.

\medskip

The research of N.R. is supported by an NSERC Discovery Grant (RGPIN-2025-0696).

\medskip

The research of Y.V. was partially supported by the ISF grants
2091/21 and 2889/25. 

\section{The enhancement construction} \label{s:enh}

In this section we introduce a general categorical procedure, in the framework of which
we will define $\AGCat$. As was mentioned in the introduction (see \secref{sss:Stefanich}),
this procedure appears to have a broader use.

\ssec{Setup and construction}\label{ss:ker setup}

\sssec{}

Let $\bO$ be a closed\footnote{Recall that a (symmetric)
monoidal category is \emph{closed} if it has internal homs; i.e. for every $x\in \bO$, the functor $(-)
\otimes x: \bO \to \bO$ admits a right adjoint.} symmetric monoidal category, $\CV$ a presentable symmetric monoidal category
\footnote{I.e., an
commutative algebra object in $\Cat^{\on{pres}}$, see \secref{sss:pres 2}.} and
$$ F: \bO \to \CV $$
a right-lax symmetric monoidal functor.

\medskip

We will construct a symmetric monoidal $\CV$-module category $\EnhOV$ along with a symmetric monoidal functor
$$ \wtF: \bO \to \EnhOV $$
together with an identification of right-lax symmetric monoidal functors
\begin{equation} \label{e:first req}
F \simeq \uHom_{\EnhOV,\CV}(\uno_{\EnhOV}, -)\circ \wtF,
\end{equation}
where the notation $\uHom_{\EnhOV,\CV}$ stands for the \emph{relative internal}
Hom in $\EnhOV$ with respect to $\CV$.

\sssec{}

As was explained in the introduction, the construction of the category $\EnhOV$ uses the theory of enriched categories.
In Sect. \ref{s:enriched cat}, we provide a digest of the relevant aspects of the theory of enriched categories (in the $\infty$-categorical setting).

\medskip

For a monoidal category $\CA$, we will denote by $\on{Cat}^{\CA}$ the category of $\CA$-enriched categories.  When $\CA$ is \emph{symmetric} monoidal,
$\on{Cat}^{\CA}$ is naturally a symmetric monoidal category, see \secref{sss:enr sym mon}. 

\medskip

We will also be considering module categories over a given presentable monoidal category $\CA$ (see \secref{sss:pres monoidal}); we denote
the 2-category of such by $\CA\mmod$. For $\CM\in \CA\mmod$ and $m_1,m_2\in \CM$, we denote by
$$\uHom_{\CM,\CA}(m_2,m_2)\in \CA$$
the \emph{internal-relative-to} $\CA$ hom, i.e.,
$$\Maps_\CA(a,\uHom_{\CM,\CA}(m_2,m_2)):=\Maps_\CM(a\otimes m_1,m_2).$$

The assignment
$$(m_1,m_2)\mapsto \uHom_{\CM,\CA}(m_2,m_2)$$
makes $\CM$ into an object of $\Cat^\CA$, whose underlying category is the original $\CM$, i.e.,
$$\Hom^\CA_\CM(m_1,m_2):=\uHom_{\CM,\CA}(m_2,m_2).$$

\sssec{}

By \cite[Sect. 7]{GepHaug} (see also \cite{Heine}), the closed symmetric monoidal structure on $\bO$ gives rise to a symmetric monoidal $\bO$-enriched category
$$\on{Self-Enr}(\bO) \in \on{ComAlg}(\Cat^{\bO}) .$$
The space of objects of $\on{Self-Enr}(\bO)$ is the same as that of $\bO$ and for $\bo_1,\bo_2 \in \bO$
$$\Hom^\bO_{\on{Self-Enr}(\bO)}(\bo_1,\bo_2) \simeq \uHom_{\bO}(\bo_1,\bo_2) \in \bO $$
is given by the internal hom in $\bO$.

\medskip

We have a symmetric monoidal equivalence
$$\oblv_{\on{Enr}}(\on{Self-Enr}(\bO)) \simeq \bO ,$$
where $\oblv_{\on{Enr}}$ is the functor that forgets enrichment (see Sect. \ref{sss:forget enr}).

\sssec{}

We define
$$ \enhOV \in \Cat^{\CV} $$
to be the $\CV$-enriched category obtained from $\on{Self-Enr}(\bO)$ and changing enrichment by $F$, i.e.,
$$\enhOV:=\ind_\bO^\CV(\on{Self-Enr}(\bO)).$$

Since $F$ is right-lax \emph{symmetric} monoidal, $\enhOV$ is naturally a symmetric monoidal $\CV$-enriched category
(see \secref{sss:enh sym}).

\medskip

The category $\enhOV$ has the objects \emph{flagged by}\footnote{See \secref{ss:flagged} for what this means.}
$\bO^{\on{grpd}}$ and morphism objects are given by
$$\Hom^\CV_{\enhOV}(\bo_1,\bo_2) = F(\underline{\on{Hom}}_{\bO}(\bo_1,\bo_2)) ,$$
with composition given by the right-lax monoidal structure of $F$.  Moreover, tensor product in $\enhOV$ agrees with that in $\bO$
at the level of objects (see \secref{sss:enh sym}).

\medskip

We have a tautologically defined symmetric monoidal functor
\begin{equation} \label{e:O to enh}
\bO \simeq \oblv_{\bO\on{-Enr}}(\on{Self-Enr}(\bO)) \to \oblv_{\CV\on{-Enr}}(\enhOV).
\end{equation}

\sssec{}

For any $\CV$-enriched category $\CC$, there is a corresponding $\CV$-module category
\begin{equation} \label{e:presheaf recap}
\mathbf{P}^{\CV}(\CC) := \on{Funct}^{\CV}(\CC^{\on{op}}, \CV)
\end{equation}
of $\CV$-functors from the opposite category of $\CC$ to $\CV$.  As in the unenriched category theory, this
category has the universal property that it is the free cocompletion of $\CC$ as a $\CV$-enriched category
(see Proposition \ref{p:env enh} for the precise assertion).

\sssec{}\label{s:ker definitions}

We define $\EnhOV$ as the enriched presheaf category:
$$\EnhOV := \mathbf{P}^{\CV}(\enhOV).$$

\medskip

The symmetric monoidal structure on $\enhOV$ induces one on $\EnhOV$, see Sect. \ref{sss:sym mon preshv}.

\medskip

Composing \eqref{e:O to enh} with the enriched Yoneda embedding (see Sect. \ref{ss:enriched Yoneda}),
we obtain the desired symmetric monoidal functor
$$\wtF: \bO \to \EnhOV .$$

Sometimes, when no confusion is likely to occur, we will write
$$\ul\CV:= \EnhOV.$$

\sssec{}

By the enriched Yoneda lemma (\corref{c:enriched Yoneda}), for $\bo_1,\bo_2\in \bO$,
$$\uHom_{\ul\CV,\CV}(\wtF(\bo_1),\wtF(\bo_2))\simeq \Hom^\CV_{\enhOV}(\bo_1,\bo_2):=
F(\uHom_\bO(\bo_1,\bo_2)).$$

In particular,
$$\uHom_{\ul\CV,\CV}(\uno_{\ul\CV},\wtF(\bo))\simeq F(\bo),$$
as desired, see \eqref{e:first req}.

\sssec{} \label{sss:objects of KER}

Using \eqref{e:presheaf recap}, we will regard an object $\ul{v} \in \ul\CV$
as an assignment
$$ \bo\in \bO \mapsto \ul{v} (\bo) \in \CV,$$
together with compatible system of maps
$$F(\underline{\Hom}_{\bO}(\bo_1,\bo_2)) \otimes \ul{v} (\bo_2) \to \ul{v} (\bo_1) .$$

For example, for $\ul{v}=\wtF(\bo')$, we have
\begin{equation} \label{e:Yo O}
\ul{v} (\bo)=\Hom^\CV_{\enhOV}(\bo,\bo')=F(\uHom_\bO(\bo,\bo')).
\end{equation}

\medskip

By the enriched Yoneda lemma (\lemref{l:enriched Yoneda}), for $\ul{v}\in \ul\CV$ and $\bo\in \bO$,
$$\ul{v} (\bo)\simeq \uHom_{\ul\CV,\CV}(\wtF(\bo),\ul{v}).$$

\ssec{Basic properties of the enhancement construction}

\sssec{}
Let $\bO^{\on{grpd}}$ denote the space of objects of $\bO$.  By Sect. \ref{sss:enr monadic}, we have a monadic adjunction in $\CV\mathbf{\mbox{-}mod}$.

\begin{equation}\label{e:ker-oblv}
\xymatrix{
\on{Funct}(\bO^{\on{grpd}},\CV) \ar@<0.8ex>[r] & \ar@<0.8ex>[l] \EnhOV=\ul\CV.
}
\end{equation}

In terms of Sect. \ref{sss:objects of KER}, the right adjoint in \Cref{e:ker-oblv} sends $\ul{v}\in  \ul\CV$
to its restriction along $\bO^{\on{grpd}}\hookrightarrow \bO$.

\sssec{} \label{sss:prop Enh}

In particular, we obtain that the right adjoint in \Cref{e:ker-oblv} is conservative, and it follows from \Cref{s:weighted limits} that
it preserves limits and colimits and tensors and cotensors\footnote{See Sect. \ref{sss:cotensor as limit} for what this means.}
with objects of $\CV$.

\medskip

By Proposition \ref{p:weighted limits decomp}, this implies that the right adjoint in \Cref{e:ker-oblv} preserves all
weighted limits and colimits.

\medskip

In other words, when we view objects of $\EnhOV$ as in Sect. \ref{sss:objects of KER}, weighted limits and colimits
are computed value-wise.

\sssec{}

Consider the functor
$$\CV \times \bO^{\on{grpd}} \overset{\on{Id}\otimes \on{Yon}_\bO}\longrightarrow
\CV\otimes \on{Funct}(\bO^{\on{grpd}}, \Spc) \to 
\on{Funct}(\bO^{\on{grpd}},\CV).$$
The images of the objects $v\times \bo$ for $\bo\in \bO^{\on{grpd}}$, $v\in \CV$ generate the category $\on{Funct}(\bO^{\on{grpd}},\CV)$.
It follows that the corresponding objects
$$v\otimes \wtF(\bo) \in \ul\CV $$
generate $\ul\CV$ under colimits (see Corollary \ref{c:colimit pres of enr pshv} and Sect. \ref{sss:Yoneda generates weighted} for more precise assertions to this effect).

\sssec{}

We observe:

\begin{prop} \label{p:Enh is closed}
Suppose that every object of $\bO$ is dualizable. Then
the symmetric monoidal category $\EnhOV$ is closed; i.e. admits internal homs.
\end{prop}

\begin{proof}
Since every object in $\EnhOV$ is a weighted colimit of $\ul{F}(\bo)$ for $\bo \in \bO$ and $\EnhOV$ admits all weighted \emph{limits},
by \propref{p:weighted colims}, it suffices to show that for an object $\ul{v} \in \EnhOV$, there exists an internal hom
$$\uHom_{\EnhOV}(\ul{F}(\bo), \ul{v}) \in \EnhOV.$$

Since the functor $\ul{F}:\bO\to \EnhOV$ is symmetric monoidal and $\bo$ is dualizable, the object $\ul{F}(\bo)$ is dualizable. Hence,
the above internal Hom exists and equals
$$\ul{F}(\bo^\vee)\otimes \ul{v}.$$

\end{proof}

\sssec{}
By the construction of $\ul\CV:=\EnhOV$, we have a right-lax symmetric monoidal adjunction of $\CV$-module categories
\begin{equation}\label{e:adjunction for KER}
\xymatrix{
\bi: \CV \ar@<0.8ex>[r] & \ar@<0.8ex>[l] \ul\CV : \be
}
\end{equation}
given by $\bi(v)=v \otimes \uno_{\ul\CV}$ and
$$\be(\ul{v} ) = \uHom_{\ul\CV,\CV}(\uno_{\ul\CV}, \ul{v} ) \simeq \ul{v} (\uno_{\bO}),$$
see Sect. \ref{sss:objects of KER} for the notation.

\medskip

Note that by definition, for $\bo\in \bO$
\begin{equation}
\bi(v)(\bo) = v\otimes F(\uHom_{\bO}(\bo, \uno_{\bO})) .
\end{equation}

\sssec{}
In practice, the functor $\bi$ is often fully faithful.

\begin{prop}\label{p:strictly unital inc ff}
Suppose that the functor $F$ is strictly unital, i.e. the natural map
$$ \uno_{\CV} \to F(\uno_{\bO}) $$
is an isomorphism.  Then the functor $\bi$ is fully faithful.
\end{prop}
\begin{proof}
By assumption, the natural map
$$ \uHom_{\CV}(\uno_{\CV}, \uno_{\CV}) \to \uHom_{\ul\CV,\CV}(\bi(\uno_{\CV}), \bi(\uno_{\CV})) $$
is an isomorphism.  Hence, by
\Cref{c:enriched pshv fully faithful}, the functor
$$\bi: \CV \simeq \mathbf{P}^{\CV}(*) \to \EnhOV=\ul\CV$$ is fully faithful.
\end{proof}


\sssec{}

We now show that the general set-up of $F:\bO\to \CV$ can be reduced to one of Proposition \ref{p:strictly unital inc ff}:

\medskip

Note that the functor $F$ induces a right-lax symmetric monoidal functor
$$F^{\on{enh}}: \bO \to F(\uno_\bO)\on{-mod}(\CV),$$
so that the original $F$ is the composition
$$ \bO \overset{F^{\on{enh}}}\longrightarrow  F(\uno_\bO)\on{-mod}(\CV) \overset{\oblv_{F(\uno_\bO)}}\longrightarrow \CV.$$
Note that the functor $F^{\on{enh}}$ is strictly unital.

\begin{prop}\label{p:ker insensitive to target}
There is a natural equivalence of $\CV$-module categories
$$\on{Enh}(\bO,F(\uno_\bO)\on{-mod}(\CV)) \simeq \EnhOV .$$
\end{prop}

\begin{proof}

We have a right-lax symmetric monoidal adjunction
\begin{equation}\label{e:module adjunction}
\xymatrix{
\CV \ar@<0.8ex>[r] & \ar@<0.8ex>[l] F(\uno_\bO)\on{-mod}(\CV)
}
\end{equation}

By functoriality of the change of enrichment, the counit of the adjunction \Cref{e:module adjunction}
induces a functor
\begin{equation}\label{e:module adjunction 1}
\ind_{\CV}^{F(\uno_\bO)\on{-mod}(\CV)} (\enhOV) \to \on{enh}(\bO,F(\uno_\bO)\on{-mod}(\CV)).
\end{equation}

Since the left adjoint in \secref{e:module adjunction} is \emph{strictly} symmetric monoidal, by Sect. \ref{sss:change of enrichment tensor},
passing to enriched presheaves, the functor \eqref{e:module adjunction 1} gives rise to a functor of $(F(\uno)\on{-mod}(\CV))$-module categories
$$F(\uno_\bO)\on{-mod}(\CV) \underset{\CV}{\otimes} \EnhOV  \to \on{Enh}(\bO,F(\uno_\bO)\on{-mod}(\CV)).$$
By adjunction, this gives rise to a functor of $\CV$-module categories
\begin{equation}\label{e:kernel enhancement}
\EnhOV \to \on{Enh}(\bO,F(\uno_\bO)\on{-mod}(\CV)).
\end{equation}
By construction, the functor \Cref{e:kernel enhancement} maps representable objects to representable objects and is
fully faithful (in the enriched sense) on representable objects.  Thus, by \Cref{c:equiv to enriched pshv} it is an equivalence.
\end{proof}

\ssec{The universal property} \label{ss:univ property}

\sssec{}

The isomorphism
$$F\simeq \be\circ \wtF$$
gives rise to a natural transformation
$$\bi\circ F\overset{\alpha}\to \wtF$$
as right-lax symmetric monoidal functors $\bO\to \ul\CV$.

\sssec{}

We claim that the above data $(\wtF,\bi,\alpha)$ satisfy the following universal property:

\begin{prop} \label{p:univ property}
For a presentable symmetric monoidal category $\bC$, precomposition with the triple $(F,\bi,\alpha)$ defines an
equivalence between:

\smallskip

\noindent{\em(i)} The groupoid of colimit-preserving symmetric monoidal functors $\ul\CV\to \bC$;

\smallskip

\noindent{\em(ii)} The groupoid of triples $(G,\bj,\beta)$, where:

\begin{itemize}

\item $G$ is a symmetric monoidal functor $\bO\to \bC$;

\item $\bj$ is a colimit-preserving symmetric monoidal functor $\CV\to \bC$;

\item $\beta$ is a natural transformation $\bj\circ F \to G$
as right-lax symmetric monoidal functors $\bO\to \bC$.

\end{itemize}

\end{prop}

\begin{proof}

We construct a map in the opposite direction.

\medskip

Let us be given a colimit-preserving symmetric monoidal functor $\bj:\CV\to \bC$. In particular,
we can view $\bC$ a commutative algebra object in $\Cat^\CV$.

\medskip

Then, by \propref{p:env enh}, the datum of a colimit-preserving symmetric monoidal functor $\ul\CV\to \bC$ is equivalent to that
of a symmetric monoidal $\CV$-enriched functor
$$\enhOV\to \bC.$$

By the universal property in \secref{sss:induce rich}, the latter is the same as a symmetric monoidal functor
$$\on{Self-Enr}(\bO) \to \bC$$
compatible with $F$. In particular, this produces a symmetric monoidal functor $G:\bO\to \bC$.

\medskip

The structure on $G$ of compatibility with $F$ amounts to maps
$$F(\Hom^\bO_{\on{Enr}(\bO)}(\bo_1,\bo_2))\to \uHom_{\bC,\CV}(G(\bo_1),G(\bo_2)),$$
compatible with the symmetric monoidal structure.

\medskip

The latter structure is equivalent to a natural transformation
$$F\to j^R\circ G$$
as right-lax symmetric monoidal functors $\bO\to \CV$, or which is the same, as a natural transformation
$$j\circ F\to G$$
as right-lax symmetric monoidal functors $\bO\to \bC$.

\end{proof}

\sssec{}

Here is an application of Proposition \ref{p:univ property}:

\begin{prop}\label{p:kernel nothing new}
Suppose that $F$ is (strictly) symmetric monoidal and every object in $\bO$ is dualizable.  Then the adjunction \Cref{e:adjunction for KER} is an equivalence.
\end{prop}

\begin{proof}

It suffices to show that (under the assumptions on $\bO$ and $F$), in the situation of
Proposition \ref{p:univ property}, the natural transformation $\beta$ is automatically an isomorphism.

\medskip

This follows from the fact that a natural transformation between two (strict) symmetric monoidal functors
$$G_1\to G_2, \quad G_i:\bO\to \bC$$
is an isomorphism, provided that every object of $\bO$ is dualizable.

\end{proof}

\ssec{Change of source}

\sssec{}

Let us now be given two pairs $(\bO_1,F_1:\bO_1\to \CV)$ and $(\bO_2,F_2:\bO_2\to \CV)$.
Denote
$$\ul\CV_1:=\on{Enh}(\bO_1,\CV) \text{ and } \ul\CV_2:=\on{Enh}(\bO_2,\CV).$$

Let us be given a symmetric monoidal functor $\Phi:\bO_1\to \bO_2$, equipped with a symmetric monoidal
natural transformation
\begin{equation} \label{e:F 12}
F_1\to F_2\circ \Phi,
\end{equation}
as right-lax symmetric monoidal functors $\bO_1\to \CV$.

\sssec{} \label{sss:change of source}

We claim that in this case there a naturally defined symmetric monoidal functor
$$\ul\Phi:\ul\CV_1\to \ul\CV_2$$
equipped with the identifications
$$\ul{F_2}\circ \Phi\simeq \ul\Phi\circ \ul{F_1} \text{ and } \bi_2\simeq \ul\Phi\circ \bi_1$$
that make the following diagram commute:
$$
\CD
\ul\Phi\circ \bi_1\circ F_1 & @>{\alpha_1}>> & \ul\Phi\circ \ul{F_1} \\
@V{\sim}VV & & @V{\sim}VV \\
\bi_2 \circ F_1 @>>> \bi_2 \circ F_2\circ \Phi @>{\alpha_2}>> \ul{F_2}\circ \Phi.
\endCD
$$

 \sssec{}

 Indeed, this follows from \propref{p:univ property}, where we take
 $$\bC:=\ul\CV_2,\,\, G=\ul{F_2}\circ \Phi,\,\, \bj=\bi_2$$
 and $\beta$ be the composite bottom horizontal arrow in the above diagram.

 \sssec{}

 Note that from \propref{p:fully faithful from enriched pshv}, we obtain:

 \begin{cor} \label{c:change of source ff}
 Suppose that all objects in $\bO_1$ are dualizable and \eqref{e:F 12} is an isomorphism. Then the functor $\ul\Phi$
 is fully faithful.
 \end{cor}

\section{Construction of \texorpdfstring{$\AGCat$}{AGCat}} \label{s:AGCat}

In this section we construct the $(\infty,2)$-category $\AGCat$ and study its basic properties.

\ssec{The category of correspondences} \label{ss:corr}

\sssec{}

Given a category $\bC$ with finite limits (i.e., fiber products and a final object), we let $\on{Corr}(\bC)$ be the category of correspondences as
defined in \cite[Chapter 7, Sect. 1]{GR1}.

\medskip

Explicitly, objects of $\on{Corr}(\bC)$ are the same as objects of $\bC$. For a pair of objects $\bc_1,\bc_2$,
the groupoid of morphisms $\Maps_{\on{Corr}(\bC)}(\bc_1,\bc_2)$ consists of diagrams
\begin{equation} \label{e:corr Intro}
\CD
\bc_{1,2} @>>> \bc_2 \\
@VVV \\
\bc_1,
\endCD
\end{equation}
with the composition given by
$$(\bc_{1,2},\bc_{2,3})\mapsto \bc_{1,2}\underset{\bc_2}\times \bc_{2,3}.$$

\medskip

We will refer to the horizontal morphism in \eqref{e:corr Intro} as the \emph{forward} arrow, and
the vertical one as the \emph{backward} arrow.

\sssec{} \label{sss:dual corr}

The category $\on{Corr}(\bC)$ has a natural symmetric monoidal structure, induced by the Cartesian
product in $\bC$.

\medskip

Note that every object in $\on{Corr}(\bC)$ is dualizable and canonically self-dual. Namely,
the unit and counit of the self-duality are given by
$$
\CD
\bc @>{\Delta_\bC}>> \bc\times \bc \\
@VVV \\
\{*\}
\endCD
$$
and
$$
\CD
\bc @>>> \{*\} \\
@V{\Delta_\bC}VV \\
\bc\times \bc,
\endCD
$$
respectively, where $\{*\}$ is the final object in $\bC$.

\medskip

In particular, we obtain that $\on{Corr}(\bC)$ is closed and for $\bc_1,\bc_2\in \bC$, the internal hom
object $\uHom_{\on{Corr}(\bC)}(\bc_1,\bc_2)$ is $\bc_1\times \bc_2$.

\sssec{} \label{sss:corr var}

Let $\on{class}_1$ and $\on{class}_2$ be two classes of morphisms in $\bC$, each stable
under composition, and under base change with respect to each other, i.e., for a Cartesian square
$$
\CD
\bc'_1 @>{f'}>> \bc'_2 \\
@V{g_1}VV @VV{g_2}V \\
\bc_1 @>{f}>> \bc_2
\endCD,
$$
if $f\in \on{class}_1$ and $g_2\in \on{class}_2$, then $f'\in \on{class}_1$ and $g_1\in \on{class}_2$.

\medskip

Then we define a 1-full subcategory
$$\on{Corr}_{\on{class}_1,\on{class}_2}(\bC)\subset \on{Corr}(\bC),$$
where we restrict morphisms to those diagrams \eqref{e:corr Intro}, where
the forward morphism belongs to $\on{class}_1$ and the backward
morphism belongs to $\on{class}_2$.

\sssec{}

For $\on{class}_1=\on{all}$ and $\on{class}_2=\on{isom}$, we obtain
an equivalence
$$\on{Corr}_{\on{all},\on{isom}}(\bC)=\bC$$
and for $\on{class}_1=\on{isom}$ and $\on{class}_2=\on{class}$, we obtain
an equivalence
$$\on{Corr}_{\on{isom},\on{all}}(\bC)=\bC^{\on{op}}.$$

\ssec{The functor \texorpdfstring{$\Shv(-)$}{\Shv(-)}}

\sssec{}

Let $B$ be a base scheme, and consider the category $\on{Corr}(\Sch_{/B})$.

\medskip

The sheaf-theoretic input to the construction of the category $\AGCat$ is provided by a functor
\begin{equation} \label{e:Shv corr}
\Shv(-):\on{Corr}(\Sch_{/B})\to \DGCat,
\end{equation}
which is a \emph{constructible sheaf theory} in the sense of \cite[Sect. 1.1.4]{AGKRRV1}.

\sssec{}

This functor assigns to a scheme $X$ over $B$ the corresponding category $\Shv(X)$ and to a morphism
$f:X_1\to X_2$ the functors
$$f_*:\Shv(X_1)\to \Shv(X_2) \text{ and } f^!:\Shv(X_2)\to \Shv(X_1).$$

\medskip

Thus, the restriction $\Shv(-)|_{\on{Corr}_{\on{isom,all}}(\Sch_{/B})}$
is the functor
\begin{equation} \label{e:upper !}
(\Sch_{/B})^{\on{op}}\to \DGCat, \quad X\mapsto \Shv(X), \quad (X_1\overset{f}\to X_2)\mapsto f^!,
\end{equation}
and  the restriction $\Shv(-)|_{\on{Corr}_{\on{all,isom}}(\Sch_{/B})}$
is the functor
\begin{equation} \label{e:lower *}
\Sch_{/B}\to \DGCat, \quad X\mapsto \Shv(X), \quad (X_1\overset{f}\to X_2)\mapsto f_*.
\end{equation}

\sssec{}

The datum of extension of these functors to a functor \eqref{e:Shv corr} amounts to a homotopy-compatible
system of base-change isomorphisms: for
$$
\CD
X_1\underset{X_2}\times X'_2 @>{f'}>> X'_2 \\
@V{g_1}VV @VV{g_2}V \\
X_1 @>{f}>> X_2
\endCD,
$$
an isomorphism
$$g_2^!\circ f_* \simeq f'_*\circ g_1^!,$$
as functors $\Shv(X_1)\to \Shv(X'_2)$.

\medskip

The gist of the work \cite{GR1} can be summarized as follows: the datum of just \eqref{e:lower *}
(or \eqref{e:upper !}) uniquely recovers \eqref{e:Shv corr}.

\sssec{}

For this paper, we let \eqref{e:Shv corr} be the theory of $\ell$-adic sheaves, defined as in \cite[Sect. A.1.1(d')]{GKRV},
followed by tensoring up $(-)\underset{\BQ_\ell}\otimes \ol\BQ_\ell$.

\ssec{Construction of \texorpdfstring{$\AGCat$}{AGCat}}

\sssec{}

Morally, we would like to apply the set-up of \secref{ss:ker setup} to
$$\bO:=\on{Corr}(\on{Sch}_{/B}),\,\, \CV=\DGCat,\,\, F:=\Shv(-).$$

Unfortunately, we cannot quite do that, because in \secref{ss:ker setup} we assumed that
$\CV$ is presentable, while $\DGCat$ is not.

\medskip

To circumvent this, we will employ the recipe of \secref{sss:DG Cat}.

\sssec{}

In the setting of \secref{ss:ker setup}, we take
$$\bO:=\on{Corr}(\on{Sch}_{/B}),\,\, \CV=\DGCat_\kappa,\,\, F:=\Shv(-),$$
where $\kappa$ is as in \secref{sss:pres 1}.

\medskip

Set
\begin{equation} \label{e:AGCat kappa defn}
\AGCat_{\kappa,/B}:=\on{Enh}(\on{Corr}(\on{Sch}_{/B}),\DGCat_\kappa).
\end{equation}

\begin{equation} \label{e:AGCat defn}
\AGCat_{/B}:=\underset{\kappa<\lambda_0}{\on{colim}}\, \AGCat_{\kappa,/B},
\end{equation}
where the colimit is taken in $\Cat_{\on{large}}$,
cf. \secref{sss:DG Cat}.

\begin{rem}

From this point on, in the main body of the paper, we will ignore the difference between
$\DGCat$ and $\DGCat_\kappa$, and pretend that the construction in \secref{ss:ker setup}
was applied directly that $\DGCat$.

\medskip

We leave it to the reader to verify that the constructions and statements that follow stay
within our specified set-theoretic framework.

\end{rem}

\begin{rem}

When $B=\on{pt}=\on{Spec}(k)$, we will write simply $\AGCat$ instead of $\AGCat_{/B}$.

\end{rem}

\ssec{Basic properties of  \texorpdfstring{$\AGCat$}{AGCat}}

\sssec{}

By construction, the category $\AGCat_{/B}$ comes equipped with the following data:

\begin{itemize}

\item A symmetric monoidal functor
\begin{equation}\label{e:ushv as functor out of corr}
\underline{\on{Shv}}(-): \on{Corr}(\on{Sch}_{/B}) \to \AGCat_{/B};
\end{equation}

\item An adjunction
\begin{equation}\label{e:DGCta to AGCat}
\bi:\DGCat\rightleftarrows  \AGCat_{/B}:\be,
\end{equation}
with $\bi$ (strictly) symmetric monoidal;

\smallskip

\item An identification $\Shv(-)\simeq \be\circ \ul\Shv(-)$ as right-lax symmetric monoidal functors.

\end{itemize}

\medskip

The above data give rise to a natural transformation
$$\alpha:\bi\circ \Shv(-)\to \ul\Shv(-).$$

The triple $(\ul\Shv(-),\bi,\alpha)$ has a universal property formulated in Sect. \ref{ss:univ property}. I.e.,
we can think of $(\AGCat_{/B},\ul\Shv(-))$ as a universal way of turning $\Shv(-)$ into being strictly
monoidal.

\sssec{} \label{sss:AGCat expl}

As explained in \Cref{sss:objects of KER}, we will regard objects of $\AGCat$ as functors
$$ \ul{\bC}: \on{Sch}_{/B}^{\on{grpd}} \to \DGCat$$
endowed with additional structure: namely, for every $X_1,X_2\in \on{Sch}_{/B}$, and every sheaf
$$\CF\in \Shv(X_1\underset{B}\times X_2),$$
of a functor
\begin{equation} \label{e:functoriality KER}
\ul\bC(\CF): \ul{\bC}(X_1) \to \ul{\bC}(X_2)
\end{equation}
that is natural in all parameters.

\sssec{} \label{sss:Shv line} \label{sss:conv}

From this point of view, for an object $Z \in \on{Sch}_{/B}$, the object
$\ul\Shv(Z) \in \AGCat_{/B}$ is given by the assignment
$$ \on{Sch}_{/B} \ni X \mapsto \on{Shv}(Z\underset{B}\times X)$$
with the functors
$$\on{Shv}(Z\underset{B}\times X_1)\to \on{Shv}(Z\underset{B}\times X_2), \quad \CF\in \on{Shv}(X_1\underset{B}\times X_2)$$
given by convolution:
$$\CF \mapsto [\CF], \quad [\CF](\CF_1):=(p_2)_*(p_1^!(\CF_1)\sotimes p^!_{1,2}(\CF)),$$
where $p_1,p_2,p_{1,2}$ are the three projections from $Z\underset{B}\times X_1\underset{B}\times X_2$ to
$$Z\underset{B}\times X_1,\,\, Z\underset{B}\times X_2,\,\, X_1\underset{B}\times X_2,$$
respectively.

\sssec{} \label{sss:AGCat simple}

The data of \eqref{e:functoriality KER} can be rewritten as follows:

\medskip

\begin{itemize}

\item For every $X$, the category $\ul\bC(X)$ is acted on by $\Shv(X)$, which is regarded as a (symmetric) monoidal
category with respect to the $\sotimes$ operation;

\medskip

\item For every $f:X_1\to X_2$ (over $B$), we have a pullback functor $f^!:\ul\bC(X_2)\to \ul\bC(X_1)$, which is $\Shv(X_2)$-linear
with respect to $f^!:\Shv(X_2)\to \Shv(X_1)$.

\medskip

\item For every $f:X_1\to X_2$, we have a pushforward functor $f_*:\ul\bC(X_1)\to \ul\bC(X_2)$, which is $\Shv(X_2)$-linear
with respect to $f^!:\Shv(X_2)\to \Shv(X_1)$.

\medskip

\item For a diagram
$$
\CD
X_1\underset{Y}\times X_2 @>{f'_1}>> X_2 \\
@V{f'_2}VV @VV{f_2}V \\
X_1 @>{f_1}>> Y
\endCD
$$
an identification
$$(f'_1)_*\circ (f'_2)^!\simeq (f_2)^!\circ (f_1)_*, \quad \ul\bC(X_1)\to \ul\bC(X_2);$$

\item A homotopy-coherent system of compatibilities between the above data.

\end{itemize}

\sssec{Example}

For $\ul\bC=\ul\Shv(Z)$, the data from \secref{sss:AGCat simple} is particularly evident.

\sssec{}

In terms of the description in \secref{sss:AGCat expl}, the adjoint functors \eqref{e:DGCta to AGCat} look
as follows:

\medskip

For $\bC\in \DGCat$,
$$\bi(\bC)(X)=\bC\otimes \Shv(X)$$
and
$$\be(\ul\bC)=\ul\bC(B).$$

The natural transformation
$$\alpha:\bi\circ \Shv(-)\to \ul\Shv(-)$$
assigns to $Z\in \Sch_{/B}$ the collection of functors
$$\Shv(Z)\otimes \Shv(X)\to \Shv(Z\underset{B}\times X).$$

\sssec{} \label{sss:limits AGCat}

The category $\AGCat_{/B}$ has limits/colimits and tensors/cotensors by objects of $\DGCat$.
By \secref{sss:prop Enh}, all these operations are computed value-wise in terms of \eqref{sss:AGCat expl}.

\medskip

Additionally, thanks to \propref{p:Enh is closed}, $\AGCat_{/B}$ is closed; i.e. admits internal homs.

\sssec{}

By \Cref{p:ker insensitive to target}, the construction of $\AGCat_{/B}$ is not very sensitive to the target of the functor
$\Shv(-)$. In particular, we would get the same category if we replace
$$\DGCat \simeq \on{Vect}\on{-mod}(\Cat^{\on{pres}}) $$
by $\Cat^{\on{pres}}$ or $\Shv(B)\mod(\DGCat)$.

\medskip

In any case, the functor $\bi$ extends to a symmetric monoidal functor
$$\Shv(B)\mmod(\DGCat) \to \AGCat_{/B},$$
and the latter is fully faithful (see Propositions \ref{p:strictly unital inc ff} and \ref{p:ker insensitive to target}).

\ssec{Variants}

We can consider several variants of the category $\AGCat$.

\sssec{D-module variant}

Let $k$ be an algebraically closed field of characteristic zero.  To a scheme $X$ over $k$, we can attach the category of D-modules $\Dmod(X)$.
These assemble to a (strictly) symmetric monoidal functor
$$\Dmod(-): \on{Corr}(\on{Sch}_k) \to \DGCat_k ,$$
and we can perform the same construction as above.

\medskip

However, by \Cref{p:kernel nothing new}, there is an equivalence
$$\on{Enh}(\on{Corr}(\on{Sch}_k),\DGCat_k) \simeq \DGCat_k.$$
This suggests the following guiding meta-principle for much of this work: theorems about categories about
D-modules should hold in the $\ell$-adic setting when categories of sheaves are interpreted as objects of $\AGCat$.

\sssec{Affine schemes are enough}

Instead of considering the source of $F$ to be all schemes over $B$, we can instead consider only affine schemes.
As we shall presently see, this produces an equivalent category.

\medskip

Let $\on{Shv}|_{\on{aff}}(-)$ denote the restriction of the functor $\Shv(-)$ to the (1-full) subcategory
$$\on{Corr}(\on{Sch}_{/B}^{\on{aff}}) \subset \on{Corr}(\on{Sch}_{/B}), $$
and let
$$\AGCat_{/B}^{\on{aff}} := \on{Enh}(\Corr(\affSch_{/B}),\DGCat).$$

By \secref{sss:change of source}, we have a naturally defined functor
\begin{equation}\label{e:aff to all}
\AGCat^{\on{aff}}_{/B} \to \AGCat_{/B}.
\end{equation}

Moreover, by \corref{c:change of source ff} (and \secref{sss:dual corr}), the functor \eqref{e:aff to all}
is fully faithful. We claim:

\begin{prop}\label{p:affines are enough}
The functor \Cref{e:aff to all} is an equivalence.
\end{prop}
\begin{proof}
By \corref{c:equiv to enriched pshv}, it suffices to show that for every $Z \in \on{Sch}_{/B}$, the object
$$ \underline{\on{Shv}}(Z) \in \AGCat_{/B}$$
is in the essential image of the functor \Cref{e:aff to all}.  By construction, this is the case if $X$ is affine.

\medskip

Now, every scheme $Z$ is a sifted colimit (in the category of Zariski sheaves) of affine schemes
$$ Z \simeq \underset{i\in I}{\on{colim}}\  S_i.$$

The operation of *-pushforward give rise to a functor
\begin{equation}\label{e:affine cover}
\underset{i\in I}{\on{colim}}\ \underline{\on{Shv}}(S_i)\to \underline{\on{Shv}}(Z)
\end{equation}
in $\AGCat_{/B}$.  Since \Cref{e:aff to all} commutes with colimits, it is enough to show that \eqref{e:affine cover}
is an equivalence.

\medskip

Thus, by \secref{sss:limits AGCat}, we have to show that \eqref{e:affine cover} evaluates to an equivalence on every $X\in \Sch_{/B}$. I.e.,
we have to show that the functor
\begin{equation} \label{e:Zar codescent}
\underset{i\in I}{\on{colim}}\, \Shv(S_i\underset{B}\times X)\to \Shv(Z\underset{B}\times X)
\end{equation}
is an equivalence.

\medskip

The transition functors in \eqref{e:Zar codescent} preserve compactness. Hence, the left-hand side in \eqref{e:Zar codescent}
is compactly generated and its dual identifies with
$$\underset{i\in I^{\on{op}}}{\on{lim}}\, (\Shv(S_i\times X))^\vee.$$

Using Verdier duality (see \secref{ss:recall Verdier}), we identify the functor dual to \eqref{e:Zar codescent} with
\begin{equation} \label{e:Zar descent}
\Shv(Z\underset{B}\times X)\to \underset{i\in I^{\on{op}}}{\on{lim}}\, \Shv(S_i\underset{B}\times X),
\end{equation}
where in the right-hand side, the transition functors are given by !-restriction.

\medskip

Thus, it suffices to show that \eqref{e:Zar descent} is an equivalence. However,
this follows from Zariski descent for $\Shv(-)$.

\end{proof}

In light of \Cref{p:affines are enough}, in the sequel we will sometimes restrict to affine schemes as test objects for
$\AGCat_{/B}$.

\sssec{The \emph{left} version of $\AGCat$} \label{sss:left version}

In the setting of $\ell$-adic sheaves, there is another version of the category $\AGCat_{/B}$.  Namely, we have another functor
$$ \on{Shv}^{\on{left}}(-): \on{Corr}(\on{Sch}_{/B}) \to \DGCat $$
which takes a scheme $X$ to $\on{Shv}(X)$ and morphisms given by $*$-pullback and $!$-pushforward.

\medskip

This functor can be formally expressed in terms of $\on{Shv}$ as follows.  The value of $\on{Shv}$ on every morphism in $\on{Corr}(\on{Sch}_{/B})$ admits a left adjoint.
Let
$$ \on{Shv}^{L}(-): \on{Corr}(\on{Sch}_{/B})^{\on{op}} \to \DGCat $$
denote the functor obtained from $\on{Shv}$ by passing to left adjoints.  Now, the functor $\on{Shv}^{\on{left}}(-)$ is given by the composite
$$\on{Corr}(\on{Sch}_{/B}) \simeq \on{Corr}(\on{Sch}_{/B})^{\on{op}} \overset{\on{Shv}^L}{\longrightarrow} \DGCat $$
where the anti-equivalence of $\on{Corr}(\on{Sch}_{/B})$ is given by reading correspondences backwards (equivalently, it corresponding to passing to duals).

\medskip

We define
$$\AGCat^{\on{left}}_{/B}:=\on{Enh}(\on{Corr}(\on{Sch}_{/B}),\DGCat),$$
formed using the above functor $\on{Shv}^{\on{left}}(-)$.

\medskip

As before, we have a symmetric monoidal functor
$$\underline{\on{Shv}}^{\on{left}}(-): \on{Corr}(\on{Sch}_{/B}) \to \AGCat^{\on{left}}_{/B}.$$

\medskip

The contents of the previous two sections carry over to the present setting without any change\footnote{The two versions are, however,
substantially different from the point of view of interaction of monoidal duality on $\AGCat$ and Verdier duality on schemes discussed in \secref{s:duality}.}.

\begin{rem}

For most applications (such as those involving the sheaves-functions correspondence), $\AGCat_{/B}$ is the more useful category.

\medskip

Yet, there are some instances (notably ones that have to do with cohomology of shtukas), where it is more convenient to
use $\AGCat^{\on{left}}_{/B}$.

\end{rem}

\ssec{Change of base}

\sssec{}

Let $\phi:B'\to B$ be a map between base schemes. Base change defines a symmetric monoidal functor
$$\Phi:\Corr(\Sch_{/B})\to \Corr(\Sch_{/B'}).$$

In addition, !-pullback defines a symmetric monoidal natural transformation:

\begin{equation} \label{e:bc nat trans}
\Shv(-)\to \Shv(\Phi(-)).
\end{equation}

\sssec{}

Hence, applying \secref{sss:change of source}, we obtain a symmetric monoidal functor
$$\ul\Phi:\AGCat_{/B}\to \AGCat_{/B'}$$
equipped with identifications
$$\bi'\simeq \ul\Phi\circ \bi \text{ and } \ul\Shv(\Phi(-)) \simeq \ul\Phi\circ \ul\Shv(-).$$

\medskip

In addition, from \corref{c:change of source ff} we obtain:

\begin{cor}
Suppose that \eqref{e:bc nat trans} is an isomorphism. Then the functor
$\ul\Phi$ is fully faithful.
\end{cor}

\sssec{}

Let now $k'\supset k$ be an algebraically closed field extension; set $B':=\Spec(k')\underset{\Spec(k)}\times B$. Base change gives rise to a
symmetric functor
$$\Phi:\Corr(\Sch_{/B})\to \Corr(\Sch_{/B'})$$
and a symmetric monoidal natural transformation:
\begin{equation} \label{e:bc nat trans fields}
\Shv(-)\to \Shv(\Phi(-)).
\end{equation}

Applying \secref{sss:change of source} again, we obtain a symmetric monoidal functor
$$\ul\Phi:\AGCat_{/B}\to \AGCat_{/B'}$$
equipped with identifications
$$\bi'\simeq \ul\Phi\circ \bi \text{ and } \ul\Shv(\Phi(-)) \simeq \ul\Phi\circ \ul\Shv(-).$$

If \eqref{e:bc nat trans fields} is an isomorphism, then $\ul\Phi$ is fully faithful.

\ssec{Adjunctions in $\AGCat$}

\sssec{}

The category $\AGCat_{/B}$ is tensored over $\DGCat$, and in particular $\on{Cat}$; hence, it naturally has the structure of a 2-category (see \Cref{sss:enriched 2cat}).

\medskip

We have the following characterization of adjunctions in $\AGCat$.

\begin{prop} \label{p:adj AGCat1}
Let $F: \ul\bC_1 \to \ul\bC_2$ be a 1-morphism in $\AGCat_{/B}$.  Then $F$ admits a left (resp. right) adjoint if and only if
for every scheme $X$, the corresponding functor
$$\ul\bC_1(X) \to \ul\bC_2(X) $$
admits a left (resp. right) adjoint, and for every scheme $X'$ and a \emph{constructible sheaf}
$\CF\in \Shv(X \underset{B}\times X')^{c}$, the square
$$
\xymatrix{
\ul\bC_1(X) \ar[r]\ar[d] & \ul\bC_2(X) \ar[d]\\
\ul\bC_1(X') \ar[r] & \ul\bC_2(X')
}
$$
satisfies the left (resp. right) Beck-Chevalley condition (see \Cref{sss:BC conditions}).
\end{prop}

\begin{proof}

We prove the assertion for the left adjoint (the proof is identical for the right adjoint).

\medskip

Suppose the conditions are satisfied.
By Proposition \ref{p:adjoints in V-functors}, it remains to show that for every scheme $X'$, the diagram
$$
\xymatrix{
\Shv(X\times X') \otimes \ul\bC_1(X) \ar[r]\ar[d] & \Shv(X \times X') \otimes \ul\bC_2(X) \ar[d] \\
\ul\bC_1(X') \ar[r] & \ul\bC_2(X')
}$$
satisfies the left Beck-Chevalley condition.

\medskip

The category $\Shv(X\times X') \otimes \ul\bC_2(X)$ is generated under colimits by objects of the form
$\CF \boxtimes \bc_2$ for $\CF \in \Shv(X\times X')^{c}$
and $\bc_2\in \ul\bC_2(X)$.  Thus, since all functors preserve colimits,
it suffices to check the Beck-Chevalley condition on these objects.  This is exactly the condition in the proposition.

\end{proof}

\sssec{}

\Cref{p:basic bc} gives a general characterization of functors from the 2-category $\on{Adj}$.  Specializing to the target $\AGCat_{/B}$, we have

\begin{prop}\label{p:adj in dgcat ag}
The restriction functor
$$ \on{Funct}(\on{Adj}, \AGCat_{/B}) \to \on{Funct}([1], \AGCat_{/B}) $$
is 1-full and the essential image is characterized as follows:

\smallskip

\noindent{\em(a)}
An object $\ul\bC_1 \to \ul\bC_2 \in \on{Funct}([1], \AGCat)$ is in the essential image if and only if
it admits a right (resp. left) adjoint in $\AGCat_{/B}$;

\smallskip

\noindent{\em(b)}
A morphism from $(\ul\bC_1 \to \ul\bC_2)$ to $(\ul\bC_1' \to \ul\bC_2')$ lies in the essential image if and only if
for every scheme $X$, the corresponding square
$$ \xymatrix{
\ul\bC_1(X) \ar[r]\ar[d] & \ul\bC_2(X)\ar[d]\\
\ul\bC'_1(X) \ar[r] & \ul\bC_2'(X)
}$$
satisfies the right (resp. left) Beck-Chevalley condition.

\end{prop}

\section{The value of \texorpdfstring{$\uShv$}{ulShv} on prestacks} \label{s:prestacks}

In this subsection we study the two ways to associate objects of $\AGCat$ to prestacks
$$\CY\mapsto \ul\Shv(\CY) \text { and } \CY\mapsto \ul\Shv(\CY)_{\on{co}}.$$

We study their interaction, questions of dualizability and behavior with respect to products.

\ssec{Sheaves on prestacks}

\sssec{}\label{sss:def of sh on prestack}

Let $\CY$ be a prestack locally of finite type over $B$.  Recall from \cite[Appendix F]{AGKRRV1}, that the category of $\ell$-adic sheaves on $\CY$
is defined by
$$ \Shv(\CY):=  \underset{S\in (\affSch_{/\CY})^{\on{op}}}{\on{lim}}\  \Shv(S) \in \DGCat,$$
where the limit is taken with respect to $!$-pullbacks as transition functors.

\medskip

Note that by Zariski descent, we can equivalently rewrite the above limit as
$$\underset{X\in (\Sch_{/\CY})^{\on{op}}}{\on{lim}}\  \Shv(X) \in \DGCat.$$

In particular, when $\CY=X$ is a scheme, the above definition recovers the usual $\Shv(X)$.

\begin{rem} \label{r:comp gen prestack}

Since we are working in the constructible context, we can rewrite the limit presentation of $\Shv(X)$ as
a colimit:
$$\underset{X\in \Sch_{/\CY}}{\on{colim}}\  \Shv(X),$$
with $(-)_!$ as transition functors, see \cite[Chapter 1, Proposition 2.5.7]{GR1}.

\medskip

In particular, $\Shv(\CY)$ is compactly generated and hence dualizable.

\end{rem}

\sssec{} \label{sss:def of ush on prestack}

The category $\Shv(\CY)$ naturally upgrades to an object of $\AGCat_{/B}$; namely, we define
$$ \uShv(\CY) := \underset{X\in (\Sch_{/\CY})^{\on{op}}}{\on{lim}}\  \uShv(S) \in \AGCat_{/B}.$$

\medskip

It follows from \secref{sss:limits AGCat} that for any scheme $Y$,
\begin{equation} \label{e:value of ushv prestack}
\uShv(\CY)(X) \simeq \Shv(\CY \underset{B}\times X).
\end{equation}
In particular, for every $X$, the DG category $\uShv(\CY)(X)$ is compactly generated.

\sssec{} \label{sss:def of ush on prestack !}

Let $f: \CY_1 \to \CY_2$ be a morphism of prestacks over $B$. The functor between index categories
$$\affSch_{/\CY_1}\overset{f\circ -}\to \affSch_{/\CY_2}$$
gives rise to a pullback functor
$$f^!:\ul\Shv(\CY_2)\to \ul\Shv(\CY_1).$$

In terms of \eqref{e:value of ushv prestack}, the above functor is given by
$$(f\times \on{id})^!:\Shv(\CY_2\underset{B}\times X)\to \Shv(\CY_1\underset{B}\times X).$$

\medskip

The assignment $\CY\mapsto \ul\Shv(\CY)$ with !-pullbacks extends to a functor
$$(\PreStk_{/B})^{\on{op}}\to \AGCat_{/B}.$$

\sssec{}

Let $f: \CY_1 \to \CY_2$ is a schematic morphism of prestacks over $B$.  In this case, we can define the map
$$ f_*: \uShv(\CY_1) \to \uShv(\CY_2) $$
given in terms of \secref{sss:AGCat expl} by $*$-pushforward (which is defined for schematic morphisms of prestacks):
$$(f \times \on{id}_X)_*: \Shv(\CY_1 \underset{B}\times X) \to \Shv(\CY_2 \underset{B}\times X).$$

\sssec{}\label{sss:extension of ushv to prestacks}

The functoriality of the $*$-pushforward functor along schematic maps can be organized as follows.  Consider the 1-full symmetric monoidal category
$$\on{Corr}_{\on{sch;all}}(\on{PreStk}_{/B})\subset  \on{Corr}(\on{PreStk}_{/B}),$$
where we require the \emph{forward} morphisms to be schematic (see \secref{sss:corr var}).

\medskip

Let
\begin{equation}\label{e:ushv from prestacks}
\uShv(-): \on{Corr}_{\on{sch;all}}(\on{PreStk}_{/B}) \to \AGCat_{/B}
\end{equation}
be the right Kan extension of \Cref{e:ushv as functor out of corr} along the inclusion functor
$$\on{Corr}(\on{Sch}_{/B}) \to \on{Corr}_{\on{sch;all}}(\on{PreStk}_{/B}). $$

\medskip

Note that for every prestack $\CY$, the (opposite of the) functor
$$(\Sch^{\on{op}})_{/\CY}\to \on{Corr}(\on{Sch}_{/B})
\underset{\on{Corr}_{\on{sch;all}}(\on{PreStk}_{/B})}\times \on{Corr}_{\on{sch;all}}(\on{PreStk}_{/B})_{\CY/}$$
is cofinal (since it admits a right adjoint). In particular, the restriction of \Cref{e:ushv from prestacks} to
$$(\on{PreStk}_{/B})^{\on{op}} \subset \on{Corr}_{\on{sch;all}}(\on{PreStk}_{/B}) $$
agrees with that in \Cref{sss:def of ush on prestack !}.

\ssec{Cosheaves on prestacks}

\sssec{}

For a prestack $\CY$ define
$$\Shv(\CY)_{\on{co}}:=\underset{S\in \affSch_{/\CY}}{\on{colim}}\  \Shv(S) \in \DGCat,$$
where the colimit is taken with respect to $*$-pullshforwards as transition functors.

\medskip

The equivalence \eqref{e:Zar codescent} implies that the above colimit maps isomorphically to
$$\underset{X\in \Sch_{/\CY}}{\on{colim}}\  \Shv(X).$$

In particular, when $\CY=X$ is a scheme, we recover the usual $\Shv(X)$.

\sssec{} \label{sss:functoriality co}

For a map of prestacks $f:\CY_1\to \CY_2$ we have a tautologically defined functor
$$f_{*,\on{co}}:\Shv(\CY_1)_{\on{co}}\to \Shv(\CY_2)_{\on{co}}.$$

If $f$ is schematic, we have the base change functor
$$\affSch_{/\CY_2}\to \affSch_{/\CY_1}, \quad S\mapsto S\underset{\CY_2}\times \CY_1,$$
which gives rise to a functor
$$f^!_{\on{co}}: \Shv(\CY_2)_{\on{co}}\to \Shv(\CY_1)_{\on{co}}.$$

\sssec{}

The assignment
$$\CY\mapsto \Shv(\CY)_{\on{co}}$$
extends to a functor
$$\Shv(-)_{\on{co}}:\on{Corr}_{\on{all,sch}}(\on{PreStk})\to \DGCat,$$
where
$$\on{Corr}_{\on{all,sch}}(\on{PreStk})\subset \on{Corr}(\on{PreStk})$$ is a 1-full subcategory,
in which the backward morphisms are required to be schematic.

\medskip

Namely, $\Shv(-)_{\on{co}}$
is the left Kan extension of $$\Shv(-):\on{Corr}(\Sch)\to \DGCat$$ along
$$\on{Corr}(\Sch)\hookrightarrow \on{Corr}_{\on{all,sch}}(\on{PreStk}).$$

\medskip

Since the functor
$$\Sch_{/\CY}\to \on{Corr}(\Sch)\underset{\on{Corr}_{\on{all,sch}}(\on{PreStk})}\times (\on{Corr}_{\on{all,sch}}(\on{PreStk}))_{/\CY}$$
is cofinal (it admits a left adjoint), the value of the above left Kan extension on $\CY$ recovers $\Shv(\CY)_{\on{co}}$.

\sssec{}

For a prestack $\CY$ over $B$, set
$$\ul\Shv(\CY)_{\on{co}}:=\underset{X\in \Sch_{/\CY}}{\on{colim}}\  \uShv(X)\in \AGCat_{/B},$$
where the colimit is taken with respect to the *-pushforward functors.

\medskip

For $\CY=X\in \Sch$, we recover the usual object $\uShv(X)$.

\sssec{}

We claim:

\begin{lem} \label{l:value of u co shv}
The value $\ul\Shv(\CY)_{\on{co}}(X)$ of $\ul\Shv(\CY)_{\on{co}}$ on $X$ is given by
$\Shv(\CY\underset{B}\times X)_{\on{co}}$.
\end{lem}

\begin{proof}

By \secref{sss:prop Enh}, we have to establish an isomorphism
$$\underset{Z\in \Sch_{/\CY}}{\on{colim}}\, \Shv(Z\underset{B}\times X) \simeq
\underset{Z'\in \Sch_{/\CY\underset{B}\times X}}{\on{colim}}\, \Shv(Z').$$

This follows from the fact that the functor
$$\Sch_{/\CY}\to \Sch_{/\CY\underset{B}\times X}, \quad Z\mapsto Z\underset{B}\times X$$
is cofinal. Indeed, it has a left adjoint given by
$$(Z' \to \CY \underset{B}\times Y) \mapsto (Z' \to \CY).$$

\end{proof}

\sssec{}

As above, for any $f:\CY_1\to \CY_2$, we have a map
$$f_{\on{co},*}:\ul\Shv(\CY_1)_{\on{co}}\to \ul\Shv(\CY_2)_{\on{co}},$$
and when $f$ is schematic also a map
$$f_{\on{co}}^!:\ul\Shv(\CY_2)_{\on{co}}\to \ul\Shv(\CY_1)_{\on{co}}.$$

When we evaluate the above functors on $X\in \Sch$, in terms of the identification of \lemref{l:value of u co shv},
we recover the corresponding functors from \secref{sss:functoriality co} for
$$\CY_1\underset{B}\times X \overset{f\times \on{id}}\to \CY_2\underset{B}\times X.$$

\sssec{}

We define the functor
$$\uShv(-)_{\on{co}}:\on{Corr}_{\on{all,sch}}(\on{PreStk}_{/B})\to \AGCat_{/B},$$
to be the left Kan extension of
$$\uShv(-):\on{Corr}(\Sch_{/B})\to \AGCat_{/B}$$ along
$$\on{Corr}(\Sch_{/B})\hookrightarrow \on{Corr}_{\on{all,sch}}(\on{PreStk}_{/B}).$$

\medskip

As above, the value of the above left Kan extension on $\CY$ recovers $\ul\Shv(\CY)_{\on{co}}$.

\ssec{The \emph{right-lax} symmetric monoidal structure(s)}

\sssec{}

By \cite[Theorem B.1.3]{AGKRRV1} (or a special case of \cite[Chapter 9, Proposition 3.2.4]{GR1}),
the functor \Cref{e:ushv from prestacks} is naturally right-lax symmetric monoidal.

\medskip

Concretely, this means that for a pair of prestacks $\CY_1$ and $\CY_2$ over $B$, we have a canonically defined map
\begin{equation} \label{e:ten prod prestack}
\uShv(\CY_1)\otimes \uShv(\CY_2)\overset{\boxtimes}\to \uShv(\CY_1\underset{B}\times \CY_2).
\end{equation}

\sssec{}

Explicitly, the map \eqref{e:ten prod prestack} is given by
\begin{multline} \label{e:boxtimes}
\uShv(\CY_1)\otimes \uShv(\CY_2) \simeq
\left(\underset{X_1 \in (\Sch_{/\CY_1})^{\on{op}}}{\on{lim}}\,  \uShv(X_1)\right)\otimes \uShv(\CY_2) \overset{\rm{(A)}}\to \\
\to \underset{X_1 \in (\Sch_{/\CY_1})^{\on{op}}}{\on{lim}}\,  \left(\uShv(X_1)\otimes \uShv(\CY_2)\right)
\simeq \underset{X_1 \in (\Sch_{/\CY_1})^{\on{op}}}{\on{lim}}\,
\left(\uShv(X_1)\otimes \left(\underset{X_2 \in (\Sch_{/\CY_2})^{\on{op}}}{\on{lim}}\, \uShv(X_2)\right)\right) \overset{\rm{(B)}}\to \\
\to \underset{X_1 \in (\Sch_{/\CY_1})^{\on{op}}}{\on{lim}}\,
\left(\underset{X_2 \in (\Sch_{/\CY_2})^{\on{op}}}{\on{lim}}\, \uShv(X_1)\otimes \uShv(X_2)\right)\simeq \\
\simeq \underset{X_1 \in (\Sch_{/\CY_1})^{\on{op}},X_2 \in (\Sch_{/\CY_2})^{\on{op}}}{\on{lim}}\,
\uShv(X_1)\otimes \uShv(X_2)
\simeq \underset{X_1 \in (\Sch_{/\CY_1})^{\on{op}},X_2 \in (\Sch_{/\CY_2})^{\on{op}}}{\on{lim}}\,
\uShv(X_1 \underset{B}\times X_2) \simeq \\
\simeq \underset{X \in (\Sch_{/\CY_1 \underset{B}\times \CY_2})^{\on{op}}}{\on{lim}}\, \uShv(X)=\ul\Shv(\CY_1 \underset{B}\times \CY_2),
\end{multline}
where the last isomorphism is due to the fact that the functor
\begin{equation} \label{e:prod cofinal}
\Sch_{/\CY_1} \times \Sch_{/\CY_2} \to \Sch_{/(\CY_1 \underset{B}\times \CY_2)},
\quad
((X_1\to \CY_1), (X_2\to \CY_2))\mapsto (X_1 \underset{B}\times X_2\to \CY_1 \underset{B}\times \CY_2)
\end{equation}
is cofinal.

\sssec{}

When $\CY_1$ and $\CY_2$ are both schemes, the map \eqref{e:ten prod prestack} is an isomorphism; this
was the main point of introducing $\AGCat_{/B}$.

\medskip

However, it is not necessarily an isomorphism for
general prestacks $\CY_1$ and $\CY_2$: this is due to the fact that the map (A) in \eqref{e:boxtimes}
is not necessarily an isomorphism (we will see, however, that the map (B) \emph{is} always an isomorphism).

\medskip

However, in the course of this section we will distinguish a class of prestacks $\CY_1$, namely,
\emph{AG tame} prestacks, for which the map \eqref{e:ten prod prestack} is
an isomorphism against any $\CY_2$, see \propref{c:AG good ten}.

\sssec{}

By contrast, the functor
$$\uShv(-)_{\on{co}}:\on{Corr}_{\on{all,sch}}(\on{PreStk}_{/B})\to \AGCat_{/B}$$
is \emph{strictly} symmetric monoidal.

\medskip

Namely, for a pair of prestacks $\CY_1$ and $\CY_2$, we have
\begin{multline} \label{e:boxtimes co}
\uShv(\CY_1)_{\on{co}}\otimes \uShv(\CY_2)_{\on{co}} \simeq
\left(\underset{X_1 \in \Sch_{/\CY_1}}{\on{colim}}\,  \uShv(X_1)\right)\otimes
\left(\underset{X_2 \in \Sch_{/\CY_2}}{\on{colim}}\,  \uShv(X_2)\right) \simeq \\
\simeq \underset{X_1 \in \Sch_{/\CY_1},X_2 \in \Sch_{/\CY_2}}{\on{colim}}\,   \uShv(X_1)\otimes  \uShv(X_2)\simeq
\underset{X_1 \in \Sch_{/\CY_1},X_2 \in \Sch_{/\CY_2}}{\on{colim}}\, \uShv(X_1 \underset{B}\times X_2) \simeq \\
\simeq \underset{X \in \Sch_{/\CY_1 \underset{B}\times \CY_2}}{\on{colim}}\, \uShv(X)=\ul\Shv(\CY_1 \underset{B}\times \CY_2)_{\on{co}},
\end{multline}
where the last isomorphism is due to the cofinality of \eqref{e:prod cofinal}.

\ssec{Tame prestacks}

In this subsection we let $\CY$ be a prestack with a schematic diagonal.

\sssec{} \label{sss:Omega}

Due to the assumption on $\CY$, for any a map $f:X\to \CY$, where $X$ is a scheme, is schematic.
Hence, we have a well-defined funcor
$$f_*:\Shv(X)\to \Shv(\CY).$$

These functors assemble to a functor
\begin{equation} \label{e:Omega}
\Omega_\CY:\Shv(\CY)_{\on{co}}\to \Shv(\CY).
\end{equation}

\sssec{}

Note that for a schematic map between prestacks $f:\CY_1\to \CY_2$, we have the commutative diagrams
\begin{equation} \label{e:Omega *}
\CD
\Shv(\CY_1)_{\on{co}} @>{\Omega_{\CY_1}}>>  \Shv(\CY_1) \\
@V{f_{\on{co},*}}VV @VV{f_*}V \\
\Shv(\CY_2)_{\on{co}} @>{\Omega_{\CY_2}}>>  \Shv(\CY_2)
\endCD
\end{equation}
and
\begin{equation} \label{e:Omega !}
\CD
\Shv(\CY_1)_{\on{co}} @>{\Omega_{\CY_1}}>>  \Shv(\CY_1) \\
@A{f^!_{\on{co}}}AA @AA{f^!}A \\
\Shv(\CY_2)_{\on{co}} @>{\Omega_{\CY_2}}>>  \Shv(\CY_2)
\endCD
\end{equation}

\sssec{}

We shall say that $\CY$ is \emph{tame} if $\Omega_\CY$ is an equivalence.

\medskip

Not all prestacks are tame. However, in \secref{sss:Verdier compat} we will
single out a class of algebraic stacks, for which it is. Conjecturally, this includes all quasi-compact algebraic
stacks.

\sssec{}

Recall that $\CY$ is a \emph{pseudo-scheme} if (up to sheafification in the h-topology)
it can be rewritten as a (not necessarily filtered) colimit of schemes with transition maps that are proper.  For example,
and ind-scheme is such.

\medskip

We claim:

\begin{lem}
Any pseudo-scheme is tame.
\end{lem}

\begin{proof}

Write (up to h-sheafification)
$$\CY=\underset{i\in I}{\on{colim}}\, Y_i,$$
where $Y_i$ are schemes are the transition maps are proper. Then
\begin{equation} \label{e:ps sch 1}
\Shv(\CY)\simeq \underset{i\in I^{\on{op}},(-)^!}{\on{lim}}\, \Shv(Y_i),
\end{equation}
and
\begin{equation} \label{e:ps sch 2}
\Shv(\CY)_{\on{co}}\simeq \underset{i\in I,(-)_*}{\on{colim}}\, \Shv(Y_i).
\end{equation}

Note, however, that since the transition functors in \eqref{e:ps sch 1} are the right adjoints
of those in \eqref{e:ps sch 2}, we obtain that the functor
\begin{equation} \label{e:Omega proof}
\Shv(\CY)_{\on{co}}\simeq \underset{i\in I,(-)_*}{\on{colim}}\, \Shv(Y_i)\to \Shv(\CY)
\end{equation}
assembled from the left adjoints of the evaluation functors
$$\Shv(\CY)\to  \Shv(Y_i),$$
is an equivalence, see \cite[Chapter 1, Proposition 2.5.7]{GR1}.

\medskip

Now, by construction, the functor \eqref{e:Omega proof} is the functor $\Omega_\CY$.

\end{proof}

\sssec{}

As in \secref{sss:Omega}, the maps
$$f_*:\ul\Shv(X)\to \uShv(\CY), \quad (X,f)\in \Sch_{/\CY}$$
assemble into a map
\begin{equation}\label{e:ul Omega}
\ul\Omega_\CY:\uShv(\CY)_{\on{co}}\to \uShv(\CY).
\end{equation}

\medskip

In terms of the identification given by \lemref{l:value of u co shv}, the values of the map
\eqref{e:ul Omega} are given by the functors \eqref{e:Omega} for $\CY\underset{B}\times X$.

\medskip

The commutative diagrams \eqref{e:Omega *} and \eqref{e:Omega !} give rise to
commutative diagrams

\begin{equation} \label{e:u Omega *}
\CD
\uShv(\CY_1)_{\on{co}} @>{\ul\Omega_{\CY_1}}>>  \uShv(\CY_1) \\
@V{f_{\on{co},*}}VV @VV{f_*}V \\
\uShv(\CY_2)_{\on{co}} @>{\ul\Omega_{\CY_2}}>>  \uShv(\CY_2)
\endCD
\end{equation}
and
\begin{equation} \label{e:u Omega !}
\CD
\uShv(\CY_1)_{\on{co}} @>{\ul\Omega_{\CY_1}}>>  \uShv(\CY_1) \\
@A{f^!_{\on{co}}}AA @AA{f^!}A \\
\uShv(\CY_2)_{\on{co}} @>{\ul\Omega_{\CY_2}}>>  \uShv(\CY_2)
\endCD
\end{equation}

\sssec{}

We shall say that $\CY\in \on{PreStk}_{/B}$ is \emph{universally tame} if the map \eqref{e:ul Omega} is an isomorphism
in $\AGCat_{/B}$. Explicitly, this means that for any $X\in \Sch$, the prestack
$$\CY\underset{B}\times X$$
is tame.

\medskip

We shall say that $\CY\in \on{PreStk}_{/B}$ is \emph{AG tame} of both $\CY$ and  $\CY \underset{B}\times \CY$ are universally tame.

\sssec{}

Let $\CY_1$ and $\CY_2$ be a pair of prestacks over $B$, both with a schematic diagonal. Note that we have a commutative diagram
\begin{equation} \label{e:Omega prod diag}
\CD
\ul\Shv(\CY_1)_{\on{co}}\otimes \ul\Shv(\CY_2)_{\on{co}} @>{\text{\eqref{e:boxtimes co}}}>{\sim}> \ul\Shv(\CY_1\underset{B}\times \CY_2)_{\on{co}} \\
@V{\ul\Omega_{\CY_1}\otimes \ul\Omega_{\CY_2}}VV @VV{\ul\Omega_{\CY_1\underset{B}\times \CY_2}}V \\
\ul\Shv(\CY_1)\otimes \ul\Shv(\CY_2) @>>{\text{\eqref{e:ten prod prestack}}}> \ul\Shv(\CY_1\underset{B}\times \CY_2).
\endCD
\end{equation}

Hence, we obtain:

\begin{cor} \label{c:Omega prod diag}
Suppose that $\CY_1$, $\CY_2$ are universally tame. Then the following are equivalent:

\medskip

\noindent{\em(i)} The prestack $\CY_1\underset{B}\times \CY_2$ is also universally tame;

\noindent{\em(ii)} The map
$$\ul\Shv(\CY_1)\otimes \ul\Shv(\CY_2) \overset{\boxtimes}\to \ul\Shv(\CY_1\underset{B}\times \CY_2)$$
of \eqref{e:ten prod prestack} is an isomorphism.

\end{cor}

In particular, we obtain:

\begin{cor}\label{c:tensor product verd comp}
Let $\CY$ be an AG tame prestack.  Then the map
$$ \uShv(\CY) \otimes \uShv(\CY) \overset{\boxtimes}\to \uShv(\CY \underset{B}\times \CY)$$
is an isomorphism.
\end{cor}

\ssec{Dualizability in $\AGCat$}

\sssec{} \label{sss:self-dual AG sch}

Since the functor \Cref{e:ushv as functor out of corr} is symmetric monoidal, by \secref{sss:dual corr},
for every scheme $X$, the object
$$\uShv(X) \in \AGCat_{/B}$$
is canonically self-dual with the unit map given by
$$\uno_{\AGCat_{/B}} = \uShv(B) \overset{(p_{X/B})^!}\longrightarrow \uShv(X)
\overset{(\Delta_{X/B})_*}\longrightarrow \uShv(X \underset{B}\times X)\simeq \uShv(X) \otimes \uShv(X)$$
and the counit map given by
$$\uShv(X)\otimes \uShv(X) \simeq \uShv(X\underset{B}\times X)
\overset{(\Delta_{X/B})^!}\longrightarrow \uShv(X) \overset{(p_{X/B})_*}\longrightarrow \uShv(B) \simeq \uno_{\AGCat_{/B}},$$
where $p_{X/B}: X \to B$ is the projection map and $\Delta_{X/B}: X \to X\underset{B}\times X$ is the diagonal.

\sssec{} \label{sss:dual of maps}

It follows formally that for a map of schemes $f:X_1\to X_2$, under the above identifications
$$\ul\Shv(X_i)^\vee \simeq \ul\Shv(X_i),$$
we have
$$(f^!)^\vee\simeq f_*$$
as maps
$$\ul\Shv(X_1)\to \ul\Shv(X_2).$$

\sssec{}

We will be interested in more general dualizable objects in $\AGCat$.  First, we claim:

\begin{prop}\label{p:tensor with dualizable}
Let $\CY$ be a prestack such that $\uShv(\CY) \in \AGCat_{/B}$ is a dualizable object.
Then for any prestack $\CY'$, the natural map \emph{(}see \eqref{e:ten prod prestack}\emph{)}
$$ \uShv(\CY') \otimes \uShv(\CY) \overset{\boxtimes}{\to} \uShv(\CY' \underset{B}\times \CY)$$
is an isomorphism.
\end{prop}

\begin{proof}

We need to show that both maps (A) and (B) in \eqref{e:boxtimes} are isomorphisms
(with $\CY_1=\CY'$ and $\CY_2=\CY$).

\medskip

The fact that (A) is an isomorphism follows from the assumption that $\uShv(\CY)$
is dualizable.

\medskip

The fact that (B) is an isomorphism follows from the fact that the object $\ul\Shv(X)$
for $X\in \Sch$ is dualizable.

\end{proof}

\sssec{}

Next, we claim:

\begin{prop} \label{p:internal Hom co}
For a pair of prestacks we have a canonical identifification
$$\uHom_{\AGCat_{/B}}(\uShv(\CY_1)_{\on{co}}, \uShv(\CY_2)) \simeq \uShv(\CY_1 \underset{B}\times \CY_2).$$
\end{prop}

\begin{proof}

We have
\begin{multline*}
\uHom_{\AGCat_{/B}}(\uShv(\CY_1)_{\on{co}}, \uShv(\CY_2)) \simeq
 \underset{X \in (\affSch_{/\CY_1})^{\on{op}}}{\on{lim}}\, \uHom_{\AGCat_{/B}}(\uShv(X), \uShv(\CY_2)) \overset{\text{\secref{sss:self-dual AG sch}}}\simeq \\
 \simeq  \underset{X \in (\affSch_{/\CY_1})^{\on{op}}}{\on{lim}}\,  \uShv(X)\otimes \uShv(\CY_2) \overset{\text{\propref{p:tensor with dualizable}}}\simeq  \\
\simeq  \underset{X \in (\affSch_{/\CY_1})^{\on{op}}}{\on{lim}}\, \uShv(X \underset{B}\times \CY_2) \simeq
 \uShv(\CY_1 \underset{B}\times \CY_2),
 \end{multline*}
 where the last isomorphism follows from the cofinality of \eqref{e:prod cofinal}.
 \end{proof}

 In particular, we obtain:

 \begin{cor} \label{c:dual co ushv}
 For a prestack $\CY$, we have
 $$\uHom_{\AGCat_{/B}}(\uShv(\CY)_{\on{co}}, \uno_{\AGCat_{/B}})\simeq \uShv(\CY).$$
 \end{cor}

\sssec{}  \label{sss:dual co ushv}

Using \secref{sss:dual of maps}, we obtain that for a map $f:\CY_1\to \CY_2$ between arbitrary prestacks,
under the identification of \corref{c:dual co ushv}, the map
$$(f_{\on{co},*})^\vee:\uHom_{\AGCat_{/B}}(\uShv(\CY_2)_{\on{co}}, \uno_{\AGCat_{/B}})\to
\uHom_{\AGCat_{/B}}(\uShv(\CY_1)_{\on{co}}, \uno_{\AGCat_{/B}})$$
corresponds to $f^!$.

\medskip

Similarly, we obtain that if $f$ is schematic, the functor
 $$(f^!_{\on{co}})^\vee:\uHom_{\AGCat_{/B}}(\uShv(\CY_1)_{\on{co}}, \uno_{\AGCat_{/B}})\to
\uHom_{\AGCat_{/B}}(\uShv(\CY_1)_{\on{co}}, \uno_{\AGCat_{/B}})$$
corresponds to $f_*$.

 \sssec{}

 As another consequence of \propref{p:internal Hom co}, we obtain:

\begin{cor} \label{c:hom verd comp}
Let $\CY$ be universally tame. There is a natural isomorphism
$$\uHom_{\AGCat_{/B}}(\uShv(\CY), \uShv(\CY')) \simeq \uShv(\CY \underset{B}\times \CY') .$$
\end{cor}

\sssec{} \label{sss:univ self dual}

Thus, we obtain that \emph{if} $\CY$ is a universally tame prestack such that
$\uShv(\CY)$ is dualizable, \emph{then} $\uShv(\CY)$ is canonically self-dual as an object of $\AGCat_{/B}$.

\medskip

In the next subsection we will see that this is the case for AG tame prestacks.

\ssec{Dualizability of AG tame prestacks}

\sssec{}

We are going to prove:

\begin{thm} \label{t:self duality of verd comp}
Let $\CY$ be an AG tame prestack.  Then the object $\uShv(\CY)$ is dualizable in $\AGCat_{/B}$.
\end{thm}

\sssec{}

The proof of \thmref{t:self duality of verd comp} will rely on the following characterization of dualizability:

\medskip

Recall that if $\CA$ is a monoidal category we have the notion of left and right dualizable objects (see \cite[Sect. 4.6.1]{HA})
which agree in the case when $\CA$ is symmetric monoidal.  Recall also that for two objects $a_1,a_2\in \CA$
the internal hom $\uHom_{\CA}(a,b)$
is the object representing the functor
$$ a \mapsto \on{Maps}_\CA(a \otimes a_1, a_2) .$$
The proof of the following is identical to \cite[Lemma 4.6.1.6]{HA}.

\begin{prop} \label{p:charact of duality}
Let $\CA$ be a monoidal category.  An object $a \in \CA$ is right dualizable if and only if

\smallskip

\noindent{\em(i)} There exist internal hom objects $\uHom_\CA(a,\uno_\CA)$ and $\uHom_\CA(a,a)$, and

\smallskip

\noindent{\em(ii)} The tautological map $\uHom_\CA(a,\uno_\CA)\otimes a\to \uHom_\CA(a,a)$
is an isomorphism.

\end{prop}

\sssec{Proof of \thmref{t:self duality of verd comp}}

By \propref{p:charact of duality} we need to show that the canonical map
\begin{equation} \label{e:XX}
\uHom_{\AGCat_{/B}}(\uShv(\CY), \uno_{\AGCat_{/B}})\otimes \uShv(\CY)\to
\uHom_{\AGCat_{/B}}(\uShv(\CY),\uShv(\CY))
\end{equation}
is an isomorphism. We identify
$$\uHom_{\AGCat_{/B}}(\uShv(\CY), \uno_{\AGCat_{/B}}) \simeq \ul\Shv(\CY) \text{ and }
\uHom_{\AGCat_{/B}}(\uShv(\CY),\uShv(\CY)) \simeq \ul\Shv(\CY \underset{B}\times \CY)$$
by \corref{c:hom verd comp}. Under these identifications, the map \eqref{e:XX} is the map
$$\ul\Shv(\CY) \otimes \ul\Shv(\CY)\overset{\boxtimes}{\to} \ul\Shv(\CY \underset{B}\times \CY)$$
of \eqref{e:boxtimes}. The latter map is an isomorphism by \Cref{c:tensor product verd comp}.

\qed[\thmref{t:self duality of verd comp}]

\sssec{}

By \Cref{p:tensor with dualizable}, we obtain:

\begin{cor} \label{c:AG good ten}
Let $\CY$ be an AG tame prestack and $\CY'$ a prestack.  Then the canonical map
$$\uShv(\CY) \otimes \uShv(\CY')\to \uShv(\CY \underset{B}\times \CY')$$
is an isomorphism.
\end{cor}

\sssec{}

We now obtain a characterization of AG tame prestacks in terms of $\AGCat_{/B}$:

\begin{prop}\label{p:AG Verdier in terms of dgcatag}
Let $\CY$ be a \emph{universally} tame prestack. Then the following conditions are equivalent:

\smallskip

\noindent{\em(i)} $\CY$ is AG tame;

\medskip

\noindent{\em(ii)} The object $\uShv(\CY)\in \AGCat_{/B}$ is dualizable;

\medskip

\noindent{\em(iii)} The map $\uShv(\CY)\otimes \uShv(\CY)\to \uShv(\CY\underset{B}\times \CY)$ is an isomorphism.

\end{prop}

\begin{proof}

First, (i) implies (ii) by \Cref{t:self duality of verd comp}. Next, (ii) implies (iii) by \propref{p:tensor with dualizable}.
Finally, (iii) implies (i) by \corref{c:Omega prod diag}.

\end{proof}

\sssec{}

We can now prove that the notion of being AG Verdier-compatible is stable under products:

\begin{prop}\label{p:AG verdier products}
If $\CY_1$ and $\CY_2$ are AG tame, then so is $\CY_1 \underset{B}\times \CY_2$.
\end{prop}

\begin{proof}
We apply the characterization of AG tame prestacks from \Cref{p:AG Verdier in terms of dgcatag}.
By \corref{c:AG good ten},
we have:
$$ \uShv(\CY_1) \otimes \uShv(\CY_2) \simeq \uShv(\CY_1 \underset{B}\times \CY_2) .$$
Hence, since $\uShv(\CY_1)$ and $\uShv(\CY_2)$ are dualizable, so is $\uShv(\CY_1 \underset{B}\times \CY_2)$.

\medskip

The fact that $\CY_1 \underset{B}\times \CY_2$ is universally tame follows from
\corref{c:Omega prod diag} using \propref{p:tensor with dualizable}.

\end{proof}

\ssec{Description of unit and counit}

In this section we let $\CY$ be an AG tame prestack. By \thmref{t:self duality of verd comp}, we know that
$\uShv(\CY)$ is dualizable as an object of  $\AGCat_{/B}$. Moreover, by \secref{sss:univ self dual}, it is self-dual.

\medskip

In this subsection we will describe explicitly the unit and counit maps for this self-duality.

\sssec{}

Note that since $\CY$ was assumed to have a schematic diagonal, we have a well-defined map
$$(\Delta_{\CY/B})_*:\uShv(\CY)\to \uShv(\CY\underset{B}\times \CY).$$

First, we will prove:

\begin{prop} \label{p:unit self Y}
The unit of the self-duality on $\uShv(\CY)$ is given by the map
\begin{equation} \label{e:unit Y formula}
\uno_{\AGCat_{/B}}=\ul\Shv(B) \overset{(p_{\CY/B})^!}\to \ul\Shv(\CY)\overset{(\Delta_{\CY/B})_*}\to
\uShv(\CY\underset{B}\times \CY) \overset{\text{\corref{c:tensor product verd comp}}}\simeq
\uShv(\CY)\otimes \uShv(\CY).
\end{equation}
\end{prop}

\begin{proof}

We will think of the left copy of $\uShv(\CY)$ in the tensor product as $\uShv(\CY)^\vee$ via
$$\uShv(\CY)^\vee \overset{\text{\corref{c:dual co ushv}}}\simeq \uShv(\CY)_{\on{co}}\overset{\ul\Omega_\CY}\to
\uShv(\CY).$$

Since $\uShv(\CY)$ is the limit of $\ul\Shv(X)$ over $X\overset{f}\to \CY$, using \secref{sss:self-dual AG sch},
we conclude that we need to show that for every $f:X\to \CY$, the composition of \eqref{e:unit Y formula} with
$$\uShv(\CY)\otimes \uShv(\CY) \overset{\on{id}\otimes f^!}\longrightarrow \uShv(\CY)\otimes \uShv(X)\simeq
\uShv(\CY)^\vee\otimes \uShv(X)$$
identifies with
\begin{multline*}
\uno_{\AGCat_{/B}}=\ul\Shv(B) \overset{(p_{X/B})^!}\to \ul\Shv(X)\overset{(\Delta_{X/B})_*}\to
\uShv(X\underset{B}\times X) \simeq \\
\simeq \uShv(X) \otimes \uShv(X) \simeq \uShv(X)^\vee \otimes \uShv(X)\overset{(f^!)^\vee\otimes \on{Id}}\longrightarrow
\uShv(\CY)^\vee\otimes \uShv(X),
\end{multline*}
functorially in $(X,f)$.

\medskip

Using the commutative diagram
$$
\CD
\uShv(\CY)\otimes \uShv(X) @>{\boxtimes}>> \uShv(\CY\underset{B}\times X) \\
@A{\on{Id}\otimes f^!}AA @AA{(\on{id}\times f)^!}A \\
\uShv(\CY)\otimes \uShv(\CY) @>{\boxtimes}>> \uShv(\CY\underset{B}\times \CY),
\endCD
$$
and base change, we obtain that composition of \eqref{e:unit Y formula} with $\on{id}\otimes f^!$ is the map
\begin{multline*}
\uno_{\AGCat_{/B}}=\ul\Shv(B) \overset{(p_{X/B})^!}\to \ul\Shv(X)\overset{(\Delta_{X/B})_*}\to
\uShv(X\underset{B}\times X) \simeq \\
\simeq \uShv(X)\otimes \uShv(X)\overset{f_*\otimes \on{Id}}\longrightarrow \uShv(\CY)\otimes \uShv(X).
\end{multline*}

Hence, it remains to show that the functor
$$\uShv(X)\simeq \uShv(X)^\vee \overset{(f^!)^\vee}\to \uShv(\CY)^\vee$$
identifies with
$$\uShv(X) \overset{f_*}\to \uShv(\CY) \simeq \uShv(\CY)^\vee.$$

However, this follows by combining \eqref{e:u Omega *} and \secref{sss:dual co ushv}.

\end{proof}

\sssec{} \label{sss:black triangle}

For a (not necessarily) schematic map $f:\CY_1\to \CY_2$ between tame prestacks,
we introduce the functor
$$f_\blacktriangle:\Shv(\CY_1)\to \Shv(\CY_2)$$
to be the composition
$$\Shv(\CY_1) \overset{\Omega_{\CY_1}}\simeq
\Shv(\CY_1)_{\on{co}} \overset{f_{\on{co,*}}}\longrightarrow \Shv(\CY_2)_{\on{co}} \overset{\Omega_{\CY_2}}\simeq
\Shv(\CY_2).$$

Note that if $f$ is schematic, it follows from \eqref{e:Omega *} that we have
$$f_\blacktriangle\simeq f_*.$$

\medskip

For future use, we note:

\begin{lem} \label{l:base change renorm}
Let
$$
\CD
\CY'_1 @>{g_1}>> \CY_1 \\
@V{f'}VV @VV{f}V \\
\CY'_2 @>{g_2}>> \CY_2
\endCD
$$
be a Cartesian diagram of tame prestacks, where the horizontal maps are schematic. Then
there exists a canonical isomorphism
$$g_2^!\circ f_\blacktriangle\simeq f'_\blacktriangle \circ g_1^!.$$
\end{lem}

\begin{proof}

Follows from \eqref{e:Omega !}.

\end{proof}

\begin{cor} \label{c:proj formula}
Suppose that $\CY_1,\CY_2,\CY_1\times \CY_2,\CY_2\times \CY_2$ are tame.
Then the functor $f_\blacktriangle$ satisfies the projection formula:
$$f_\blacktriangle(\CF_1)\sotimes\CF_2\simeq f_\blacktriangle(\CF_1\sotimes f^!(\CF_2)).$$
\end{cor}

\begin{proof}

We apply \lemref{l:base change renorm} to
$$
\CD
\CY_1 @>{\on{Graph}_f}>> \CY_1\times \CY_2 \\
@V{f}VV @VV{f\times \on{id}}V  \\
\CY_2 @>{\Delta_{\CY_2}}>> \CY_2\times \CY_2.
\endCD
$$

\end{proof}

\sssec{}

Similarly, for a map $f:\CY_1\to \CY_2$ between \emph{universally} tame prestacks
we define the morphism
$$f_\blacktriangle:\uShv(\CY_1)\to \uShv(\CY_2)$$
to be the composition
$$\uShv(\CY_1) \overset{\ul\Omega_{\CY_1}}\simeq
\uShv(\CY_1)_{\on{co}} \overset{f_{\on{co,*}}}\longrightarrow \uShv(\CY_2)_{\on{co}} \overset{\ul\Omega_{\CY_2}}\simeq
\uShv(\CY_2).$$

The values of the latter map are given by $(f\times \on{id})_\blacktriangle$ for
$$f\times \on{id}:\CY_1\underset{B}\times X \to \CY_2\underset{B}\times X.$$

\sssec{}

Unwinding, from \propref{p:unit self Y} and \secref{sss:dual co ushv} we obtain:

\begin{cor} \label{c:dual of ! u}
For a map $f:\CY_1\to \CY_2$ between AG tame prestacks, with respect to
the self-dualities $\uShv(\CY_i)\simeq  \uShv(\CY_i)^\vee$, we have
$$(f^!)^\vee \simeq f_\blacktriangle.$$
\end{cor}

\sssec{}

As a formal consequence of \corref{c:dual of ! u} we obtain:

\begin{cor} \label{c:counit u}
The counit of the self-duality of $\uShv(\CY)$ is given by
$$\uShv(\CY)\otimes \uShv(\CY)\overset{\boxtimes}\to \uShv(\CY\underset{B}\times \CY)
\overset{(\Delta_{\CY/B})^!}\longrightarrow \ul\Shv(\CY) \overset{(p_{\CY/B})_\blacktriangle}\longrightarrow
\ul\Shv(B)=\uno_{\AGCat_{/B}}.$$
\end{cor}

\section{Duality in \texorpdfstring{$\DGCat$}{DGCat} and \texorpdfstring{$\AGCat$}{AGCat}} \label{s:duality}

In this section we study how monoidal dualilty in $\DGCat$ interacts with Verdier duality
on schemes and algebraic stacks.

\ssec{Ambidexterity} \label{ss:ambidex}

\sssec{} \label{ss:ker duality setup}

Recall the general setup in \Cref{ss:ker setup}: we have a closed symmetric monoidal category $\bO$ and a right-lax
symmetric monoidal functor
$$ F: \bO \to \CV .$$

Suppose now that $F$ satisfies the following additional hypotheses:

\medskip

\begin{enumerate}
\item\label{i:duality unit}
$F$ is strictly unital; i.e. the natural map
$$ \uno_{\CV} \to F(\uno_{\bO}) $$
is an isomorphism.

\smallskip

\item\label{i:duals in O}
Every object $\bo\in \bO$ is dualizable.

\medskip

\item\label{i:pairing after F}
For every $\bo\in \bO$, the pairing
$$ F(\bo) \otimes F(\bo^{\vee}) \to F(\bo \otimes \bo^{\vee}) \to F(\uno_{\bO}) \simeq \uno_{\CV} $$
is the counit of a duality pairing, where $\bo^{\vee}$ is the dual of $\bo$ and the right map is the image under $F$ of the duality pairing of $\bo$ and $\bo^{\vee}$.
\end{enumerate}

\sssec{}

Recall that we have the adjunction \Cref{e:adjunction for KER}
$$ \xymatrix{
\bi: \CV \ar@<0.8ex>[r] & \ar@<0.8ex>[l] \ul\CV : \be }
$$
of $\CV$-module categories.

\medskip

By \Cref{p:strictly unital inc ff}, condition (\ref{i:duality unit}) implies that the functor $\bi$ is fully faithful.  Thus,
the unit natural transformation
\begin{equation}\label{e:unit of alpha-rho}
\on{Id}_{\CV} \to \be \circ \bi
\end{equation}
is an isomorphism.

\medskip

By \propref{p:Enh is closed}, $\ul\CV$ is a closed symmetric monoidal category. We have:
\begin{equation}\label{e:two inner Homs}
\be(\uHom_{\ul\CV}(\ul{v}_1,\ul{v}_2)) \simeq \uHom_{\ul\CV,\CV}(\ul{v}_1,\ul{v}_2) \in \CV.
\end{equation}

\sssec{}\label{sss:reform of dual cond}

Assuming conditions (\ref{i:duality unit}) and (\ref{i:duals in O}), condition (\ref{i:pairing after F}) is equivalent to the following two conditions:

\medskip

\noindent{(3')} For every $\bo\in \bO$, the object $F(\bo)\in \CV$ is dualizable, and

\medskip

\noindent{(3'')} The map
$$\uHom_{\ul\CV,\CV}(\wtF(\bo), \uno_{\ul\CV}) \to \uHom_{\ul\CV,\CV}(\bi\circ \be(\wtF(\bo)), \uno_{\ul\CV}) \simeq
\uHom_{\ul\CV,\CV}(\bi(F(\bo)), \uno_{\ul\CV}) \simeq \uHom_{\CV}(F(\bo), \uno_\CV),$$
given by restriction along the counit map $\bi \circ \be(\wtF(\bo)) \to \wtF(\bo)$, is an isomorphism.

\sssec{}

We claim:

\begin{prop}\label{p:alpha rho ambi}
The adjunction $(\bi, \be)$ is ambidexterous, with the counit of the $(\be, \bi)$-adjunction given by the inverse
of \Cref{e:unit of alpha-rho}.
\end{prop}
\begin{proof}
We need to show that for every $v\in \CV$ and every $\ul{v} \in \ul\CV$, the natural map
\begin{equation}\label{e:ambidex check}
\uHom_{\ul\CV,\CV}(\ul{v}, \bi(v)) \to \uHom_{\CV}(\be(\ul{v}), \be\circ \bi(v)) \overset{\sim}{\to}
\uHom_{\CV}(\be(\ul{v}), v)
\end{equation}
is an isomorphism.

\medskip

Since both sides take weighted colimits in the $\ul{v}$-variable to weighted limits, by
\Cref{e:Yoneda weighted colim}, it suffices to show that
\Cref{e:ambidex check} is an isomorphism for $\ul{v}=\wtF(\bo)$ for $\bo\in \bO$.

\medskip

By definition, $\bi(v) \simeq \uno_{\ul\CV}\otimes v$.  Hence, we have a commutative diagram
$$ \xymatrix{
\uHom_{\ul\CV,\CV}(\wtF(\bo), \uno_{\ul\CV}) \otimes v \ar[r]\ar[d] & \uHom_{\ul\CV,\CV}(\wtF(\bo), \bi(v)) \ar[d]\\
\uHom_{\ul\CV,\CV}(\bi\circ \be(\wtF(\bo)), \uno_{\ul\CV}) \otimes v \ar[r] & \uHom_{\ul\CV,\CV}(\bi\circ \be(\wtF(\bo)), \bi(v))\ar[r]^-{\sim} &  \uHom_\CV(\be(\wtF(\bo)), v),
}$$
where the composite map
$$\uHom_{\ul\CV,\CV}(\wtF(\bo), \bi(v))  \to \uHom_{\ul\CV,\CV}(\bi\circ \be(\wtF(\bo)), \bi(v))\simeq  \uHom_\CV(\be(\wtF(\bo)), v)$$
is the map \eqref{e:ambidex check} for $\ul{v}=\wtF(\bo)$. Thus, we need to check that the right vertical map in this diagram is an isomorphism.

\medskip

In the above diagram, the left vertical map is an isomorphism by condition (3'') above; the top horizontal map is an isomorphism since $\wtF(\bo)$ is dualizable;
and the bottom horizontal map is an isomorphism because $\bi \circ \be (\wtF(\bo))\simeq \bi(F(\bo))$ is dualizable by condition (3') above.
Hence, the right vertical map is an isomorphism, as desired.

\end{proof}

\begin{rem}

In the specific situation of $\AGCat$, we will show in \Cref{c:counit has right adjoint AGCat} that more is true.  Namely, in this case the counit map
$$ \bi \circ \be \to \on{Id} $$
admits a right adjoint (which makes sense since $\AGCat$ is naturally a 2-category). In this situation (i.e., when $\be\circ \bi=\on{Id}$), this right adjoint
is automatically the unit of an $(\be, \bi)$-adjunction, which is automatically the same as the $(\be, \bi)$-adjunction described above.

\end{rem}

\ssec{Duality of the values} \label{ss:duality values}

We keep the assumptions of \Cref{ss:ker duality setup}. In this subsection, we will show that the functor
$$ \be: \ul\CV \to \CV $$
maps dualizable objects to dualizable objects, and we will describe the duality data explicitly.

\sssec{}

We are going to prove:

\begin{prop}\label{p:rho preserves duality}
Suppose that $\ul{v}\in \ul\CV$ is a dualizable object.  Then $\be(\ul{v})$ is also dualizable, and the paring
\begin{equation}\label{e:counit on rho}
 \be(\ul{v}) \otimes \be(\ul{v}^{\vee}) \to \be(\ul{v} \otimes \ul{v}^{\vee}) \to \be(\uno_{\CV}) \simeq \uno_{\CV}
\end{equation}
is the counit of a duality.
\end{prop}

\begin{proof}

First, we will show that $\be(\ul{v})$ is dualizable.  We need to show that for any $v\in \CV$, the natural map
$$ \uHom_\CV(\be(\ul{v}), \uno_\CV)\otimes v \to \uHom_\CV(\be(\ul{v}), v)$$
is an isomorphism. We have a commutative diagram
\begin{equation}\label{e:rho dual diagram}
\xymatrix{
\uHom_{\ul\CV,\CV}(\ul{v}, \uno_{\ul\CV})\otimes v \ar[r] \ar[d] & \uHom_{\ul\CV,\CV}(\ul{v}, \bi(v)) \ar[d]\\
\uHom_{\CV}(\be(\ul{v}), \uno_\CV) \otimes v \ar[r] & \uHom_{\CV}(\be(\ul{v}), v)
}
\end{equation}
The vertical maps is \Cref{e:rho dual diagram} are isomorphisms by \Cref{p:alpha rho ambi}.
Since $\ul{v}$ is dualizable, the map
$$ \uHom_{\ul\CV}(\ul{v}, \uno_{\ul\CV}) \otimes v \to \uHom_{\ul\CV}(\ul{v}, \bi(v))$$
is an isomorphism.  Applying the functor $\be$ and using \eqref{e:two inner Homs},
we obtain that the top horizontal map in \Cref{e:rho dual diagram} is an isomorphism.
Hence the bottom horizontal map is also an isomorphism, as desired.  Thus, the object $\be(\ul{v})$ is dualizable.

\medskip

It remains to show that \Cref{e:counit on rho} is the counit of a duality.  We need to show that the map
\begin{equation}\label{e:check duality}
\be(\ul{v}^{\vee}) \to \uHom_\CV(\be(\ul{v}), \uno_{\CV})
\end{equation}
induced by \Cref{e:counit on rho} is an isomorphism.  Unraveling the definition, the map \Cref{e:check duality} is the composite
$$ \be(\ul{v}^{\vee}) \simeq \be(\uHom_{\ul\CV}(\ul{v}, \uno_{\ul\CV})) \simeq
\uHom_{\ul\CV,\CV}(\ul{v}, \uno_{\ul\CV}) \to \uHom_\CV(\be(\ul{v}), \uno_\CV), $$
which is an isomorphism by \Cref{p:alpha rho ambi}.

\end{proof}

\sssec{}
As a consequence of \Cref{p:rho preserves duality}, we obtain that all values of a dualizable object are dualizable:

\begin{cor}\label{c:duality of values}
Let $\ul{v}\in \ul\CV$ be a dualizable object and $\bo\in \bO$.  Then $\ul{v}(\bo)\in \CV$ is dualizable with
dual $\ul{v}^{\vee}(\bo^{\vee})$ and the counit given by
$$ \ul{v}(\bo)\otimes \ul{v}^{\vee}(\bo^{\vee}) \to (\ul{v} \otimes \ul{v}^{\vee})(\bo \otimes \bo^{\vee}) \to (\ul{v}\otimes \ul{v}^{\vee})(\uno_{\bO}) \to
\uno_{\ul\CV}(\uno_{\bO}) \simeq \uno_{\CV}.$$
\end{cor}
\begin{proof}
We have
$$ \ul{v}(\bo) \simeq \uHom_{\ul\CV,\CV}(\wtF(\bo), \ul{v}) \simeq \be(\uHom_{\ul\CV}(\wtF(\bo), \ul{v}))\simeq
\be(\wtF(\bo^{\vee})\otimes \ul{v}) .$$

Now, $\wtF(\bo^{\vee}) \otimes \ul{v} \in \ul\CV$ is a dualizable object (since it's a tensor product of two dualizable objects).
Applying \Cref{p:rho preserves duality} to this object, we obtain the desired result.
\end{proof}

\ssec{Recollections on Verdier duality on schemes} \label{ss:recall Verdier}

\sssec{}

Let $X$ be a scheme.  In this case, by definition, we have
$$\Shv(X) = \on{Ind}(\Shv(X)^{\on{constr}}) $$
is the ind-category of constructible $\ell$-adic sheaves.  Now, Verdier duality gives an equivalence
\begin{equation} \label{e:Verdier sch}
\BD_X: (\Shv(X)^{\on{constr}})^{\on{op}} \simeq (\Shv(X)^{\on{constr}}).
\end{equation}

\sssec{}

Recall that if $\bC$ is a compactly generated DG category, then it is dualizable and the dual is given by
\begin{equation} \label{e:dual of comp gen}
\bC^{\vee} \simeq \on{Ind}((\bC^{c})^{\on{op}}) ,
\end{equation}
where $\bC^{c} \subset \bC$ is the full subcategory consisting of compact objects.

\medskip

Moreover, suppose that
$$ F: \bC_1 \to \bC_2 $$
is a functor between compactly generated categories that preserves compact objects; i.e. $F = \on{Ind}(F^c)$
for a DG functor
$$ F^c: \bC_1^c \to \bC_2^c .$$
Then the dual $F^{\vee}$ of $F$ identifies with
$$\bC_2^{\vee} \simeq \on{Ind}((\bC_2^c)^{\on{op}}) \overset{\on{Ind}((F^c)^{\on{op}})^R}\longrightarrow
\on{Ind}((\bC_1^c)^{\on{op}})\simeq \bC_1^\vee,$$
i.e., the right adjoint of the ind-extension of $(F^c)^{\on{op}}: \bC_1^{\on{op}} \to \bC_2^{\on{op}}$.

\sssec{} \label{sss:Verdier schemes}

Applying this to the category $\Shv(X)$, we see that the Verdier duality functor gives rise to an identification
\begin{equation}\label{e:Verdier as self-duality}
\Shv(X)^\vee \simeq \Shv(X).
\end{equation}

\sssec{} \label{sss:Verdier unit}

Recall also that for $\CF_1,\CF_2\in \Shv(X)^{\on{constr}}$, we have
$$\CHom_{\Shv(X)}(\CF_1,\CF_2) \simeq \on{C}^\cdot(X,\BD_X(\CF_1)\sotimes \CF_2).$$

This implies that the pairing
$$\Shv(X)\otimes \Shv(X)\to \Vect$$
corresponding to \eqref{e:Verdier as self-duality} is given by
$$\Shv(X)\otimes \Shv(X)\overset{\boxtimes}\to \Shv(X\times X)\overset{\Delta_X^!}\to \Shv(X)
\overset{(p_X)_*}\to \Shv(\on{pt})\simeq \Vect.$$

\sssec{}

The above formula for the pairing implies that for a map of schemes $f: X_1 \to X_2$, we have an identification of functors on the
subcategories of compact objects
$$ f^! \circ \BD_{X_2} \simeq \BD_{X_1} \circ f^* $$
and therefore with respect to the self-duality \Cref{e:Verdier as self-duality}, we have
\begin{equation} \label{e:dual !}
(f^!)^{\vee} \simeq (f^*)^R \simeq f_* .
\end{equation}

\sssec{}

Let $\CY$ be a prestack, and recall the category $\Shv(\CY)_{\on{co}}$. Since the transition functors
in the colimit preserve compactness, the category $\Shv(\CY)_{\on{co}}$ is compactly generated, and
hence dualizable.

\medskip

From \eqref{e:dual !} it follows that we have a canonical identification
\begin{equation} \label{e:dual of co}
(\Shv(\CY)_{\on{co}})^\vee \simeq \Shv(\CY).
\end{equation}

Under this identification, for a map of prestacks $f:\CY_1\to \CY_2$, we have
\begin{equation} \label{e:dual of co ! new}
f^!\simeq (f_{\on{co},*})^\vee,
\end{equation}
and if $f$ is schematic
\begin{equation} \label{e:dual of co * new}
f_*\simeq (f^!_{\on{co}})^\vee.
\end{equation}

\begin{rem}

Note that formulas \eqref{e:dual !}, \eqref{e:dual of co ! new} and \eqref{e:dual of co * new} look
similar to those in Sects. \ref{sss:dual of maps} and \ref{sss:dual co ushv}. However, for
now, these are ``apples and oranges": the former talk about duality in $\DGCat$ and the
latter about duality in $\AGCat$. The two will be related in the next subsection, see \corref{c:counit values}.

\end{rem}

\ssec{Duality in $\AGCat$}

We now specialize the discussion in Sects. \ref{ss:ambidex} and \ref{ss:duality values} to
$$\Shv:\on{Corr}(\Sch)\to \DGCat,$$
i.e., to the setting of \secref{s:AGCat} with $B=\on{pt}$.

\medskip

The latter assumption is needed in order for $\Vect\to \Shv(B)$ to be
an equivalence.  As was mentioned already, in this case we write $\AGCat$ instead of $\AGCat_{/B}$.

\sssec{}

We claim that conditions (1)-(3) in \secref{ss:ker duality setup} are satisfied. Indeed, condition (1)
is satisfied since $B=\on{pt}$.
Condition (2) is satisfied because every object in $\on{Corr}(\on{Sch})$ is self-dual
(see \secref{sss:dual corr}). Condition (3) is satisfied by \secref{sss:Verdier unit}.

\medskip

Thus, the contents of Sects. \ref{ss:ambidex}-\ref{ss:duality values} are applicable.

\sssec{}

Applying \Cref{c:duality of values} to $\AGCat$, we obtain:

\begin{cor} \label{c:counit values}
Let $\ul{\bC} \in \AGCat$ be a dualizable object.  Then for any scheme $X$,
$$\ul{\bC}(X) \in \DGCat $$
is dualizable, with dual $\ul{\bC}^{\vee}(X)$ and counit pairing
\begin{multline*}
\ul{\bC}(X)\otimes \ul{\bC}^{\vee}(X) \to (\ul{\bC}\otimes \ul{\bC}^{\vee})(X \times X)
\overset{\Delta_X^!}\to (\ul{\bC}\otimes \ul{\bC}^{\vee})(X) \overset{(p_X)_*}\to
(\ul{\bC}\otimes \ul{\bC}^{\vee})(\on{pt})\to \uno_{\AGCat}(\on{pt})\simeq
\on{Vect}.
\end{multline*}

In particular, $\ul\bC(\on{pt})$ is dualizable with counit given by
$$\ul{\bC}(\on{pt})\otimes \ul{\bC}^{\vee}(\on{pt}) \to (\ul{\bC}\otimes \ul{\bC}^{\vee})(\on{pt})\to
\uno_{\AGCat}(\on{pt})\simeq
\on{Vect} .$$
\end{cor}

\sssec{}

Let $\CY$ be an AG tame prestack. From Corollaries \ref{c:counit values} and \ref{c:counit u}, we obtain:

\begin{cor} \label{c:counit AG tame value}
The category $\Shv(\CY)$ is self-dual with the counit given by
\begin{equation} \label{e:counit AG tame value}
\Shv(\CY)\otimes \Shv(\CY)\overset{\boxtimes}\to \Shv(\CY\times \CY)
\overset{\Delta_\CY^!}\to \Shv(\CY) \overset{(p_\CY)_\blacktriangle}\to \Vect.
\end{equation}
\end{cor}

\sssec{}

Unwinding, we obtain that the self-duality
$$\Shv(\CY)^\vee\simeq \Shv(\CY)$$
given by \corref{c:counit AG tame value} equals
\begin{equation} \label{e:Verdier tame prel}
\Shv(\CY)^\vee \overset{\Omega_\CY^\vee}\simeq (\Shv(\CY)_{\on{co}})^\vee \overset{\text{\eqref{e:dual of co}}}\simeq \Shv(\CY).
\end{equation}

We will denote the resulting equivalence
\begin{equation} \label{e:Verdier tame}
(\Shv(\CY)^c)^{\on{op}}\simeq \Shv(\CY)^c
\end{equation}
by $\BD_\CY$, see \eqref{e:dual of comp gen}.

\medskip

Note that by construction, when $\CY=X$ is a scheme, this is the same functor as \eqref{e:Verdier sch}.

\begin{rem}

Note that the identification \eqref{e:Verdier tame prel} and the equivalence \eqref{e:Verdier tame} take
place for any tame prestack (it does not need to be AG or universally tame).

\medskip

For future reference, we note that if $\CY_1$, $\CY_2$ and $\CY_1\times \CY_2$ are tame, and $\CF_i\in \Shv(\CY_i)^c$, we have
a canonical identification
\begin{equation} \label{e:Verdier ext}
\BD_{\CY_1}(\CF_1)\boxtimes \BD_{\CY_2}(\CF_2)\simeq \BD_{\CY_1\times \CY_2}(\CF_1\boxtimes \CF_2).
\end{equation}

\end{rem}

\ssec{Ambidexterity in $\AGCat$}

\sssec{}

Recall the adjunction 
\begin{equation} \label{e:DGCta to AGCat pt}
\xymatrix{
\bi: \DGCat \ar@<0.8ex>[r] & \ar@<0.8ex>[l] \AGCat : \be }
\end{equation}
of \eqref{e:DGCta to AGCat}.

\medskip

The counit is a natural transformation
\begin{equation}\label{e:conunit of alpha rho adjunction}
\bi \circ \be \to \on{Id}_{\AGCat}
\end{equation}
of $\DGCat$-module functors $\AGCat \to \AGCat$.

\sssec{}

We wll prove:

\begin{prop} \label{p:adjoint of counit in DGCat-ag}
The counit natural transformation \Cref{e:conunit of alpha rho adjunction} admits a right adjoint in the 2-category
$\on{Funct}_{\DGCat}(\AGCat, \AGCat)$.
\end{prop}

Let us explain what \propref{p:adjoint of counit in DGCat-ag} says in concrete terms:

\medskip

First, for any $\ul\bC\in \AGCat$
and $X\in \Sch$, the functor
$$\Shv(X)\otimes \ul\bC(\on{pt})\to \ul\bC(X)$$
admits a right adjoint.

\medskip

Second (see \propref{p:adj AGCat1}), for every $\CF\in \Shv(X'\times X'')$, the diagram
$$
\CD
\Shv(X')\otimes \ul\bC(\on{pt}) @>>> \ul\bC(X') \\
@V{\CF\otimes \on{Id}}VV @VV{\ul\bC(\CF)}V \\
\Shv(X'')\otimes \ul\bC(\on{pt}) @>>> \ul\bC(X'')
\endCD
$$
satisfies the right Beck-Chevalley condition.

\medskip

And third (see \propref{p:adj in dgcat ag}), for a map $\alpha:\ul\bC_1\to \ul\bC_2$ and $X\in \Sch$, the
diagram
$$
\CD
\Shv(X)\otimes \ul\bC_1(\on{pt}) @>>> \ul\bC_1(X) \\
@V{\on{Id}\otimes \alpha(\on{pt})}VV @VV{\alpha(X)}V \\
\Shv(X)\otimes \ul\bC_2(\on{pt}) @>>> \ul\bC_2(X)
\endCD
$$
satisfies the right Beck-Chevalley condition.

\sssec{Proof of \propref{p:adjoint of counit in DGCat-ag}}

The assignment
$$ \ul\bC \mapsto (\bi(\ul\bC(\on{pt})) \to \ul\bC) $$
can be viewed as a $\DGCat$-module functor

\begin{equation}\label{e:curry counit restr}
\AGCat \to \on{Funct}([1], \AGCat) .
\end{equation}

The assertion of the proposition is equivalent to the assertion that the functor \eqref{e:curry counit restr}
factors through $\on{Funct}(\on{Adj}, \AGCat)$.

\medskip

By \propref{p:env enh}, we have
\begin{multline}\label{e:AGCat adj}
 \on{Funct}_{\DGCat}(\AGCat, \on{Funct}([1], \AGCat)) \simeq \\
\simeq \on{Funct}^{\DGCat}(\on{enh}(\on{Corr}(\Sch),\DGCat), \on{Funct}([1], \AGCat)),
\end{multline}
and we need to show that the object corresponding to \Cref{e:curry counit restr} in the right-hand
side of \eqref{e:AGCat adj}
lies in the 1-full subcategory
\begin{multline}
\on{Funct}^{\DGCat}(\on{enh}(\on{Corr}(\Sch),\DGCat), \on{Funct}(\on{Adj}, \AGCat))\subset \\
\subset \on{Funct}^{\DGCat}(\on{enh}(\on{Corr}(\Sch),\DGCat), \on{Funct}([1], \AGCat)).
\end{multline}

\medskip

By \Cref{p:adj in dgcat ag} and \Cref{p:1-full functors}, we need to show the following:

\medskip

\begin{enumerate}

\item
For every scheme $Z$ and another scheme $X$, the functor
\begin{equation} \label{e:boxtimes XZ}
\Shv(X)\otimes \Shv(Z) \overset{\boxtimes}\to \Shv(X\times Z)
\end{equation}
admits a right adjoint.

\medskip

\item
For $\CF\in \Shv(X'\times X'')$, the diagram
$$
\CD
\Shv(X')\otimes \Shv(Z) @>>> \Shv(X'\times Z) \\
@V{[\CF]\otimes \on{Id}}VV @VV{[\CF]}V \\
\Shv(X'')\otimes \Shv(Z) @>>> \Shv(X''\times Z)
\endCD
$$
satisfies the right Beck-Chevalley condition, where the notation $[-]$ is as in \secref{sss:conv}.

\medskip

\item
For $\CG\in \Shv(Z_1\times Z_2)$, the diagram
$$
\CD
\Shv(X)\otimes \Shv(Z_1) @>>> \Shv(X\times Z_1) \\
@V{\on{Id}\otimes [\CG]}VV @VV{[\CG]}V \\
\Shv(X)\otimes \Shv(Z_2) @>>> \Shv(X\times Z_2)
\endCD
$$
satisfies the right Beck-Chevalley condition.

\end{enumerate}

Note, however, that conditions (2) and (3) amount to the same, up to swapping
$X$ and $Z$.

\medskip

Condition (1) follows from the fact that the functor \eqref{e:boxtimes XZ}
maps compact objects to compact objects.

\medskip

To prove condition (2), we can assume that $\CF$ is compact (i.e., constructible). In this case,
the functor $[\CF]$ admits a left adjoint, which is explicitly given by
$$\CF''\mapsto (p')_!(\BD_{X'\times X''}(\CF)\otimes (p'')^*(\CF'')),$$
where $p',p''$ are the projections from $X'\times X''$ to $X'$ and $X''$, respectively.

\medskip

Now condition (2) follows from the fact that the rotated diagram
$$
\CD
\Shv(X')\otimes \Shv(Z) @>{[\CF]\otimes \on{Id}}>> \Shv(X'')\otimes \Shv(Z)  \\
@VVV @VVV \\
\Shv(X'\times Z) @>{\CF}>>  \Shv(X''\times Z)
\endCD
$$
satisfies the \emph{left} Beck-Chevalley condition.

\qed[\propref{p:adjoint of counit in DGCat-ag}]

\sssec{}

Recall from \Cref{p:alpha rho ambi} that the adjunction \Cref{e:DGCta to AGCat pt} is canonically ambidexterous.
A consequence of \Cref{p:adjoint of counit in DGCat-ag} is a more precise description of the resulting
$(\be,\bi)$-adjunction:

\begin{cor}\label{c:counit has right adjoint AGCat}
The unit of the $(\be, \bi)$-adjunction
$$\on{Id}_{\AGCat} \to \bi \circ \be $$
is canonically isomorphic to the right adjoint of the counit of the adjunction \Cref{e:DGCta to AGCat pt}.
\end{cor}

\begin{proof}

Both the unit and the counit of \Cref{e:DGCta to AGCat pt} admit right adjoints: the unit because it's
an isomorphism and the counit by \Cref{p:adjoint of counit in DGCat-ag}.  Passing to right adjoints gives an adjunction
in the other direction.

\medskip

By construction, the counit of this new $(\be, \bi)$-adjunction equals the counit from \Cref{p:alpha rho ambi}:
both are given by the inverse of $\on{Id}\to \be\circ \bi$.

\end{proof}

\ssec{The notion of Verdier compatible algebraic stack}

In this subsection we let $\CY$ be an algebraic stack with a schematic diagonal.

\sssec{}

Note that we have an inclusion:
$$\Shv(\CY)^{c} \subset \Shv(\CY)^{\on{constr}}.$$

I.e., compact objects are constructible (but, in general, not vice versa).  Recall from \cite[Sect. F.2.6]{AGKRRV1} that $\CY$
is called \emph{Verdier-compatible} if the subcategory $\Shv(\CY)^c$ is preserved by Verdier duality functor\footnote{We are insert the
superscript $\on{Verdier}$ in the notation in order to avoid conflating it with the functor $\BD_\CY$ in \eqref{e:Verdier tame}. See
Remark \ref{r:Verdier funct} for the
relation between the two functors.}
$$\BD^{\on{Verdier}}_\CY:(\Shv(\CY)^{\on{constr}})^{\on{op}}\simeq \Shv(\CY)^{\on{constr}}.$$

\medskip

By \cite[Theorem F.2.8]{AGKRRV1}, all quasi-compact stacks of finite type that are locally quotient stacks are Verdier-compatible.
According to \cite[Conjecture F.2.7]{AGKRRV1}, all quasi-compact algebraic stack with a schematic diagonal are Verdier-compatible.

\sssec{}


\medskip

By \eqref{e:dual of comp gen}, we obtain that for a Verdier-compatible algebraic stack there is a canonical equivalence
\begin{equation}\label{e:Verdier as self-duality stack}
\Shv(\CY)^\vee \simeq \Shv(\CY)
\end{equation}

\sssec{} \label{sss:Verdier compat}

Since $\CY$ has a schematic diagonal, we have a well-defined functor
\begin{equation} \label{e:Omega again}
\Omega_\CY:\Shv(\CY)_{\on{co}}\to \Shv(\CY),
\end{equation}
see \eqref{e:Omega}. We claim:

\begin{prop}\label{p:Verdier-compatible in terms of colim}
Let $\CY$ be an algebraic stack.  $\CY$ is Verdier-compatible if and only if it is tame.
\end{prop}

\begin{proof}
The transition functors in the colimit $$\Shv(\CY)_{\on{co}}=\underset{S\in \affSch_{/\CY}}{\on{colim}}\  \Shv(S)$$
preserve compact objects.  In particular, the compact objects in the colimit
are generated by objects of the form
$$ \on{ins}_f(\mathcal{F}) \in \underset{S\in \affSch_{/\CY}}{\on{colim}} \  \Shv(S), $$
for a map $f: S \to \CY$, $\CF\in \Shv(S)^{\on{constr}}$, and where
$$ \on{ins}_f: \Shv(S) \to \underset{S\in \affSch_{/\CY}}{\on{colim}} \  \Shv(S)$$
is the tautological functor.

\medskip

By definition, $\CY$ is Verdier-compatible if the Verdier dual of every compact object of $\Shv(\CY)$ is compact.
The compact objects in $\Shv(\CY)$ are generated by objects of the form
$f_!(\mathcal{F})$ for a map $f: S\to \CY$ and $\mathcal{F}\in \Shv(S)^{\on{constr}}$.  The Verdier dual of $f_!(\mathcal{F})$ is given by
$$ f_*(\mathbb{D}_S(\mathcal{F})) \simeq \Omega_\CY\Bigl(\on{ins}_f(\mathbb{D}_\CY(\mathcal{F}))\Bigr) .$$
In particular, if \Cref{e:Omega} is an equivalence, the objects $f_*(\mathbb{D}(\mathcal{F}))$
are compact and hence $\CY$ is Verdier-compatible.

\medskip

Now, suppose that $\CY$ is Verdier-compatible. In order to show that \eqref{e:Omega} is an equivalence,
it suffices to show that its dual functor is an equivalence. However, it is easy to see that the composition
$$\Shv(\CY) \overset{\text{\eqref{e:Verdier as self-duality stack}}}\simeq
\Shv(\CY)^\vee \overset{\Omega_\CY^\vee}\to \Shv(\CY)^\vee_{\on{co}}
\overset{\text{\eqref{e:dual of co}}}\simeq \Shv(\CY)$$
is the identity functor.

\end{proof}

\begin{rem} \label{r:Verdier funct}

Note that in the process of proof of \propref{p:Verdier-compatible in terms of colim}, we have shown that for
a Verider-compatible algebaric stack, the equivalence \eqref{e:Verdier as self-duality stack} equals
\begin{equation}\label{e:Verdier alg stack big}
\Shv(\CY)^\vee \overset{\text{\eqref{e:dual of co}}}\simeq \Shv(\CY)_{\on{co}} \overset{\Omega_\CY}\to \Shv(\CY).
\end{equation}

In particular, the equivalence
$$(\Shv(\CY)^c)^{\on{op}}\simeq \Shv(\CY)^c.$$
induced by $\BD^{\on{Verdier}}_\CY$ equals $\BD_\CY$ from \eqref{e:Verdier tame}.

\end{rem}

\begin{rem}

Let $f:\CY_1\to \CY_2$ be a map between Verdier-compatible algebraic stacks. Recall the functor
$$f_\blacktriangle:\Shv(\CY_1)\to \Shv(\CY_2),$$
see \secref{sss:black triangle}.

\medskip

It is not difficult to show that $f_\blacktriangle$ is the ind-extension of the functor
$$\Shv(\CY_1)^c\to \Shv(\CY_2)^c,$$
obtained by restriction from the \emph{usual} $f_*$, defined to be the (a priori, non-colimit preserving)
right adjoint of the functor $f^*$, see Remark \ref{r:comp gen prestack}.

\medskip

Thus, $f_\blacktriangle$ has the same meaning as in \cite[Sect. 1.2]{GaVa} and \cite[Sect. A.2.3]{AGKRRV2}.

\end{rem}

\begin{rem}

We note also that for a map $f:\CY_1\to \CY_2$ between Verdier-compatible algebraic stacks,
we have a natural transformation
\begin{equation} \label{e:triangle to star}
f_\blacktriangle \to f_*.
\end{equation}

Namely, \eqref{e:triangle to star} is obtained by adjunction from the natural transformation
\begin{equation} \label{e:triangle to star adj}
f^*\circ f_\blacktriangle \to \on{Id},
\end{equation}
which in its turn described as follows:

\medskip

For any $(X,g)\in \Sch_{/\CY_1}$, the precomposition of \eqref{e:triangle to star adj} with $g_*$ is the
natural transformation
\begin{multline*}
f^*\circ f_\blacktriangle \circ g_*\overset{g\text{ is schematic}}\simeq
f^*\circ f_\blacktriangle \circ g_\blacktriangle\simeq f^*\circ (f\circ g)_\blacktriangle\overset{f\circ g\text{ is schematic}}\simeq
f^*\circ (f\circ g)_*\overset{\text{unit}}\to \\
\to g_*\circ g^*\circ f^*\circ (f\circ g)_*\simeq
g_*\circ (f\circ g)^*\circ (f\circ g)_*\overset{\text{counit}}\to g_*.
\end{multline*}

\end{rem}

\sssec{}

We give the following definitions:

\medskip

\begin{itemize}
\item
$\CY$ is \emph{universally Verdier-compatible} if $\CY \underset{B}\times X$ is Verdier-compatible for every scheme $X$;

\smallskip

\item
$\CY$ is \emph{AG Verdier-compatible} if both $\CY$ and $\CY \underset{B}\times \CY$ are universally Verdier-compatible.

\end{itemize}

\medskip

From \propref{p:Verdier-compatible in terms of colim} we obtain:

\begin{cor} \label{c:colimit pres of ushv} \hfill

\smallskip

\noindent{\em(a)} $\CY$ is universally Verdier-compatible if and only if it is universally tame.

\smallskip

\noindent{\em(b)}
$\CY$ is AG Verdier-compatible if and only if it is AG tame.

\end{cor}

From \Cref{p:AG verdier products}, we obtain:

\begin{cor}
If $\CY_1$ and $\CY_2$ are AG Verdier compatible, then so is $\CY_1\underset{B}\times \CY_2$.
\end{cor}

\begin{rem} \label{r:counit Y}

Let $\CY$ be AG Verdier-compatible. It follows from Remark \ref{r:Verdier funct} and \corref{c:counit AG tame value}
that the counit of the self-duality \eqref{e:Verdier as self-duality stack} is given by formula
\eqref{e:counit AG tame value}.

\medskip

However, one can show that this formula for the counit is valid for \emph{any} Verdier compatible
prestack, see  \cite[Sect. 9]{DrGa}.

\end{rem}

\section{Computing trace in \texorpdfstring{$\AGCat$}{AGCat}} \label{s:Trace}

In this subsection we perform the trace calculation, deriving formula \eqref{e:Tr calc},
and relate it to Grothendieck's sheaves-functions correspondence.

\ssec{The Grothendieck sheaves-functions correspondence}   \label{ss:less Frob}

In this subsection we let $\CY$ be a Verdier-compatible algebraic stack over $\ol\BF_q$, defined over $\BF_q$.

\sssec{} \label{sss:over F_q}

Let $$\on{Frob}_\CY: \CY \to \CY $$
denote the geometric Frobenius morphism. (We will simply write $\Frob$ if no confusion is likely to occur.)

\medskip

Suppose that $\CF \in \on{Shv}(\CY)^c\subset  \on{Shv}(\CY)^{\on{constr}}$ is an object equipped with a (weak) Weil structure, i.e. a map
$$\Frob^*(\CF)\to \CF,$$
or, equivalently, a map
$$\alpha: \CF \to \on{Frob}_*(\CF) .$$

\medskip

Given this data, the sheaves-functions correspondence associates to $\CF$ a function
$$\sfunct(\CF, \alpha) \in \sFunct(\CY(\mathbb{F}_q), \ol\BQ_\ell) $$
on the set of $\BF_q$-points of $\CY$.

\medskip

Given a point
$$y \in \CY(\BF_q),$$
the value of $\sfunct(\CF, \alpha)$ at $y$ is the trace of Frobenius acting on the *-stalk $\CF_y$.

\sssec{}

Let us rewrite the above construction in a more categorical language. First, we note:

\begin{lem} \label{l:Frob triangle}
There is a canonical isomorphism
$$\Frob_*\simeq \Frob_\blacktriangle;$$
in particular, the functor $\Frob_*$ is a map in $\DGCat$.
\end{lem}

\begin{proof}

Recall that for a map between (Verdier-compatible) algebraic stacks we have a natural transformation
$$(-)_\blacktriangle \to (-)_*,$$
see \eqref{e:triangle to star}. We claim that this map is an isomorphism for the geometric Frobenius map.

\medskip

Recall that the two functors coincide on compact objects, while $(-)_\blacktriangle$ preserves
colimits. Hence, the natural transformation between them is an isomorphism if and only if
$(-)_*$ also preserves colimits.

\medskip

We claim that this is the case for the Frobenius map. In fact, we claim that $\Frob_*$
is an equivalence. Since $\Frob_*$ is by definition the right adjoint of $\Frob^*$,
it suffices to show that the latter is an equivalence.\footnote{Here is an alternative argument:
we can write $\Shv(\CY)$ as the limit of $\Shv(X)$ under *-pullbacks for schemes $X$ mapping to $\CY$.
Since for each $\Frob_X^*$ is an equivalence, so is $\Frob_\CY^*$.}

\medskip

To prove this, it suffices to show that for any scheme $X$ mapping to $\CY$, pullback with respect to the map
$$\Frob_{\CY/X}: X\underset{\CY,\Frob_\CY}\times \CY\to X$$
induces an equivalence on $\Shv(-)$.

\medskip

Note that the map $\Frob_X:X\to X$ factors as
$$X\to X\underset{\CY,\Frob}\times \CY \overset{\Frob_{\CY/X}}\to X,$$
with the first arrow being radicial, and hence pullback with respect to it
is an equivalence.

\medskip

This implies that $\Frob_{\CY/X}^*$ is an equivalence by the 2-out-3 property, since
$\Frob_X^*$ is an equivalence.

\end{proof}

\sssec{}

Thanks to \lemref{l:Frob triangle}, it makes sense to consider the object
$$\on{Tr}(\on{Frob}_*, \Shv(\CY))\in \Vect.$$

\medskip

Furthermore, the pair $(\CF, \alpha)$ defines a class
$$\on{cl}(\CF, \alpha) \in \on{Tr}(\on{Frob}_*, \Shv(\CY)),$$
see \cite[Sect. 3.4.3]{GKRV}.

\sssec{}

We claim that there is a natural map, called ``naive local term''
\begin{equation}\label{e:def LT naive}
\on{LT}^{\on{naive}}: \on{Tr}(\on{Frob}_*, \Shv(\CY)) \to \sFunct(\CY(\BF_q), \ol\BQ_\ell)
\end{equation}
that can be described as follows:

\medskip

For every $y \in \CY(\ol\BF_q)$, let
$$ i_y: \on{pt} \to \CY$$
denote the inclusion of $y$. Consider the functor
$$i_y^* : \Shv(\CY) \to \on{Shv}(\on{pt})=\Vect.$$

\medskip

If $y\in \CY(\BF_q)$, we have
$$(i_y)^*\circ \Frob_*\simeq (i_y)^*.$$

Hence, by the functoriality of the categorical trace (see \cite[Sect. 3.4.1]{GKRV}), we obtain a map
$$\on{Tr}(\Frob,i_y^*):\on{Tr}(\on{Frob}_*, \Shv(\CY)) \to \on{Tr}(\on{Id},\Vect) \simeq \ol\BQ_\ell.$$

The map $\on{LT}^{\on{naive}}$ is defined so that the composition
$$\on{Tr}(\on{Frob}_*, \Shv(\CY)) \to \sFunct(\CY(\BF_q), \ol\BQ_\ell) \overset{\on{ev}_y}\to \ol\BQ_\ell$$
is the above map $\on{Tr}(\Frob,i_y^*)$.

\medskip

Note, however, that the local term map \Cref{e:def LT naive} is typically \emph{not} an isomorphism.

\sssec{}

Unwinding, we obtain that
\begin{equation} \label{e:faisceaux-fonctions}
\sfunct(\CF, \alpha)=\on{LT}^{\on{naive}}(\on{cl}(\CF, \alpha)).
\end{equation}

\sssec{}

In this section we will show that the Grothendieck sheaves-functions correspondence and the local term map have natural
interpretations via $\AGCat$.

\medskip

Specifically, we will show that
\begin{equation} \label{e:LT AGCat}
\on{Tr}(\on{Frob}_*, \uShv(\CY)) \simeq \sFunct(\CY(\BF_q), \ol\BQ_\ell),
\end{equation}
where the trace in the left-hand side is taking place in the symmetric monoidal 2-category $\AGCat$
and so is an object of
$$\Maps_{\AGCat}(\uno_{\AGCat},\uno_{\AGCat})\simeq \Vect.$$

\medskip

The isomorphism \eqref{e:LT AGCat}, which we will denote by $\on{LT}^{\on{AG}}$, ``corrects" the failure of
$\on{LT}^{\on{naive}}$ to be an isomorphism. That said, the map $\on{LT}^{\on{naive}}$ also
has a natural interpretation via $\AGCat$, see \thmref{t:LT Frob}.

\ssec{Traces in $\AGCat$}

In this subsection, we let $\CY$ be an AG tame prestack.

\sssec{}

Let $\CF$ be an object of $\Shv(\CY\times \CY)$. We can think of $\CF$ as a map
$$\uno_{\AGCat}\to \uShv(\CY\times \CY)\overset{\text{\corref{c:hom verd comp}}}\simeq
\uHom_{\AGCat}(\uShv(\CY),\uShv(\CY)),$$
i.e., an object in $\Maps_{\AGCat}(\uShv(\CY),\uShv(\CY))$, which we will denote by $\ul{[\CF]}$.

\medskip

Since $\uShv(\CY)$ is dualizable, it makes sense to consider
\begin{equation} \label{e:Tr sheaf}
\Tr(\ul{[\CF]},\uShv(\CY))\in \Maps_{\AGCat}(\uno_{\AGCat},\uno_{\AGCat})\simeq \Vect.
\end{equation}

\sssec{} \label{sss:ren cochain}

We claim:

\begin{prop} \label{p:Tr sheaf}
The object \eqref{e:Tr sheaf} identifies canonically with
$$\on{C}^\cdot_\blacktriangle(\CY,\Delta_\CY^!(\CF)).$$
\end{prop}

In the above formula and elsewhere,
$$\on{C}^\cdot_\blacktriangle(\CY,-):=(p_\CY)_\blacktriangle.$$

\begin{proof}[Proof of \propref{p:Tr sheaf}]

By definition, $\Tr(\ul{[\CF]},\uShv(\CY))$ is the composition
$$\uno_{\AGCat} \overset{\on{unit}}\to \uShv(\CY)\otimes \uShv(\CY) \overset{\on{id}\otimes \ul{[\CF]}}\longrightarrow
\uShv(\CY)\otimes \uShv(\CY) \overset{\on{counit}}\to \uno_{\AGCat}.$$

Using \corref{c:tensor product verd comp}, we identify
\begin{equation} \label{e:sq Y}
\uShv(\CY)\otimes \uShv(\CY) \simeq \uShv(\CY\times \CY),
\end{equation}
and by definition, under this identification, the composition
$$\uno_{\AGCat} \overset{\on{unit}}\to \uShv(\CY)\otimes \uShv(\CY) \overset{\on{id}\otimes \ul{[\CF]}}\longrightarrow
\uShv(\CY)\otimes \uShv(\CY)\simeq  \uShv(\CY\times \CY),$$
viewed as an object of
$$\Maps(\uno_{\AGCat},  \uShv(\CY\times \CY))\simeq \uShv(\CY\times \CY)(\on{pt})\simeq
\Shv(\CY\times \CY),$$
equals $\CF$.

\medskip

Thus, it suffices to show that the map
$$\uShv(\CY\times \CY)\overset{\text{\eqref{e:sq Y}}}\simeq
\uShv(\CY)\otimes \uShv(\CY) \overset{\on{counit}}\to \uno_{\AGCat},$$
evaluated on $\on{pt}$ equals the map
$$\Shv(\CY\times \CY)\overset{\Delta_\CY^!}\to \Shv(\CY)\overset{\on{C}^\cdot_\blacktriangle(\CY,-)}\longrightarrow \Vect.$$

However, this is implied by \corref{c:counit u}.

\end{proof}

\sssec{}

We will denote the isomorphism established in the above proposition by
$$\on{LT}^{\on{AG}}:\Tr(\ul{[\CF]},\uShv(\CY)) \simeq \on{C}^\cdot_\blacktriangle(\CY,\Delta_\CY^!(\CF)).$$

\sssec{} \label{sss:corr microcosm}

We now consider a particular case of the above situation:

\medskip

Let
\begin{equation} \label{e:hom corr}
\xymatrix{
\CY & \ar[l]_{c_l} \CZ \ar[r]^{c_r} & \CY
}
\end{equation}
be a correspondence, where both $\CZ$ and
$$ \on{Fix}(\CZ) := \CZ \underset{\CY \times \CY}{\times} \CY $$
are also AG tame.

\medskip

Consider the object
$$(c_l,c_r)_{\blacktriangle}(\omega_\CZ) \in \Shv(\CX \times \CX).$$
defines a morphism. Denote by $$\ul{[\CZ]}:=[(c_l,c_r)_{\blacktriangle}(\omega_\CZ)]$$ the corresponding
object in $\Maps_{\AGCat}(\uShv(\CY),\uShv(\CY))$.

\sssec{}

As a special case of \propref{p:Tr sheaf}, and using \lemref{l:base change renorm},
we obtain:

\begin{cor} \label{c:Tr corr}
The object
$$\Tr(\ul{[\CZ]},\uShv(\CY))\in \Maps_{\AGCat}(\uno_{\AGCat},\uno_{\AGCat})\simeq \Vect$$
identifies canonically with
$$\on{C}^{\cdot}_{\blacktriangle}(\on{Fix}(C), \omega_{\on{Fix}(C)}).$$
\end{cor}

\ssec{Functoriality properties}

\sssec{}

Let $\CY_i$, $i=1,2$ be AG tame prestacks and $\CF_i\in \Shv(\CY_i\times \CY_i)$.

\medskip

Let now $\CH$ be an object in $\Shv(\CY_1\times \CY_2)$. Consider the corresponding map
$$\ul{[\CH]}:\uShv(\CY_1)\to \uShv(\CY_2).$$

Let us be given a map
$$\alpha:\CF_1\star \CH\to \CH\star \CF_2$$ 
in $\Shv(\CY_1\times \CY_2)$, where $(-)\star (-)$ denote the usual convolution operation:
\begin{multline*}
\CF_{1,2}\in \Shv(\CY_1\times \CY_2),\,\, \CF_{2,3}\in \Shv(\CY_2\times \CY_3) \mapsto \\
\CF_{1,2}\star \CF_{2,3}:=(p_1\times p_3)_\blacktriangle(\on{id}\times \Delta_{\CY_2}\times \on{id})^!(\CF_{1,2}\boxtimes \CF_{2,3}).
\end{multline*}

\sssec{} \label{sss:funct Tr gen}

Assume now that the 1-morphism $\ul{[\CH]}$ admits a right adjoint in $\AGCat$. I.e., there exists
$\CH^R\in \Shv(\CY_2\times \CY_1)$, equipped with the maps
$$\Delta_*(\omega_{\CY_1})\overset{\on{u}}\longrightarrow \CH\star \CH^R \text{ and }
\CH^R\star \CH \overset{\on{co-u}}\longrightarrow \Delta_*(\omega_{\CY_2}),$$
satisfying the usual adjunction axioms.

\medskip

In this case, by \cite[Sect. 3.3.4]{GKRV}, the map $\alpha$ induces a map
\begin{equation} \label{e:UL alpha 1}
\Tr(\ul{[\CF_1]},\uShv(\CY_1))\to \Tr(\ul{[\CF_2]},\uShv(\CY_2)).
\end{equation}

\sssec{}

On the other hand, we claim that the morphisms $\alpha$, $\on{u}$ and $\on{co-u}$ define a map
\begin{equation} \label{e:UL alpha 2}
\on{C}^\cdot_\blacktriangle(\CY_1,\Delta^!_{\CY_1}(\CF_1))\to
\on{C}^\cdot_\blacktriangle(\CY_2,\Delta^!_{\CY_2}(\CF_2)).
\end{equation}

Namely, \eqref{e:UL alpha 2} is the composition the following three morphisms:

\begin{enumerate}

\item The map
$$\on{C}^\cdot_\blacktriangle(\CY_1,\Delta^!_{\CY_1}(\CF_1)) \to
\on{C}^\cdot_\blacktriangle(\CY_1\times \CY_1\times \CY_2,{}'\Delta^!(\CF_1\boxtimes \CH\boxtimes \CH^R)),$$
induced by $\on{u}$,
where $'\Delta$ is the morphism
$$\CY_1\times \CY_1\times \CY_2\to \CY_1\times \CY_1\times \CY_1\times \CY_2\times \CY_2\times\CY_1, \quad
(y'_1,y''_1,y_2)\mapsto (y'_1,y''_1,y''_1,y_2,y_2,y'_1);$$

\item The map
$$\on{C}^\cdot_\blacktriangle(\CY_1\times \CY_1\times \CY_2,{}'\Delta^!(\CF_1\boxtimes \CH\boxtimes \CH^R))\to
\on{C}^\cdot_\blacktriangle(\CY_1\times \CY_2\times \CY_2,{}''\Delta^!(\CH\boxtimes \CF_2\boxtimes \CH^R)),$$
induced by $\alpha$, where $''\Delta$ is the morphism
$$\CY_1\times \CY_2\times \CY_2\to \CY_1\times \CY_2\times \CY_2\times \CY_2\times \CY_2\times \CY_1, \quad
(y_1,y'_2,y''_2)\mapsto (y_1,y'_2,y'_2,y''_2,y''_2,y_1);$$

\item The map
$$\on{C}^\cdot_\blacktriangle(\CY_1\times \CY_2\times \CY_2,{}''\Delta^!(\CH\boxtimes \CF_2\boxtimes \CH^R))\to
\on{C}^\cdot_\blacktriangle(\CY_2,\Delta^!_{\CY_2}(\CF_2)),$$
induced by $\on{co-u}$.

\end{enumerate}

\sssec{}

Unwinding, we obtain:

\begin{lem} \label{l:funct Tr}
The following diagram commutes:
$$
\CD
\Tr(\ul{[\CF_1]},\uShv(\CY_1)) @>{\text{\eqref{e:UL alpha 1}}}>>  \Tr(\ul{[\CF_2]},\uShv(\CY_2)) \\
@V{\text{\propref{p:Tr sheaf}}}V{\sim}V @V{\sim}V{\text{\propref{p:Tr sheaf}}}V  \\
\on{C}^\cdot_\blacktriangle(\CY_1,\Delta^!_{\CY_1}(\CF_1))  @>{\text{\eqref{e:UL alpha 2}}}>>
\on{C}^\cdot_\blacktriangle(\CY_2,\Delta^!_{\CY_2}(\CF_2)).
\endCD
$$
\end{lem}

\sssec{}

Here is a particular case of how \lemref{l:funct Tr} can be applied:

\medskip

Let $h:\CY_1\to \CY_2$ be a \emph{proper} schematic map. Note that by \propref{p:adj AGCat1},
the corresponding 1-morphism 
$$h_*:\uShv(\CY_1)\to \uShv(\CY_2)$$
admits a right adjoint, given by $h^!$. 

\medskip

Let us be given a map
$$\alpha: (\on{id}\times f)_*(\CF_1)\to (f\times \on{id})^!(\CF_2).$$

Note that we have a canonical map
\begin{equation} \label{e:fxd pt map}
\on{C}^\cdot_\blacktriangle(\CY_1,\Delta^!_{\CY_1}(\CF_1)) \to
\on{C}^\cdot_\blacktriangle(\CY_1,\Delta^!_{\CY_2}(\CF_2))
\end{equation}
defined as follows:

\begin{multline*}
\on{C}^\cdot_\blacktriangle(\CY_1,\Delta^!_{\CY_1}(\CF_1)) \to
\on{C}^\cdot_\blacktriangle(\CY_1,\Delta^!_{\CY_1}\circ (\on{id}\times f)^!\circ (\on{id}\times f)_*(\CF_1)) = \\
= \on{C}^\cdot_\blacktriangle(\CY_1,\on{Graph}_f^! \circ (\on{id}\times f)_*(\CF_1)) \overset{\alpha}\to
\on{C}^\cdot_\blacktriangle(\CY_1, \on{Graph}_f^! \circ (f\times \on{id})^!(\CF_2)) \simeq \\
\simeq \on{C}^\cdot_\blacktriangle(\CY_1,f^!\circ \Delta^!_{\CY_2}(\CF_2)) \simeq
\on{C}^\cdot_\blacktriangle(\CY_2,f_*\circ f^!\circ \Delta^!_{\CY_2}(\CF_2)) \to
\on{C}^\cdot_\blacktriangle(\CY_2,\Delta^!_{\CY_2}(\CF_2)).
\end{multline*}

We obtain:

\begin{cor}
Under the identification of \propref{p:Tr sheaf}, the morphism
$$\Tr(\ul{[\CF_1]},\uShv(\CY_1)) \to \Tr(\ul{[\CF_2]},\uShv(\CY_2))$$
induced by $f_*$ and $\alpha$ corresponds to the map \eqref{e:fxd pt map}.
\end{cor}

\sssec{}

We now consider another case of \lemref{l:funct Tr}. Namely, given $\CF\in \Shv(\CY\times \CY)$, we will
explain a geometric procedure that produces elements in
$$\Tr(\ul{[\CF]},\uShv(\CY))\simeq \on{C}^\cdot_\blacktriangle(\CY,\Delta^!_{\CY}(\CF)).$$

\sssec{} \label{sss:class UL 1}

Let $\CG\in \Shv(\CY)$
be a compact object. We claim that the resulting 1-morphism
$$[\CG]:\uno_{\AGCat}\to \uShv(\CY)$$
admits a right adjoint, given by $[\BD_\CY(\CG)]$, where $\BD_\CY$ is an
in \eqref{e:Verdier tame}.

\medskip

Indeed, this follows from \propref{p:adj AGCat1}, using the fact that for
$$X\in \Sch,\,\, \CG'\in \Shv(\CY\times X) \text{ and } \CG_X\in \Shv(X),$$
we have
$$\CHom_{\Shv(\CY\times X)}(\CG\boxtimes \CG_X,\CG')\simeq
\CHom_{\Shv(X)}\left(\CG_X,(p_2)_\blacktriangle(p_1^!(\BD_\CY(\CG))\sotimes \CG')\right),$$
where $p_1$ and $p_2$ are the two projections from $\CY\times X$ to $\CY$ and $X$, respectively.

\medskip

Indeed, we can assume that $\CG_X$ is compact, and both sides in the above formula
identify with $$\on{C}^\cdot_\blacktriangle\left(\CY\times X,(\BD_\CY(\CG)\boxtimes \BD_X(\CG_X))\sotimes \CG'\right)\overset{\text{\eqref{e:Verdier ext}}}\simeq
\on{C}^\cdot_\blacktriangle\left(\CY\times X,\BD_{\CY\times X}(\CG\boxtimes \CG_X)\sotimes \CG'\right).$$

\sssec{} \label{sss:class UL 2}

Let us now be given a map
$$\alpha: \CG\to [\CF](\CG)\simeq \CG\star \CF.$$

By \secref{sss:funct Tr gen}, to this data, we associate a map
$$\ol\BQ_\ell\to \Tr(\ul{[\CF]},\uShv(\CY)),$$
i.e., an element of $\Tr(\ul{[\CF]},\uShv(\CY))$, which we will denote by
$$\on{cl}(\ul{[\CG]},\alpha)\in \Tr(\ul{[\CF]},\uShv(\CY)).$$

\sssec{}

Let us compute this element in terms of the identification of \propref{p:Tr sheaf}. From \lemref{l:funct Tr}, we obtain
that this element is the image of the canonical element
$$\on{u}\in \on{C}^\cdot_\blacktriangle(\CY,\CG\sotimes \BD_\CY(\CG))$$
(the unit of the adjunction) under the map
$$\on{C}^\cdot_\blacktriangle(\CY,\CG\sotimes \BD_\CY(\CG))\simeq \on{C}^\cdot_\blacktriangle(\CY,\Delta_\CY^!(\CG\boxtimes \BD_\CY(\CG)))\to
\on{C}^\cdot_\blacktriangle(\CY,\Delta_\CY^!(\CF)),$$
induced by a map
$$\CG\boxtimes \BD_\CY(\CG)\to \CF,$$
obtained in the following three steps:

\medskip

\begin{enumerate}

\item We start with the counit of the adjunction, i.e., a map
$$\on{co-u}:\CG\boxtimes \BD_\CY(\CG) \to \Delta_*(\omega_\CY);$$

\item We apply to it $[\CF]$ along the first factor and thus obtain a map
$$[\CF](\CG)\boxtimes \BD_\CY(\CG)\to \CF;$$

\item We precompose the latter map with
$$\CG\boxtimes \BD_\CY(\CG) \overset{\alpha\boxtimes \on{id}}\longrightarrow [\CF](\CG)\boxtimes \BD_\CY(\CG).$$

\end{enumerate}

\ssec{Comparing traces in $\AGCat$ and $\DGCat$}

In this subsection we continue to assume that $\CY$ be an AG tame prestack, and we let $\CF$
be an object in $\Shv(\CY\times \CY)$.

\sssec{}

We now consider the induced action of $\ul{[\CF]}$ on $\be(\uShv(\CY))\simeq \Shv(\CY)$; denote the
corresponding endofunctor $[\CF]$. Explicitly, it is given by
$$[\CF](\CF_1):=\CF_1\mapsto (p_2)_\blacktriangle(p_1^!(\CF_1)\sotimes \CF).$$

We will now recall, following \cite{GaVa}, the construction of a map
\begin{equation} \label{e:LT true}
\on{LT}^{\on{true}}:\Tr([\CF],\Shv(\CY))\to \on{C}^\cdot_\blacktriangle(\CY,\Delta_\CY^!(\CF)).
\end{equation}

\sssec{}

Consider the object
$$([\CF]\otimes \on{Id})(\on{unit}_\CY)\in \Shv(\CY)\otimes \Shv(\CY),$$
where $\on{unit}_\CY\in \Shv(\CY)\otimes \Shv(\CY)$ is the unit of the self-duality.

\medskip

We claim:

\begin{lem} \label{l:boxtimes R}
There is a canonical isomorphism
$$([\CF]\otimes \on{Id})(\on{unit}_\CY)\simeq (\boxtimes^R)(\CF),$$
where $\boxtimes^R$ is the right adjoint to
$$\Shv(\CY)\otimes \Shv(\CY) \overset{\boxtimes}\to \Shv(\CY\times \CY).$$
\end{lem}

\begin{proof}

The assertion of the lemma can be reformulated as follows: for $\CF\in \Shv(\CY\times \CY)$, the object
$$[\CF]\in \on{Funct}(\Shv(\CY),\Shv(\CY))\simeq \Shv(\CY)^\vee\otimes \Shv(\CY)\simeq \Shv(\CY)\otimes \Shv(\CY)$$
is given by $\boxtimes^R(\CF)$.

\medskip

In fact, this is true for $\CF\in \Shv(\CY_1\times \CY_2)$ and $[\CF]\in \on{Funct}(\Shv(\CY_1),\Shv(\CY_2))$. This is
\cite[Lemma~4.5 and footnote~4]{GaVa}, which we reproduce here for completeness.

\medskip

It suffices to show that for $\CF_1\in \Shv(\CY_1)^c$ and $\CF_2\in \Shv(\CY_2)^c$,
\begin{equation} \label{e:compute kernel}
\CHom_{\Shv(\CY_1)\otimes \Shv(\CY_2)}(\CF_1\otimes \CF_2,[\CF])\simeq
\CHom_{\Shv(\CY_1\times \CY_2)}(\CF_1\boxtimes \CF_2,\CF).
\end{equation}

Unwinding, we obtain that the left-hand side in \eqref{e:compute kernel} is
\begin{multline*}
\CHom_{\Shv(\CY_1)}(\CF_1,[\CF](\BD_{\CY_2}(\CF_2)))\simeq
\CHom_{\Shv(\CY_1)}\left(\CF_1,(p_1)_\blacktriangle(\CF\sotimes p_2^!(\BD_{\CY_2}(\CF_2)))\right)\simeq \\
\simeq \on{C}^\cdot_\blacktriangle\left(\CY_1,\BD_{\CY_1}(\CF_1)\sotimes (p_1)_\blacktriangle(\CF\sotimes p_2^!(\BD_{\CY_2}(\CF_2)))\right)
\simeq
\on{C}^\cdot_\blacktriangle\left(\CY_1\times \CY_2,\CF\sotimes (\BD_{\CY_1}(\CF_1)\boxtimes \BD_{\CY_2}(\CF_2))\right) \simeq \\
\overset{\text{\eqref{e:Verdier ext}}}\simeq
\on{C}^\cdot_\blacktriangle\left(\CY_1\times \CY_2,\CF\sotimes \BD_{\CY_1\times \CY_2}(\CF_1\boxtimes \CF_2)\right),
\end{multline*}
which identifies with the right-hand side of \eqref{e:compute kernel}, as desired.

\end{proof}

\sssec{}

From \lemref{l:boxtimes R} we obtain a canonical map
\begin{equation} \label{e:boxtimes R back}
\boxtimes(([\CF]\otimes \on{Id})(\on{unit}_\CY))\simeq \boxtimes(\boxtimes^R(\CF))\to \CF
\end{equation}

By definition,
$$\Tr([\CF],\Shv(\CY))=\on{counit}_\CY(([\CF]\otimes \on{Id})(\on{unit}_\CY)).$$

By Remark \ref{r:counit Y}, we obtain that
$$\Tr([\CF],\Shv(\CY)) \simeq \on{C}^\cdot_\blacktriangle\left(\CY,\Delta_\CY^!\left(\boxtimes(([\CF]\otimes \on{Id})(\on{unit}_\CY))\right)\right).$$

Composing with \eqref{e:boxtimes R back} we obtain the desired map $\on{LT}^{\on{true}}$ in
\eqref{e:LT true}.

\sssec{}

Let $\ul\bC \in \AGCat$ be a dualizable object; in particular, by \corref{c:counit values}, $\be(\ul\bC)\in \DGCat$ is also dualizable.
Since the functor
$$\bi: \DGCat \to \AGCat$$
is symmetric monoidal, we have a map of dualizable objects
$$\bi\circ \be(\ul\bC) \to \ul{\bC}.$$

\medskip

By \Cref{p:adjoint of counit in DGCat-ag}, this 1-morphism admits a right adjoint.
Hence, by \cite[Sect. 3.4.3]{GKRV} for an endomorphism $\alpha$ of $\ul\bC$, we obtain a map
\begin{equation} \label{e:funct Tr}
\Tr(\bi\circ\be(\alpha),\bi\circ \be(\ul\bC))\to \Tr(\alpha,\ul\bC)
\end{equation}
in the category $\Maps_{\AGCat}(\uno_{\AGCat},\uno_{\AGCat})$.

\sssec{}

In addition, since $\bi$ is symmetric monoidal, we have
$$\Tr(\bi\circ\be(\alpha),\bi\circ \be(\ul\bC))\simeq \bi(\Tr(\be(\alpha),\be(\ul\bC))),$$
where:

\begin{itemize}

\item $\Tr(\be(\alpha),\be(\ul\bC))\in \Maps_{\DGCat}(\uno_{\DGCat},\uno_{\DGCat})\simeq \Vect$;

\medskip

\item The symbol $\bi$ in the right-hand side 
indicates the functor $$\Vect\simeq\Maps_{\DGCat}(\uno_{\DGCat},\uno_{\DGCat})\to \Maps_{\AGCat}(\uno_{\AGCat},\uno_{\AGCat})\simeq\Vect,$$
induced by $\bi$, which is (homotopic to) the identity.

\end{itemize}

\sssec{}

Summarizing, we obtain a map
\begin{equation} \label{e:funct Tr bis}
\Tr(\be(\alpha),\be(\ul\bC))\to \Tr(\alpha,\ul\bC)
\end{equation}
in $\Vect$.

\sssec{}

We apply the above discussion to $\ul\bC=\uShv(\CY)$ and $\alpha=\ul{[\CF]}$, so that $\be(\alpha)=[\CF]$.
Thus, \eqref{e:funct Tr bis} yields a map
\begin{equation} \label{e:true again}
\Tr([\CF],\Shv(\CY))\to \Tr(\ul{[\CF]},\uShv(\CY)).
\end{equation}

\sssec{}

Let $\CG$ be a compact object of $\Shv(\CY)$ equipped with a map
$$\alpha:\CG\to [\CF](\CG).$$

By \secref{sss:class UL 2}, to this data we assign an element
$$\on{cl}(\ul{[\CG]},\alpha)\in \Tr(\ul{[\CF]},\uShv(\CY)).$$

In addition, we can consider
$$\on{cl}([\CG],\alpha)\in \Tr(\ul{\CF},\Shv(\CY)).$$

From the functoriality of the natural transformation \eqref{e:funct Tr bis}, we obtain:

\begin{lem} \label{l:class UL}
The image image of $\on{cl}([\CG],\alpha)$ under \eqref{e:true again} equals
$\on{cl}(\ul{[\CG]},\alpha)$.
\end{lem}

\sssec{}

We are going to prove:

\begin{thm} \label{t:true LT}
The map
$$\Tr([\CF],\Shv(\CY))\overset{\text{\eqref{e:true again}}}\to \Tr(\ul{[\CF]},\uShv(\CY)) \overset{\on{LT}^{\on{AG}}}\simeq
\on{C}^\cdot_\blacktriangle(\CY,\Delta_\CY^!(\CF))$$
equals the above map $\on{LT}^{\on{true}}$.
\end{thm}

\begin{proof}

We first consider the case when $\CF$ belongs to the full subcategory
$$\Shv(\CY)\otimes \Shv(\CY)\subset \Shv(\CY\times \CY).$$

In this case, the assertion of the theorem is obtained by unraveling the definitions
(and the maps $\on{LT}^{\on{true}}$ and \eqref{e:true again} are isomorphisms).

\medskip

Next we note that for a morphism $\alpha:\CF'\to \CF''$ in $\Shv(\CY\times \CY)$, we have the commutative
diagrams
$$
\CD
\Tr([\CF'],\Shv(\CY)) @>{\on{LT}^{\on{true}}}>> \on{C}^\cdot_\blacktriangle(\CY,\Delta_\CY^!(\CF')) \\
@V{\alpha}VV @VV{\alpha}V \\
\Tr([\CF''],\Shv(\CY)) @>{\on{LT}^{\on{true}}}>> \on{C}^\cdot_\blacktriangle(\CY,\Delta_\CY^!(\CF''))
\endCD
$$
and
$$
\CD
\Tr([\CF'],\Shv(\CY))  @>{\text{\eqref{e:true again}}}>> \Tr(\ul{[\CF']},\uShv(\CY)) @>{\on{LT}^{\on{AG}}}>>  \on{C}^\cdot_\blacktriangle(\CY,\Delta_\CY^!(\CF'))  \\
@V{\alpha}VV @V{\alpha}VV @VV{\alpha}V \\
\Tr([\CF''],\Shv(\CY))  @>{\text{\eqref{e:true again}}}>> \Tr(\ul{[\CF'']},\uShv(\CY)) @>{\on{LT}^{\on{AG}}}>>  \on{C}^\cdot_\blacktriangle(\CY,\Delta_\CY^!(\CF'')).
\endCD
$$

\medskip

Hence, if the assertion of the theorem holds for $\CF'$ and the map
$$\Tr([\CF'],\Shv(\CY)) \overset{\alpha}\to \Tr([\CF''],\Shv(\CY))$$
is an isomorphism, then it also holds for $\CF''$.

\medskip

Finally, we note that for $\CF'':=\CF$ and $\CF':=\boxtimes(\boxtimes^R(\CF))$, we have
$$\CF'\in \Shv(\CY)\otimes \Shv(\CY)\subset \Shv(\CY\times \CY),$$
and by \lemref{l:boxtimes R}, the counit map
$$\boxtimes(\boxtimes^R(\CF))\to \CF$$
induces an isomorphism
$$[\boxtimes(\boxtimes^R(\CF))]\to [\CF]$$
(as endofunctors of $\Shv(\CY)$), and hence also an isomorphism of traces.

\end{proof}

\ssec{The Frobenius correspondence} \label{ss:more Frob}

In this subsection we let $\CY$ be an algebraic stack over $\ol\BF_q$, defined over $\BF_q$.
We will assume that $\CY$ is AG Verdier compatible (and hence AG tame).

\sssec{} \label{sss:back to alg stack}

Note that since $\CY$ was assumed to be an algebraic stack, the stack
$\CY^\Frob$ is \emph{algebro-geometrically discrete}, i.e., it is isomorphic to the groupoid
$\CY(\BF_q)$ with the trivial algebro-geometric structure (see \cite[Lemma 6.4]{GaVa}).

\medskip

In particular,
$$\on{C}_\blacktriangle(\CY^\Frob,\omega_{\CY^\Frob})\simeq \sFunct(\CY(\mathbb{F}_q),\ol\BQ_\ell).$$

\sssec{}

Taking \eqref{e:hom corr} to be
$$ \xymatrix{
\CY & \ar[l]_{\on{id}} \CY \ar[r]^{\on{Frob}} & \CY,
}$$
from \corref{c:Tr corr}, we obtain:

\begin{cor} \label{c:Tr Frob AG}
Let $\CY$ be an AG Verdier-compatible algebraic stack defined over $\mathbb{F}_q$.
Then
$$\on{Tr}(\on{Frob}_{\blacktriangle}; \uShv(\CY)) \simeq \sFunct(\CY(\mathbb{F}_q), \ol\BQ_\ell). $$
\end{cor}

\sssec{}

Consider the composition
\begin{multline} \label{e:funct Tr bis bis}
\Tr(\Frob_*,\Shv(\CY)) \overset{\text{\lemref{l:Frob triangle}}}\simeq
\Tr(\Frob_\blacktriangle,\Shv(\CY))\overset{\text{\eqref{e:true again}}}\longrightarrow  \\
\to \Tr(\Frob_\blacktriangle,\uShv(\CY))\overset{\text{\corref{c:Tr Frob AG}}}\simeq \sFunct(\CY(\mathbb{F}_q), \ol\BQ_\ell).
\end{multline}

We claim:

\begin{thm} \label{t:LT Frob}
The map \eqref{e:funct Tr bis bis} equals the map $\on{LT}^{\on{naive}}$
of \eqref{e:def LT naive}.
\end{thm}

\begin{proof}

By \thmref{t:true LT}, the composition \eqref{e:funct Tr bis bis} equals $\on{LT}^{\on{true}}$.
Now, in the particular case of the Frobenius correspondence,
$$\on{LT}^{\on{true}}=\on{LT}^{\on{naive}},$$
by one of the main results of \cite{GaVa}, see Theorem 0.8 in {\it loc. cit.}

\end{proof}

\sssec{}

Let now $\CF$ be an object in $\Shv(\CY)$. We can view $\CF$ as a morphism
$$\ul{[\CF]}:\uno_{\AGCat}\to \uShv(\CY)$$
in $\AGCat$. Note that if $\CF$ is compact, then $[\CF]$ admits a right adjoint, see \secref{sss:class UL 1}.

\medskip

Moreover, let $\CF$ be equipped with a weak Weil structure, i.e., a morphism
$$\alpha:\CF\to \Frob_*(\CF)\simeq \Frob_\blacktriangle(\CF).$$

\medskip

Then by \secref{sss:class UL 2}, to this data we attach an element
$$\on{cl}(\ul{[\CF]},\alpha)\in  \Tr(\Frob_\blacktriangle,\uShv(\CY)).$$

From \thmref{t:LT Frob}, \lemref{l:class UL} and \eqref{e:faisceaux-fonctions}, we obtain:

\begin{cor} \label{c:faisceaux-fonctions again}
The image of $\on{cl}(\ul{[\CF]},\alpha)$ under
$$\Tr(\Frob_\blacktriangle,\uShv(\CY))\overset{\text{\corref{c:Tr Frob AG}}}\simeq
\sFunct(\CY(\mathbb{F}_q), \ol\BQ_\ell)$$
equals $\sfunct(\CF,\alpha)$.
\end{cor}

\appendix

\section{Recollections on enriched categories} \label{s:enriched cat}

This is a summary of the relevant parts of the theory of enriched $\infty$-categories developed in \cite{GepHaug,Heine,Hinich-Yoneda,Hinich-universal}.

\ssec{The basics} \label{ss:flagged}

\sssec{}

Recall that one of the (most convenient) frameworks for talking about $\infty$-categories
is via Segal spaces (see \cite[Chapter 10, Sect. 1]{GR1} for a review). In this approach, an $\infty$-category $\CC$
is encoded by a space $\CC^{\on{grpd}}$ of objects of $\bC$ and a functor
$$\CC^{\on{grpd}}\times \CC^{\on{grpd}}\to \Spc, \quad (c_1,c_2)\mapsto \Maps_{\CC}(c_1,c_2),$$
equipped with a multiplicative structure.

\medskip

By contrast, when discussing enriched categories, it is often convenient to parameterize
its objects by another space.

\sssec{}

Let $\CA$ be a monoidal category and $\CS$ a space. To this pair, one associates the notion of
$\CS$-flagged $\CA$-enriched category, see \cite[Sect. 2.4]{GepHaug}.

\medskip

An $\CS$-flagged $\CA$-enriched category
$\CC$ consists of a functor
\begin{equation} \label{e:flagged enr}
\CS\times \CS\to \CA, \quad (s_1,s_2)\mapsto \Hom^\CA_{\CS,\CC}(s_1,s_2), \quad s_1,s_2\in \CS,
\end{equation}
endowed with a multiplicative structure
$$\Hom^\CA_{\CS,\CC}(s_2,s_3)\otimes \Hom^\CA_{\CS,\CC}(s_1,s_2)\to \Hom^\CA_{\CS,\CC}(s_1,s_3),$$
equipped with a homotopy-coherent data of associativity.

\sssec{Example} \label{sss:ex enriched}

One can take $\CS=\{*\}$ and $\Hom^\CA(\{*\},\{*\})$ to be any algebra object $A\in \CA$.

\sssec{}

The totality of $\CS$-flagged $\CA$-enriched categories forms a category, denoted $\Cat^\CA(\CS)$.
Furthermore, we can form the ``total space" $\Cat^\CA(\Spc)$, equipped with a Cartesian fibration
\begin{equation} \label{e:flagged}
\Cat^\CA(\Spc)\to \Spc.
\end{equation}

\sssec{} \label{sss:induce rich}

Let $F:\CA_1\to \CA_2$ be a \emph{right-lax} monoidal functor. Given $(\CS_1,\CC_1)\in \Cat^{\CA_1}(\Spc)$
and $(\CS_2,\CC_2)\in \Cat^{\CA_2}(\Spc)$, there is a natural notion of functor
$$(\CS_1,\CC_1)\to (\CS_2,\CC_2)$$
compatible with $F$. The collections of such functors forms a functor
\begin{equation} \label{e:Hom flagged}
(\Cat^{\CA_1}(\Spc))^{\on{op}}\times \Cat^{\CA_2}(\Spc)\to \Spc.
\end{equation}

\medskip

Keeping $(\CS_1,\CC_1)$ fixed and making $(\CS_2,\CC_2)$ variable, the functor \eqref{e:Hom flagged} is corepresentable
by an object, denoted
$$\ind^{\CA_2}_{\CA_1}(\CS_1,\CC_1)\in \Cat^{\CA_2}(\Spc).$$

The assignment
$$(\CS_1,\CC_1)\mapsto \ind^{\CA_2}_{\CA_1}$$
is a functor
\begin{equation} \label{e:induce enrich flagged}
\ind^{\CA_2}_{\CA_1}:\Cat^{\CA_1}(\Spc)\to \Cat^{\CA_2}(\Spc).
\end{equation}

The object $\ind^{\CA_2}_{\CA_1}(\CS_1,\CC_1)=:(\CS_2,\CC_2)$ can be described as follows:
$\CS_2=\CS_1=:\CS$ and for $s',s''\in \CS$,
$$\Hom^{\CA_2}_{\CS,\CC_2}(s',s''):=F(\Hom^{\CA_1}_{\CS,\CC_1}(s',s'')).$$

The composition rule is induced by one on $\CC_1$ using the right-lax
monoidal structure on $F$.

\sssec{}

Take $\CA=\Spc$ equipped with the Cartesian monoidal structure. We have a fully faithful functor
$$\sft:\Cat\to \Cat^\Spc(\Spc)$$
that sends a category $\CC$ to the space $\CS:=\CC^{\on{grpd}}$ and
$$\Hom^\Spc_{\CS,\CC}(c_1,c_2):=\Maps_\CC(c_1,c_2).$$

The functor $\sft$ admits a left adjoint, to be denoted $\sft^L$. Thus, we can view
$\Cat$ as a localization of $\sft^L$.

\begin{rem}

One can give an explicit description of the arrows in $\Cat^\Spc(\Spc)$ that get sent to isomorphisms
by $\sft^L$:

\medskip

Note that to an object $(\CS,\CC)$ one can naturally attach the \emph{set} $\pi_0(\CS,\CC)$: the quotient of
$\pi_0(\CS)$ by the equivalence relation generated by $\pi_0(\CC)$.

\medskip

A morphism $(\CS_1,\CC_1)\to (\CS_2,\CC_2)$ becomes an isomorphism under $\sft^L$ if and only if:

\medskip

\begin{itemize}

\item It is Cartesian with respect to $\Cat^\Spc(\Spc)\to \Spc$;

\medskip

\item The induced map $\pi_0(\CS_1,\CC_1)\to \pi_0(\CS_2,\CC_2)$ is a surjection.

\end{itemize}

\end{rem}

\sssec{} \label{sss:underlying flagged}

Consider the (right-lax) monoidal functor
$$\CA\to \Spc, \quad a\mapsto \Maps_\CA(\uno_\CA,a).$$

Applying the corresponding functor $\ind^\Spc_\CA$, we obtain a functor
$$\Cat^\CA(\Spc)\to \Cat^\Spc(\Spc),$$
which we denote by $\oblv_{\CA\on{-Enr}}$ (or simply $\oblv_{\on{Enr}}$ if the choice of $\CA$ is clear).

\sssec{}

A morphism in $\Cat^\CA(\Spc)$ is called an \emph{enriched equivalence} if:

\medskip

\begin{itemize}

\item It is Cartesian with respect to \eqref{e:flagged};

\medskip

\item Its image under
$$\Cat^\CA(\Spc)\overset{\oblv_{\CA\on{-Enr}}}\longrightarrow \Cat^\Spc(\Spc) \overset{\sft^L}\to \Cat$$
is an isomorphism.

\end{itemize}

\medskip

We define $\Cat^\CA$ to be the localization of $\Cat^\CA(\Spc)$ obtained by inverting
enriched equivalences.

\medskip

Note that by definition
$$\Cat^\Spc\simeq \Cat.$$

\sssec{}

Let us be in the setting of \secref{sss:induce rich}. One shows that the functor \eqref{e:induce enrich flagged}
induces a functor on the corresponding localizations:
\begin{equation} \label{e:induce enrich}
\Cat^{\CA_1}\to \Cat^{\CA_2},
\end{equation}
which we denote by the same character $\ind^{\CA_2}_{\CA_1}$.

\sssec{The underlying category} \label{sss:forget enr}

In particular, applying this to the situation in \secref{sss:underlying flagged}, we obtain a functor
$$\Cat^\CA\to \Cat^\Spc\simeq \Cat,$$
which we denote by $\oblv_{\CA\on{-Enr}}$ (or simply $\oblv_{\on{Enr}}$ if the choice of $\CA$ is clear).

\medskip

For $\CC\in \Cat^\CA$, we will refer to $\oblv_{\CA\on{-Enr}}(\CC)$ as the \emph{underlying category}.

\sssec{}

Let $\CC$ be an object of $\Cat^\CA(\CS)$. Unwinding, we obtain that there is a naturally defined map
\begin{equation} \label{e:flagging to pts}
\CS\to (\oblv_{\on{Enr}}(\CC))^{\on{grpd}}.
\end{equation}

Note, however, that this map is \emph{not necessarily} an equivalence. For example, in the example of
\secref{sss:ex enriched},
$$(\oblv_{\on{Enr}}(\CC))^{\on{grpd}}=B(\Maps_\CA(\uno_\CA,A)^\times),$$
where:

\begin{itemize}

\item $\Maps_\CA(\uno_\CA,A)$ is viewed as a monoid in $\Spc$;

\medskip

\item $(-)^\times$ denotes the group of invertible points in a given monoid;

\medskip

\item $B(-)$ denotes the classifying space of a given group.

\end{itemize}

Note also that the functor $\ind^{\CA_2}_{\CA_1}$ changes the underlying groupoid of the underlying category,
which is one of the reasons to consider flaggings.

\sssec{}

That said, there is a canonical fully faithful embedding
\begin{equation} \label{e:section flagging}
\Cat^\CA\to \Cat^\CA(\Spc),
\end{equation}
whose essential image consists of those pairs $(\CS,\CC)$ for which the map \eqref{e:flagging to pts}
is an isomorphism.

\medskip

In other words, the functor \eqref{e:section flagging} sends $\CC\in \Cat^\CA$ to its canonical flagging by
$(\oblv_{\on{Enr}}(\CC))^{\on{grpd}}$.

\medskip

Thus, in particular, we have a well-defined functor
$$\Hom^\CA_{\CC}: \CC^{\on{grpd}} \times \CC^{\on{grpd}} \to \CA, \quad (c_1,c_2)\mapsto \Hom^{\CA}_\CC(c_1,c_2)\in \CA.$$

\sssec{}

The above procedure defines $\Cat^\CA$ as an $(\infty,1)$-category, in particular, for
$\CC_1,\CC_2\in \Cat^\CA$ we obtain the \emph{space}
$$\Maps_{\Cat^\CA}(\CC_1,\CC_2)$$
of $\CA$-enriched functors $\CC_1\to \CC_2$.

\medskip

Now, there is also a naturally defined notion of \emph{natural transformation} between $\CA$-enriched functors,
which upgrades $\Maps_{\Cat^\CA}(\CC_1,\CC_2)$ to an  $(\infty,1)$-category
$$\Funct^\CA(\CC_1,\CC_2), \quad \Maps_{\Cat^\CA}(\CC_1,\CC_2)\simeq \left(\Funct^\CA(\CC_1,\CC_2)\right)^{\on{grpd}}.$$

Furthermore, the assignment
$$\CC_1,\CC_2\mapsto \Funct^\CA(\CC_1,\CC_2)$$
extends to a structure of $(\infty,2)$-category on $\Cat^\CA$, in the sense of \cite[Chapter 10, Sect. 2]{GR1}.

\medskip

In addition, for $F:\CA_1\to \CA_2$, the functor $\ind^{\CA_2}_{\CA_1}$ extends to a functor of $(\infty,2)$-categories
$$\Cat^{\CA_1}\to \Cat^{\CA_2}.$$

\ssec{A set-theoretic digression}

In this subsection, we explain the approach adopted in this paper for dealing
with set-theoretic issues.

\sssec{}

Given a cardinal $\mu$, we shall say that a category is of size $<\mu$ if in Joyal's model
of quasi-categories, it can be represented by a simplicial set of cardinality $<\mu$, see \cite[Sect. 5.4.1]{HTT}.

\medskip

We let $\Cat_{<\mu}$ the category of categories of size $<\mu$.

\sssec{}

We fix a sequence of inaccessible cardinals
$$\lambda_0 < \lambda_1 < \lambda_2.$$

\medskip

In particular, for $i=0,1$, categories generated under $\lambda_i$-small colimits by subcategories of size $<\lambda_i$ are of
size $<\lambda_{i+1}$.

\sssec{}

By the definition of inaccessibility, we have:

\medskip

\begin{itemize}

\item The category $\Sch$ of schemes of finite type over the ground field $k$ is of size $<\lambda_0$;

\medskip

\item For $X\in \Sch$, the category $\Shv(X)^{\on{constr}}$
(of \emph{constructible} $\ell$-adic sheaves on $X$) is of size $<\lambda_0$.

\end{itemize}

We will refer to categories of size $<\lambda_0$ (resp., $<\lambda_1$, $<\lambda_2$) as
\emph{small} (resp., \emph{large}, \emph{huge}).

\medskip

Note that the category $\Cat_{\on{small}}$ of small categories is itself large, while the
category $\Cat_{\on{large}}$ of large categories is huge.

\sssec{} \label{sss:pres 1}

Recall, following \cite[Sect. 5.5]{HTT} (see also \cite[Sect. 5.3.2]{HA}), that given a regular cardinal $\kappa<\lambda_0$, a category is said to be $\kappa$-presentable
if it contains small colimits, and is generated under small $\kappa$-filtered colimits by
a \emph{small} subcategory of $\kappa$-compact objects.

\medskip

Note that a $\kappa$-presentable category is automatically (not more than) large.

\sssec{} \label{sss:pres 1.5}

We let $\Cat^{\on{pres}}_\kappa$ denote the category, whose objets are $\kappa$-presentable categories,
and whose morphisms are functors that commute with all colimits and preserve $\kappa$-compact
objects.

\medskip

Note that $\Cat^{\on{pres}}_\kappa$ itself is $\kappa$-presentable; in particular,
it is (not more than) large.

\medskip

We will regard $\Cat^{\on{pres}}_\kappa$ as a symmetric monoidal category with respect to the Lurie tensor product.

\sssec{}

Note that for $\kappa<\kappa'$, the inclusion
$$\Cat^{\on{pres}}_\kappa\to \Cat^{\on{pres}}_{\kappa'}$$
is symmetric monoidal and preserves all small colimits and limits of size $<\kappa$.

\sssec{} \label{sss:pres 2}

We let $\Cat^{\on{pres}}$ denote the category, whose objects are $\kappa$-presentable categories
for some $\kappa<\lambda_0$, and whose morphisms are functors that preserve $\lambda_0$-small colimits.

\medskip

We have
$$\Cat^{\on{pres}}\simeq \underset{\kappa<\lambda_0}{\on{colim}}\, \Cat^{\on{pres}}_\kappa,$$
where the colimit is taken in $\Cat_{\on{large}}$.

\medskip

We will regard $\Cat^{\on{pres}}$ as a symmetric monoidal category with respect to the Lurie tensor product.



\sssec{} \label{sss:DG Cat kappa} \label{sss:DG Cat}

Repeating the contents of Sects. \ref{sss:pres 1}-\ref{sss:pres 2}, we obtain:

\medskip

\begin{itemize}

\item The notion of $\kappa$-presentable DG category;

\medskip

\item The category $\DGCat_{\kappa}$ of $\kappa$-presentable DG categories, equipped
with a symmetric monoidal structure;

\medskip

\item The large category $$\DGCat\simeq \underset{\kappa<\lambda_0}{\on{colim}}\, \DGCat_{\kappa},$$
where the colimit is taken in $\Cat_{\on{large}}$.

\end{itemize}

\sssec{}

The category $\DGCat$ is the natural category of DG categories with which one works.

\ssec{Module categories--a recap}

\sssec{} \label{sss:free cat}

For $\CS \in \on{Spc}$ and $\bC\in \Cat^{\on{pres}}$, set
$$\CS \otimes \bC:= \underset{\CS}{\on{colim}}\, \bC.$$

For a point $s\in S$, we have a natural insertion functor
$$\on{ins}_s:\bC\to \CS \otimes \bC,$$
which admits a colimit-preserving right adjoint, to be denoted $\on{ev}_s$.

\medskip

The functors $\on{ev}_s$ assemble to a colimit-preserving functor
\begin{equation} \label{e:limit colimit}
\CS \otimes \bC \to \underset{\CS^{\on{op}}}{\on{lim}}\, \bC\simeq \underset{\CS}{\on{lim}}\, \bC\simeq \on{Funct}(\CS,\bC).
\end{equation}

The following is a particular case of \cite[Chapter 1, Proposition 2.5.7]{GR1}:

\begin{lem} \label{l:limit and colimit}
The functor \eqref{e:limit colimit} is an equivalence.
\end{lem}

\begin{cor} \label{c:limit and colimit}
For an object $\Phi \in \CS \otimes \bC$, we have a canonical isomorphism
$$\Phi\simeq \underset{s\in \CS}{\on{colim}}\, \on{ins}_s\circ \on{ev}_s(\Phi).$$
\end{cor}

\sssec{}  \label{sss:pres monoidal}

Let $\CA$ be an algebra object in $\Cat^{\on{pres}}$
(see \secref{sss:pres 2}). We let $\CA\mmod$ denote the category of $\CA$-modules in $\Cat^{\on{pres}}$.

\medskip

Given $\CM, \CN \in \CA\mathbf{\mbox{-}mod}$, we denote by
$$ \on{Funct}_{\CA}(\CM, \CN) $$
the category of $\CA$-linear functors.

\sssec{} \label{sss:easy duality}

Applying Sect. \ref{sss:free cat} to $\bC=\CA$, we obtain the object
$$\CS \otimes \CA\in \CA\mmod,$$
and an isomorphism
\begin{equation} \label{e:limit colimit V}
\CS \otimes \CA \simeq \on{Funct}(\CS,\CA)
\end{equation}
in $\CA\mmod$.

\sssec{}

For $s\in \CS$ denote $\delta_s\in \CS\otimes \CA\simeq \on{Funct}(S,\CA)$ the object
$$\on{ins}_s(\uno_\CA).$$

From Corollary \ref{c:limit and colimit}, we obtain a canonical isomorphism
\begin{equation} \label{e:object as a colimit of delta}
\Phi\simeq \underset{s\in S}{\on{colim}}\, \Phi(s)\otimes \delta_s, \quad \Phi\in \on{Funct}(S,\CA),
\end{equation}
where $(-)\otimes (-)$ stands for the action of $\CA$ on $\on{Funct}(S,\CA)$.

%
%
%

\sssec{} \label{sss:free self-dual}

Recall that given an object $\CM\in \CA\mmod$, its dual (if it exists) is an object $\CM^\vee\in \CA^{\on{rev}}\mmod$ such that
$$\on{Funct}_\CA(\CM,\CN) \simeq \CM^\vee\underset{\CA}\otimes \CN, \quad \CN\in \CA\mmod.$$
In this case we shall say that $\CM$ is dualizable\footnote{Here and elsewhere, for an algebra $A$ in a symmetric monoidal category,
$A^{\on{rev}}$ denotes the algebra obtained by reversing the multiplication.}.

\medskip

We claim that the object $\CS \otimes \CA\in \CA\mmod$ is dualizable with dual $\CS\otimes \CA^{\on{rev}}$.
Indeed, for any
$\CN\in \CA\mmod$ we have
\begin{multline*}
\on{Funct}_{\CA}(\CS \otimes \CA,\CN)=
\on{Funct}_{\CA}(\underset{\CS}{\on{colim}}\, \CA,\CN) \simeq
\underset{\CS}{\on{lim}}\, \on{Funct}_{\CA}(\CA,\CN) \simeq \\
\simeq \underset{\CS}{\on{lim}}\, \CN \overset{\text{Lemma \ref{l:limit and colimit}}}\simeq
\underset{\CS}{\on{colim}}\, \CN \simeq \underset{\CS}{\on{colim}}\,  (\CA\underset{\CA}\otimes \CN) \simeq
(\underset{\CS}{\on{colim}}\, \CA) \underset{\CA}\otimes \CN =
(\CS\otimes \CA^{\on{rev}}) \underset{\CA}\otimes \CN,
\end{multline*}
as required.

\medskip

In particular, we obtain that for $\CS_1,\CS_2$ we have a canonical equivalence of categories.
\begin{equation} \label{e:duality funct}
\on{Funct}_{\CA}(\CS_1 \otimes \CA,\CS_2 \otimes \CA)\simeq
\on{Funct}_{\CA^{\on{rev}}}(\CS_2 \otimes \CA^{\on{rev}},\CS_1 \otimes \CA^{\on{rev}}).
\end{equation}

\begin{rem}

Note, however, that when $\CS_1=\CS_2=\CS$, both sides of \eqref{e:duality funct} have natural
monoidal structures, while \eqref{e:duality funct} reverses them.

\end{rem}

\ssec{Enriched vs tensored} \label{ss:enr vs ten}

\sssec{}

In this subsection we let $\CA$ be as in \secref{sss:pres monoidal}. We will distinguish two classes
of $\CA$-enriched categories.

\medskip

We will keep the notation $\Cat^\CA$ for the category of $\CA$-enriched categories that can be flagged
by a \emph{small} space. We will denote by $\Cat_{\on{large}}^\CA$ the category of $\CA$-enriched categories
that can be flagged by a large space.

\medskip

We have the obvious inclusion
$$\Cat^\CA\subset \Cat_{\on{large}}^\CA.$$

\sssec{}

Since $\CA$ was assumed presentable, for $\CM\in \CA\mmod$ and a pair of objects $m_1,m_2\in \CM$, we have a
well-defined \emph{internal-relative to} $\CA$ Hom
$$\uHom_{\CM,\CA}(m_1,m_2)\in \CA, \quad \Maps_\CA(v,\uHom_{\CM,\CA}(m_1,m_2)):=\Maps_\CM(v\otimes m_1,m_2).$$

This allows us to view $\CM$ as \emph{enriched} over $\CA$:
$$\Hom^\CA_\CM(m_1,m_2):=\uHom_{\CM,\CA}(m_1,m_2),$$
see \cite[Corollary 7.4.9]{GepHaug}.

\medskip

This way we obtain a functor
\begin{equation} \label{e:tensored to enriched}
\ModtoEnr:\CA\mmod\to \Cat^\CA_{\on{large}}.
\end{equation}

\sssec{}

In particular, for $\CC\in \Cat^\CA_{\on{large}}$ and $\CM\in \CA\mmod$,
it makes sense to talk about the category
$$\on{Funct}^\CA(\CC,\CM):=\on{Funct}^\CA(\CC,\ModtoEnr(\CM))$$
of $\CA$-functors.

\medskip

By definition, the datum of an object $\Phi\in \on{Funct}^\CA(\CC,\CM)$ consist of a functor
$$\Phi:\oblv_{\CA\on{-Enr}}(\CC)\to \CM,$$
equipped with a compatible family of morphisms
$$\Hom^{\CA}_\CC(c_1,c_2)\otimes \Phi(c_1)\to \Phi(c_2),$$
where $(-)\otimes (-)$ refers to  the action of $\CA$ on $\CM$.

\begin{rem}
In \secref{ss:quiv} we will explain a different point of view on $\on{Funct}^\CA(\CC,\CM)$,
developed in \cite{Hinich-Yoneda,Hinich-universal}.
\end{rem}

\ssec{Enriched presheaves}

\sssec{}

We now claim that the functor $\ModtoEnr$ admits a (partially defined) left adjoint, defined on $\Cat^\CA\subset \Cat^\CA_{\on{large}}$,
to be denoted
\begin{equation} \label{e:enriched presheaves}
\CC\mapsto \bP^\CA(\CC).
\end{equation}

\medskip

Explicitly,
$$\bP^\CA(\CC)=\on{Funct}^{\CA^{\on{rev}}}(\CC^{\on{op}},\CA^{\on{rev}}),$$
where:

\begin{itemize}

\item $\CA^{\on{rev}}$ denotes the monoidal category obtained from $\CA$ by reversing the monoidal operation;

\medskip

\item $\CC^{\on{op}}$ is regarded as enriched over $\CA^{\on{rev}}$;

\medskip

\item $\on{Funct}^{\CA^{\on{rev}}}(\CC^{\on{op}},\CA^{\on{rev}})$ is regarded as an $\CA$-module category,
via the $\CA$-action on $\CA^{\on{rev}}$ ``on the right" (which commutes with the natural left action of
$\CA^{\on{rev}}$ on itself).

\end{itemize}

\begin{rem}

The assumption that $\CC\in \Cat^\CA$ is needed in order to make $\bP^\CA(\CC)$
a large (as opposed to huge) category.

\end{rem}

%
%
%
%
%
%

\sssec{}

Concretely, an object of $\bP^\CA(\CC)$ is a functor
$$\Phi:\CC^{\on{op}}\to \CA,$$
equipped with a compatible data of maps
$$\Phi(c_2)\otimes \Hom^{\CA}_\CC(c_1,c_2)\to \Phi(c_1), \quad c_1,c_2\in \CC.$$

\sssec{}

The unit of the $(\bP^\CA(-),\ModtoEnr)$-adjunction is the $\CA$-functor
\begin{equation} \label{e:enriched Yoneda}
\CC\to \on{Funct}^{\CA^{\on{rev}}}(\CC^{\on{op}},\CA^{\on{rev}}),
\end{equation}
that sends
$$c\mapsto \Hom^{\CA}_\CC(-,c).$$

We will refer to \eqref{e:enriched Yoneda} as the \emph{$\CA$-enriched Yoneda functor}, and denote it by
$$\on{Yon}^\CA_\CC\in \on{Funct}^\CA(\CC,\bP^\CA(\CC)).$$

\sssec{}

In Sect. \ref{ss:quiv} we will give a slightly different description of the category $\bP^\CA(\CC)$, from
which we will deduce (see Sect. \ref{sss:proof prshv}):

\begin{prop} \label{p:env enh}
For $\CC\in \Cat^\CA$ and $\CM \in \CA\mmod$, precomposition with $\on{Yon}^\CA_\CC$ gives rise to an equivalence
$$\on{Funct}_{\CA}(\bP^\CA(\CC), \CM) \to \on{Funct}^{\CA}(\CC, \CM).$$
\end{prop}

\medskip

Note that the assertion of Proposition \ref{p:env enh}
is equivalent to the statement that $\on{Yon}^\CA_\CC$ indeed makes $\bP^\CA(-)$ a left adjoint
of $\ModtoEnr$ on $\Cat^\CA$.

\sssec{}\label{sss:change of enrichment tensor}

Let $\CA \to \CA'$ be a monoidal functor. From Proposition \ref{p:env enh} we obtain a
canonical identification:
$$\mathbf{P}^{\CA'}(\ind_{\CA}^{\CA'} (\CC)) \simeq \CA'\underset{\CA}\otimes \mathbf{P}^{\CA}(\CC).$$

\sssec{}

Recall that a functor $\Phi:\CM\to \CN$ between $\infty$-categories is said to be \emph{1-full} if for every
$m_1,m_2\in \CC$, the map
\begin{equation} \label{e:enriched functor morphs}
\Maps_\CM(m_1,m_2)\to \Maps_{\CN}(\Phi(m_1),\Phi(m_2))
\end{equation}
is fully faithful (i.e., is an embedding of a union of connected components).

\medskip

From Proposition \ref{p:env enh} we deduce the following fact, which we will use later on:

\begin{prop}\label{p:1-full functors}
Let $\CC$ be a  $\CA$-enriched category, and let $\CM \to \CN$ be a 1-full functor of $\CA$-module categories.  Then:

\medskip

\noindent{\em(a)} The induced functor
$$\on{Funct}^{\CA}(\CC, \CM) \to \on{Funct}^{\CA}(\CC, \CN) $$
is 1-full and the essential image consists of objects $\Phi \in \on{Funct}^{\CA}(\CC, \CN)$ such that:

\smallskip

\noindent{\em(i)} $\Phi(x) \in \CM$ for all $x\in \CC$;

\smallskip

\noindent{\em(ii)} For every $c_1, c_2 \in \CC$, the morphism
$$\Hom^{\CA}_\CC(c_1,c_2) \otimes \Phi(c_1) \to \Phi(c_2) $$
given by \eqref{e:enriched functor morphs} lies in $\CM$.

\medskip

\medskip

\noindent{\em(b)} A map $\Phi' \to \Phi''$ between two such functors lies in $\on{Funct}^{\CA}(\CC, \CM)$ iff $\Phi'(c) \to \Phi''(c)$
lies in $\CM$ for all $c\in \CC$.
\end{prop}

\ssec{The enriched Yoneda lemma}\label{ss:enriched Yoneda}

Given $\CC\in \Cat^\CA$, we continue to study the $\CA$-functor
$$\on{Yon}^\CA_\CC:\CC\to \bP^\CA(\CC).$$

\sssec{}

Note that for any $c,c'\in \CC$, we have a tautoloigical identification
\begin{equation} \label{e:precomposition with Yon}
\Hom^{\CA}_\CC(c',c) \simeq \on{Yon}^\CA_\CC(c)(c'),
\end{equation}
where:

\begin{itemize}

\item $\on{Yon}^\CA_\CC(c)\in  \bP^\CA(\CC)$ is viewed as an $\CA$-functor $\CC^{\on{op}}\to \CA$;

\smallskip

\item $\on{Yon}^\CA_\CC(c)(c')$ denotes the value of the above functor on $c'$.

\end{itemize}

In particular, we obtain a canonical map
\begin{equation} \label{e:precomposition with Yon id}
\uno_\CA \to \on{Yon}^\CA_\CC(c)(c),
\end{equation}

\sssec{}

Note also that for any $c\in \CC$, evaluation on $c$ gives rise to an $\CA$-linear functor
$$\bP^\CA(\CC):=\on{Funct}^{\CA^{\on{rev}}}(\CC^{\on{op}},\CA^{\on{rev}}) \to \CA,$$
to be denoted $\on{ev}_c$.

\medskip

As with the usual Yoneda embedding, the functor $\on{Yon}^\CA_\CC$ has the following property:

\begin{lem} \label{l:enriched Yoneda}
For $c\in \CC$ and $\Phi\in \bP^\CA(\CC):=\on{Funct}^{\CA^{\on{rev}}}(\CC^{\on{op}},\CA^{\on{rev}})$, the map
$$\uHom_{\bP^\CA(\CC),\CA}(\on{Yon}^\CA_\CC(c),\Phi) \overset{\on{ev}_c}\longrightarrow
\uHom_\CA(\on{Yon}^\CA_\CC(c)(c),\Phi(c)) \overset{\text{\eqref{e:precomposition with Yon id}}}\longrightarrow \Phi(c)$$
is an isomorphism.
\end{lem}

The proof is the same as for the usual Yoneda lemma: one constructs explicitly the inverse map.

\sssec{}

As a particular case of Lemma \ref{l:enriched Yoneda}, we obtain:

\begin{cor} \label{c:enriched Yoneda}
For $c,c'\in \CC$, the map
\begin{multline*}
\uHom_{\bP^\CA(\CC),\CA}(\on{Yon}^\CA_\CC(c),\on{Yon}^\CA_\CC(c')) \overset{\on{ev}_c}\longrightarrow \\
\to \uHom_\CA(\on{Yon}^\CA_\CC(c)(c),\on{Yon}^\CA_\CC(c')(c)) \overset{\text{\eqref{e:precomposition with Yon id}}}
\longrightarrow \on{Yon}^\CA_\CC(c')(c)\simeq \Hom^{\CA}_\CC(c,c')
\end{multline*}
is an isomorphism.
\end{cor}

\ssec{More on the \texorpdfstring{$(\bP^\CA(-),\ModtoEnr)$}{adj}-adjunction}

This subsection is optional in that its contents are not used elsewhere in the paper.

\sssec{} \label{sss:counit presheaves}

Fix a regular cardinal $\kappa<\lambda_0$, set
$$\CA\mmod_{\kappa}:=\CA\mmod\underset{\Cat^{\on{pres}}}\times \Cat^{\on{pres}}_\kappa.$$

We define the functor
$$\ModtoEnr_\kappa:\CA\mmod_{\kappa}\to \Cat^\CA$$
by sending $\CM\in \CA\mmod_{\kappa}$ to the (automatically \emph{small}) category of $\kappa$-compact
objects $\CM^c_\kappa\subset \CM$, equipped with the $\CA$-enrichment inherited from that on $\ModtoEnr(\CM)$.

\sssec{}

Note that the functor $\bP^\CA(-)$ sends $\Cat^\CA$ to $\CA\mmod_{\kappa}$ for any $\kappa$ such that
$\CA\in \Cat^{\on{pres}}_\kappa$.

\medskip

We claim that we have an adjunction
$$\bP^\CA(-):\Cat^\CA\rightleftarrows \CA\mmod_{\kappa}:\ModtoEnr_\kappa.$$

\sssec{}

The unit of the adjunction is given by the enriched Yoneda functor
\begin{equation} \label{e:unit kappa}
\CC\to \bP^\CA(\CC),
\end{equation}
whose image is easily seen to lie in the subcategory of $\kappa$-compact objects.

\medskip

Let us describe the counit of the adjunction, which is by definition a map
$$\bP^\CA(\ModtoEnr(\CM)_\kappa)\to \CM.$$

\sssec{}

For a given $\CM$, we start with the enriched Yoneda functor
$$\on{Yon}^\CA_{\CM^c_\kappa}:\CM_\kappa^c \to \bP^\CA(\ModtoEnr(\CM)_\kappa),$$
and we let
$$'\!\on{Yon}^\CA_{\CM}:\CM\to \bP^\CA(\ModtoEnr(\CM)_\kappa)$$
be its extension to a colimit-preserving functor.

\medskip

In Sect. \ref{sss:left strict} we will show that the functor $'\on{Yon}^\CA_\CM$ admits a \emph{left} adjoint:
$$({}'\!\on{Yon}^\CA_\CM)^L:\bP^\CA(\ModtoEnr(\CM)_\kappa)\to \CM.$$

The functor $'\!\on{Yon}^\CA_\CM$ is naturally right-lax compatible with the action of $\CA$. Hence,
the functor $({}'\!\on{Yon}^\CA_\CM)^L$ acquires a structure of left-lax compatibility with the action of $\CA$.
However, in Sect. \ref{sss:left strict}, we will see that this left-lax compatibility is actually
strict, i.e., $({}'\!\on{Yon}^\CA_\CM)^L$ is a morphism in $\CA\mmod_\kappa$.

\sssec{}

We claim the above natural transformations indeed provide a unit and counit of the adjunction. Moreover,
for $\CC\in \Cat^\CA$ and $\CM\in \CA\mmod_\kappa$, the resulting isomorphism
$$\on{Funct}^\CA(\CC,\ModtoEnr(\CM)_\kappa) \simeq \on{Funct}_{\CA\mmod_\kappa}(\bP^\CA(\CC),\CM)$$
fits into the commutative diagram
\begin{equation} \label{e:kappa enr}
\CD
\on{Funct}_{\CA\mmod_\kappa}(\bP^\CA(\CC),\CM)  @>{\sim}>> \on{Funct}^\CA(\CC,\ModtoEnr(\CM)_\kappa)  \\
@VVV @VVV \\
\on{Funct}_{\CA}(\bP^\CA(\CC),\CM) @>{\text{\propref{p:env enh}}}>{\sim}> \on{Funct}^\CA(\CC,\ModtoEnr(\CM)).
\endCD
\end{equation}

\sssec{}

First, we claim that the composition
\begin{multline*}
\on{Funct}_{\CA\mmod_\kappa}(\bP^\CA(\CC),\CM)
\overset{(-)\circ \text{\eqref{e:unit kappa}}}\longrightarrow \on{Funct}^\CA(\CC,\ModtoEnr(\CM)_\kappa) \hookrightarrow \\
\to \on{Funct}^\CA(\CC,\ModtoEnr(\CM))
\end{multline*}
 equals
 $$\on{Funct}_{\CA\mmod_\kappa}(\bP^\CA(\CC),\CM)  \hookrightarrow \on{Funct}_{\CA}(\bP^\CA(\CC),\CM)
 \overset{\text{\propref{p:env enh}}}\longrightarrow \on{Funct}^\CA(\CC,\ModtoEnr(\CM)).$$

To check this, it is sufficient to consider the universal case, i.e., when we take $\CM:=\bP^\CA(\CC)$, and we start
with the identity functor. However, in this case, the assertion becomes tautological.

\sssec{}

It remains to show that for $\CC\in \Cat^\CA$ and an $\CA$-functor $\CC\to \CM^c_\kappa$, i.e.,
an $\CA$-enriched functor $\Phi:\CC\to \ModtoEnr(\CM)_\kappa$, the composition
$$\bP^\CA(\CC)\overset{\bP^\CA(\Phi)}\to \bP^\CA(\ModtoEnr(\CM)_\kappa)\overset{({}'\!\on{Yon}^\CA_\CM)^L}\to \CM$$
corresponds under the isomorphism of \propref{p:env enh} to the initial functor $\Phi$, followed by the embedding $\CM^c_\kappa\hookrightarrow \CM$.

\medskip

I.e., we have to show that the composite
$$\CC\overset{\on{Yon}^\CA_\CC}\to \bP^\CA(\CC)\overset{\bP^\CA(\Phi)}\to \bP^\CA(\ModtoEnr(\CM)_\kappa)\overset{({}'\!\on{Yon}^\CA_\CM)^L}\to \CM,$$
viewed as an $\CA$-functor, is isomorphic to $\Phi$, followed by the embedding $\CM^c_\kappa\hookrightarrow \CM$.

\medskip

By the functoriality of $\on{Yon}^\CA_{-}$, the above composition identifies with
$$\CC\overset{\Phi}\to \ModtoEnr(\CM)_\kappa
\overset{\on{Yon}^\CA_{\CM^c_\kappa}}\to \bP^\CA(\ModtoEnr(\CM)_\kappa)\overset{({}'\!\on{Yon}^\CA_\CM)^L}\to \CM.$$

Thus, it suffices to show that the composition
$$\CM^c_\kappa=\ModtoEnr(\CM)_\kappa \overset{\on{Yon}^\CA_{\CM^c_\kappa}}\to \bP^\CA(\ModtoEnr(\CM)_\kappa)\overset{({}'\!\on{Yon}^\CA_\CM)^L}\to \CM,$$
viewed as an $\CA$-functor, is isomorphic to the embedding $\CM^c_\kappa\hookrightarrow \CM$.

\medskip

However, this follows from the fully faithfulness of $\on{Yon}^\CA_\CM$, see Corollary \ref{c:enriched Yoneda}.

\ssec{Enriched categories via quivers} \label{ss:quiv}

In this subsection we will outline an approach to enriched categories developed in \cite{Hinich-Yoneda,Hinich-universal}.
It gives a complementary point of view on many notions discussed earlier in this section.

\medskip

A comparison between the two approaches can be found in \cite{Heine}.

\sssec{}

Let $\CA$ be an associative algebra in $\Cat^{\on{pres}}$. Let $\CS$ be an object of
$\Spc$. Consider the object
$$\CS\otimes \CA\in \CA\mmod$$
and the monoidal category
$$\Quiv_\CS(\CA):=\on{Funct}_{\CA}(\CS\otimes \CA,\CS\otimes \CA)^{\on{rev}}.$$

\medskip

By Sect. \ref{sss:free self-dual}, we can identify $\Quiv_\CS(\CA)$ as a plain category with
$$\on{Funct}(\CS\times \CS,\CA).$$

In terms of this description, the monoidal structure on $\Quiv_\CS(\CA)$ corresponds to the
\emph{convolution} monoidal structure on $\Quiv_\CS(\CA)$, i.e., the monoidal operation is
pull-push along
$$(\CS\times \CS)\times (\CS\times \CS) \leftarrow \CS\times \CS\times \CS \to \CS\times \CS,$$
where ``push" means left Kan extension.

\medskip

We will regard $\CS\otimes \CA$ as a \emph{right} module over $\Quiv_\CS(\CA)$. This (right) action
of $\Quiv_\CS(\CA)$ on $\CS\otimes \CA$ commutes with the natural (left) action
of $\CA$ on $\CS\otimes \CA$.

\sssec{}

Note that there is a canonical isomorphism
$$(\Quiv_\CS(\CA))^{\on{rev}} \overset{\tau}\simeq \Quiv_\CS(\CA^{\on{rev}}),$$
namely
\begin{multline*}
\on{Funct}_{\CA}(\CS\otimes \CA,\CS\otimes \CA)^{\on{rev}} \simeq
\on{Funct}_{\CA^{\on{rev}}}((\CS\otimes \CA)^\vee,(\CS\otimes \CA)^\vee)\overset{\text{\eqref{e:duality funct}}}\simeq \\
\simeq \on{Funct}_{\CA^{\on{rev}}}(\CS\otimes \CA^{\on{rev}},\CS\otimes \CA^{\on{rev}}).
\end{multline*}

In terms of the identifications
$$\Quiv_\CS(\CA) \simeq (\on{Funct}(\CS\times \CS,\CA))^{\on{rev}} \text{ and } \Quiv_\CS(\CA^{\on{rev}}) \simeq \on{Funct}(\CS\times \CS,\CA^{\on{rev}})^{\on{rev}},$$
it corresponds to the swap of the factors in $\CS\times \CS$, followed by the identity functor
$$\on{Funct}(\CS\times \CS,\CA) \simeq \on{Funct}(\CS\times \CS,\CA^{\on{rev}}).$$

\medskip

For an algebra object $R\in \Quiv_\CS(\CA)$, we let $\tau(R^{\on{rev}})$ denote the corresponding algebra
object in $\Quiv_\CS(\CA^{\on{rev}})$.

\sssec{}

Let $\CC$ be an $\CS$-flagged category enriched over $\CA$.

\medskip

To this data we associate an algebra object
$$R_\CC\in \Quiv_\CS(\CA)$$
by letting the value of the corresponding functor $\CS\times \CS\to \CA$
on $(s_1,s_2)\in \CS\times\CS$ be $$\Hom^\CA_{\CS,\CC}(s_1,s_2)\in \CA.$$

\sssec{}

Note that for the opposite category
$$\CC^{\on{op}}\in \Cat^{\CA^{\on{rev}}},$$
we have
$$R_{\CC^{\on{op}}}\simeq \tau(R_\CC^{\on{rev}})\in \on{Alg}(\Quiv_\CS(\CA^{\on{rev}})).$$

\sssec{}

Let $\CM$ be an $\CA$-module category. Unwinding the definitions, we obtain that the category
$\on{Funct}^\CA(\CC,\CM)$ identifies with
\begin{equation} \label{e:funct as mod}
R_\CC\mod(\on{Funct}_\CA(\CS\otimes \CA,\CM)),
\end{equation}
where we consider $\on{Funct}_\CA(\CS\otimes \CA,\CM)$ as acted on the left by
$\Quiv_\CS(\CA)$ via its right action on the source.

\sssec{}

We obtain that
$$\bP^\CA(\CC)=\tau(R_\CC^{\on{rev}})\mod(\on{Funct}_{\CA^{\on{rev}}}(\CS\otimes \CA^{\on{rev}},\CA^{\on{rev}})).$$

By \eqref{e:duality funct}, we identify
$$\on{Funct}_{\CA^{\on{rev}}}(\CS\otimes \CA^{\on{rev}},\CA^{\on{rev}})\simeq \CS\otimes \CA.$$

Under this identification, the left action of $\Quiv_\CS(\CA^{\on{rev}})$ on
$\on{Funct}_{\CA^{\on{rev}}}(\CS\otimes \CA^{\on{rev}},\CA^{\on{rev}})$ corresponds to the
right action of $\Quiv_\CS(\CA)$ on $\CS\otimes \CA$ via $\tau$.

\medskip

In particular, we obtain an identification
\begin{equation} \label{e:P via quiv}
\bP^\CA(\CC) \simeq R_\CC^{\on{rev}}\mod(\CS\otimes \CA).
\end{equation}

\sssec{}

Unwinding, we obtain that in terms of the identification \eqref{e:P via quiv}, the Yoneda functor
$$\on{Yon}^\CA_\CC:\CC\to \bP^\CA(\CC)$$
corresponds to the tautological object in
$$R_\CC\mod(\on{Funct}(\CS\otimes \CA,R_\CC^{\on{rev}}\mod(\CS\otimes \CA)),$$
whose underlying object of $\on{Funct}(\CS\otimes \CA,R_\CC^{\on{rev}}\mod(\CS\otimes \CA)$
is the induction functor
$$\ind_{R_\CC^{\on{rev}}}:\CS\otimes \CA\to R_\CC^{\on{rev}}\mod(\CS\otimes \CA).$$

\sssec{} \label{sss:proof prshv}

The latter observation makes the assertion of Proposition \ref{p:env enh} manifest:  for any $\CM\in \CA\mmod$, the map
$$\on{Funct}_{\CA}(R_\CC^{\on{rev}}\mod(\CS\otimes \CA),\CM)\to
R_\CC\mod(\on{Funct}_{\CA}(\CS\otimes \CA,\CM))$$
given by precomposition with $\ind_{R_\CC^{\on{rev}}}$ is an isomorphism: indeed, both sides are monadic
over $\on{Funct}_{\CA}(\CS\otimes \CA,\CM)$, and the above functor induces an isomorphism of monads.

\sssec{} \label{sss:prshv and ten}

Note that the interpretation of $\on{Funct}^\CA(\CC,\CM)$ given by \eqref{e:funct as mod} implies that the canonical
functor
\begin{equation} \label{e:prshv and ten}
\on{Funct}^\CA(\CC,\CA)\underset{\CA}\otimes \CM\to \on{Funct}^\CA(\CC,\CM)
\end{equation}
is an equivalence.

\medskip

Indeed, the above functor identifies with
\begin{multline*}
R_\CC\mod(\on{Funct}_\CA(\CS\otimes \CA,\CA))\underset{\CA}\otimes \CM\to
R_\CC\mod(\on{Funct}_\CA(\CS\otimes \CA,\CA)\underset{\CA}\otimes \CM)\to  \\
\to R_\CC\mod(\on{Funct}_\CA(\CS\otimes \CA,\CM)),
\end{multline*}
and both above arrows are equivalences.

\sssec{}

In the sequel, we will make use of the following fact.

\begin{prop}\label{p:duality of enriched presheaves}
The $\CA$-module category $\mathbf{P}^{\CA}(\CC)$ is dualizable (as an $\CA$-module category) with dual given by
$\mathbf{P}^{\CA^{\on{rev}}}(\CC^{\on{op}})$.
\end{prop}

\begin{proof}

It suffices to establish an equivalence
$$\on{Funct}_{\CA}(\bP^{\CA}(\CC), \CM)\simeq \bP^{\CA^{\on{rev}}}(\CC^{\on{op}})\underset{\CA}\otimes \CM, \quad \CM\in \CA\mmod.$$

We have
$$\on{Funct}_{\CA}(\bP^{\CA}(\CC), \CM)\simeq
\on{Funct}^{\CA}(\CC,\CM) \overset{\text{\eqref{e:prshv and ten}}} \simeq \on{Funct}^{\CA}(\CC,\CA) \underset{\CA}\otimes \CM \simeq
\bP^{\CA^{\on{rev}}}(\CC^{\on{op}}) \underset{\CA}\otimes \CM,$$
as required.

\end{proof}

\ssec{Bousfield-Kan formula for enriched presheaves} \label{ss:Bousfield-Kan}

In this subsection we continue to explore the point of view on enriched categories via quivers.
We use it to derive information about the category $\bP^\CA(\CC)$.

\sssec{} \label{sss:enr monadic}

Our point of departure is formula \eqref{e:P via quiv}. Hence, we obtain a monadic adjunction of $\CA$-module categories
\begin{equation} \label{e:adj for enriched}
\on{Funct}(\CS,\CA)\simeq \CS\otimes \CA \rightleftarrows
R_\CC^{\on{rev}}\mod(\CS\otimes \CA) \simeq \bP^\CA(\CC).
\end{equation}

The left adjoint in \eqref{e:adj for enriched} is the $\CA$-linear extension of the enriched
Yoneda embedding $\on{Yon}^\CA_\CC$. For the duration of this subsection, we denote it by $\bi$.

\medskip

In terms of the identification
$$\CS\otimes \CA\simeq \on{Funct}(\CS,\CA),$$
the right adjoint in \eqref{e:adj for enriched} is given by restriction along $\on{Yon}^\CA_\CC$.
For the duration of this subsection, we denote it by $\be$.

\sssec{}

From \eqref{e:object as a colimit of delta} we obtain that for $\Phi\in \on{Funct}(\CS,\CA)$,
\begin{equation} \label{e:object as a colimit of delta enriched}
\bi(\Phi)\simeq \underset{s\in \CS}{\on{colim}}\, \Phi(s)\otimes \on{Yon}^\CA_\CC(s).
\end{equation}

\sssec{} \label{sss:BK}

Let now $\Psi$ be an object $\bP^\CA(\CC)$. From \eqref{e:adj for enriched}, we obtain that
$\Psi$ can be \emph{canonically} written as a geometric realization of a simplicial object
$\on{BK}_\bullet(\Psi)$ where
\begin{equation} \label{e:BK}
\on{BK}_n(\Psi)=\bi((R_\CC^{\on{rev}})^{\otimes n}\otimes \be(\Psi)).
\end{equation}

\medskip

Combining with \eqref{e:object as a colimit of delta enriched}, we obtain:

\begin{cor} \label{c:colimit pres of enr pshv}
Every object $\Psi \in \mathbf{P}^{\CA}(\CC)$ is canonically a colimit of objects of the form
$$ a \otimes \on{Yon}^\CA_\CC(c) \in \mathbf{P}^{\CA}(\CC) ,$$
for $a\in \CA$ and $c\in \CC$.
\end{cor}

\begin{rem} \label{r:BK}

The formula for $\on{BK}_n(-)$ can be made even more explicit. Namely, for any $\Phi\in \on{Funct}(\CS,\CA)\simeq \CS\otimes \CA$
we have
$$(R_\CC^{\on{rev}})^{\otimes n}\otimes \Phi \simeq
\underset{(s_0,...,s_n)\in (\CS)^{n+1}}{\on{colim}}\, \Hom^\CA_\CC(s_0,s_1)\otimes...\otimes
\Hom^\CA_\CC(s_{n-1},s_n)\otimes \Phi(s_n)\otimes \delta_{s_0}.$$

Hence,
$$\on{BK}_n(\Psi)\simeq
\underset{(s_0,...,s_n)\in (\CS)^{n+1}}{\on{colim}}\, \Hom^\CA_\CC(s_0,s_1)\otimes...\otimes
\Hom^\CA_\CC(s_{n-1},s_n)\otimes \Psi(s_n)\otimes \on{Yon}^\CA_\CC(s_0).$$

\end{rem}

\sssec{} \label{sss:left strict}

We now return to the left adjoint of the functor $'\on{Yon}^\CA_\CM$ from Sect. \ref{sss:counit presheaves}. We will now show
that the left-lax compatibility structure $({}'\on{Yon}^\CA_\CM)^L$ with the action of $\CA$ is actually strict.

\medskip

By Corollary \ref{c:colimit pres of enr pshv}, in order to show both facts, it suffices to show
that $({}'\on{Yon}^\CA_\CM)^L$ is defined and commutes with the action of $\CA$ on objects of the form $a\otimes \on{Yon}^\CA_\CM(m)$, $a\in \CA$, $m\in \CM$.

\medskip

However, Corollary \ref{c:enriched Yoneda} implies that
$$({}'\on{Yon}^\CA_\CM)^L(a\otimes \on{Yon}^\CA_\CM(m))\simeq a\otimes m.$$

\ssec{Compatibility with the symmetric monoidal structure(s)}

In this subsection we let $\CA$ be a \emph{symmetric} monoidal category.

\sssec{} \label{sss:enr sym mon}

When $\CA$ is \emph{symmetric monoidal} (which is the case of interest in the present paper), the operation of Cartesian product of
the underlying plain categories extends to a symmetric monoidal structure on $\Cat^{\CA}$:

\medskip

For $\CC',\CC''\in \Cat^{\CA}$, $c'_1,c'_2\in \CC'$ and $c''_1,c''_2\in \CC''$, we have
$$\Hom^\CA_{\CC'\times \CC''}((c'_1,c''_1),(c'_2,c''_2)):=
\Hom^\CA_{\CC'}(c'_1,c'_2)\otimes \Hom^\CA_{\CC''}(c''_1,c''_2).$$

Note, however, that this symmetric monoidal structure on $\Cat^{\CA}$ is \emph{not} Cartesian.

\medskip

Thus, we can talk about (symmetric) monoidal $\CA$-enriched categories.

\sssec{} \label{sss:enh sym}

let $F:\CA_1\to \CA_2$ be a symmetric monoidal functor. Note that the functor $\ind^{\CA_2}_{\CA_1}$
is (strictly) symmetric monoidal.

\medskip

In particular, it sends (symmetric) monoidal $\CA_1$-enriched categories to
(symmetric) monoidal $\CA_2$-enriched categories.

\medskip

Explicitly, if we lift $\CC_1\in \Cat^{\CA_1}$ to a (commutative) algebra object in
$(\CS_1,\CC_1)\in \Cat^{\CA_1}(\Spc)$, the corresponding (commutative) algebra object
$(\CS_2,\CC_2):=\ind^{\CA_2}_{\CA_1}(\CS_1,\CC_1)$ has the property that
$\CS_2=\CS_1=\CS$ with the structure of (commutative) monoidal
coming from the (commutative) algebra structure on $(\CS_1,\CC_1)$.

\sssec{}

Assume now that $\CA$ is presentable. Note that in this case
the functor $\ModtoEnr$ has a natural right-lax symmetric monoidal structure.

\medskip

By adjunction, we obtain that the functor
$$\CC\mapsto \bP^\CA(\CC)$$
acquires a left-lax symmetric monoidal structure.

\medskip

However, from \eqref{e:P via quiv}, we obtain that the above left-lax symmetric monoidal structure
is strict.

\sssec{}

Concretely, this means that for $\CC_1,\CC_2$ there is a naturally defined $\CA$-bilinear functor
$$\bP^\CA(\CC_1)\times \bP^\CA(\CC_2)\to \bP^\CA(\CC_1\times \CC_2),$$
such that the induced morphism in $\CA\mmod$:
$$\bP^\CA(\CC_1)\underset{\CA}\otimes \bP^\CA(\CC_2)\to \bP^\CA(\CC_1\times \CC_2)$$
is an equivalence.

\sssec{} \label{sss:sym mon preshv}

In particular, we obtain that if $\CC$ is a symmetric monoidal $\CA$-enriched category, then
$\mathbf{P}^{\CA}(\CC)$ acquires a natural symmetric monoidal structure as an object of $\CA\mmod$.

\medskip

The Yoneda functor, viewed as a natural transformation
$$\on{Id}\to \ModtoEnr\circ \bP^\CA(-)$$
between right-lax symmetric monoidal functors, itself has a natural symmetric monoidal structure.

\sssec{}

In particular, if $\CC$ is a symmetric monoidal $\CA$-enriched category, the Yoneda functor
$$\on{Yon}^\CA_\CC:\CC\to  \bP^\CA(\CC)$$
has a natural symmetric monoidal structure (as a functor between symmetric monoidal $\CA$-enriched categories).

\medskip

In particular, the plain functor underlying $\on{Yon}^\CA_\CC$ has a natural symmetric monoidal structure.

\section{Weighted limits and colimits}\label{s:weighted limits}

In this section, we let
$$\CA,\,\, \Cat^\CA,\,\, \CA\mmod$$
be as in \secref{ss:enr vs ten}.

\ssec{Weighted limits} \label{ss:wt limits}

The usual notion of limits and colimits in not sufficient in the setting of enriched categories.
For this, we need the notion of weighted limits and colimits.  We consider several variants of these.

\sssec{}

For $\CC\in \Cat^\CA$ let
$$\CW: \CC \to \CA $$
be an $\CA$-functor.  Given $\CM\in \CA\mmod$ and an $\CA$-functor
$$\Phi: \CC \to \CM ,$$
we define the weighted limit of $F$ with weight $\CW$
$$\underset{\CC}{\on{lim}}^\CW\,  \Phi \in \CM $$
as follows.

\sssec{}

Using the functor $\CW$, we define the functor
$$\CW \otimes (-): \CM \to \on{Funct}^{\CA}(\CC, \CM) $$
informally given by
$$ m \mapsto (c \mapsto \CW(c) \otimes m).$$
Formally, $\CW \otimes (-)$ is defined as the composite
\begin{equation}\label{e:weight for limit}
\CM \overset{\CW \otimes \on{id}}{\longrightarrow} \on{Funct}^{\CA}(\CC, \CA) \underset{\CA}\otimes \CM \overset{\sim}\to \on{Funct}^{\CA}(\CC, \CM) .
\end{equation}

\medskip

Now given $\Phi \in \on{Funct}^{\CA}(\CC, \CM)$, the weighted limit $\underset{\CC}{\on{lim}}^\CW$
is defined as the value of the right adjoint to \Cref{e:weight for limit}, which exists by the Adjoint Functor
Theorem (\cite[Corollary 5.5.2.9]{HTT}) (indeed, \eqref{e:weight for limit} is a colimit-preserving functor between
presentable categories).

\sssec{}

Explicitly, for $m\in \CM$
\begin{equation}\label{e:presheaf weighted limit}
\Maps_{\CM}(m, \underset{\CC}{\on{lim}}^\CW\, \Phi) \simeq \Maps_{\on{Funct}^{\CA}(\CC, \CM)}(\CW(-)\otimes m, \Phi).
\end{equation}

\sssec{}

It is clear that the formation of weighted limits is functorial in $F$.  Moreover,
from \Cref{e:presheaf weighted limit}, we obtain that the formation of weighted colimits is contravariantly functorial in $\CW$.
Moreover, we have:

\begin{lem} \label{l:lim weight colim}
Let
$$\CW_I: I \to \on{Funct}^{\CA}(\CC, \CA), \quad i\mapsto \CW_i$$
be a (small) diagram of weights, and let $\CW= \underset{i\in I}{\on{colim}}\  \CW_i$.
Then for an $\CA$-functor $\Phi: \CC \to \CM$ there is a canonical isomorphism
$$\underset{\CC}{\on{lim}}^\CW\, \Phi \simeq \underset{i\in I^{\on{op}}}{\on{lim}}\, \underset{\CC}{\on{lim}}^{\CW_i}\, \Phi.$$
\end{lem}

\sssec{}\label{sss:cotensor as limit}
Suppose that $\CC = *$ (i.e. it only has one object and the object of morphisms is $\uno_{\CA}$).  In this case, $\CW$ is given
by an object $a \in \CA$ and $\Phi$ is given by an object $m\in \CM$.  In this case, the weighted limit
$$ \underset{\CC}{\on{lim}}^\CW\, \Phi \simeq m^a $$
is the \emph{cotensor} of $m$ by $a$; i.e.
$$\Maps_{\CM}(m', m^a) \simeq \Maps(a \otimes m', m).$$
In particular, if $\CM = \CA$ and $m=a'$, we have
$$ \underset{\CC}{\on{lim}}^\CW\, \Phi \simeq \uHom_\CA(a,a') $$
is the internal hom object.

\sssec{}

As in Lemma \ref{l:lim weight colim} we have
\begin{equation} \label{e:weighted lim tensor weight}
(\underset{\CC}{\on{lim}}^\CW\, \Phi)^a\simeq \underset{\CC}{\on{lim}}^{a\otimes \CW}\, \Phi, \quad a\in \CA,
\end{equation}
where $a\otimes \CW$ is obtained from $\CW$ be tensoring with $a$ on the target.

\sssec{}
Another key example of a weighted limit is when the weight is representable.  Namely, let
$$\CW \in \on{Funct}^{\CA}(\CC, \CA)$$
be of the form $\on{Yon}^{\CA^{\on{rev}}}_{\CC^{\on{op}}}(c)$ for $c\in \CC$.

\begin{prop}\label{p:ev weighted lim}
In the above situation,
$$\underset{\CC}{\on{lim}}^\CW\, \Phi \simeq \Phi(c)$$
for any $\Phi \in \on{Funct}^{\CA}(\CC, \CM)$.
\end{prop}
\begin{proof}
By Sect. \ref{sss:enr monadic}, we have an adjunction of $\CA$-module categories
$$
\xymatrix{
(-) \otimes \on{Yon}^{\CA^{\on{rev}}}_{\CC^{\on{op}}}(c): \CA \ar@<0.8ex>[r] & \ar@<0.8ex>[l] \on{Funct}^{\CA}(\CC, \CA): \be_c,
}
$$
where $\be_c$ is the functor of evaluation at $c\in \CC$ functor, i.e., $\be_c(\Phi)\simeq \Phi(c)$.
The desired result is obtained by applying the functor $(-) \underset{\CA}\otimes  \CM$ to this adjunction
and using \eqref{e:prshv and ten}.
\end{proof}

\sssec{Unenriched weighted limits}

A useful class of examples of weighted limits is when the indexing category $\CC$ is of the form
$\CC = \on{Enr}_{\CA}(I)$ for $I\in \Cat$.

\medskip

In this case, we have
$$ \on{Funct}^{\CA}(\CC, \CM) \simeq \on{Funct}(I, \CM) $$
for any $\CA$-module category $\CM$.  We will sometimes refer to weighted limits with such indexing category as unenriched weighted limits.

\medskip

There are two important classes of unenriched weighted limits.  When $I=\{\on{pt}\}$, this gives the cotensor as in
Sect. \ref{sss:cotensor as limit}.

\medskip

Another example is when the weight $\CW: I \to \CA$ is the constant functor $\uno_{\CA}$.
In this case, the weighted limit of a functor $\Phi: I \to \CM$ is the ordinary limit
$$ \underset{\CC}{\on{lim}}^\CW\, \Phi \simeq \underset{I}{\on{lim}}\, \Phi $$
in the category $\CM$.

\ssec{Further properties of weighted limits}

\sssec{}

An important aspect of weighted limits is that they are preserved by right adjoints.  Namely, suppose we have
an $\CA$-module functor
$$F: \CM_1 \to \CM_2. $$
Then for any $\CA$-enriched category $\CC$, weight $\CW: \CC \to \CA$, and an $\CA$-functor
$$\Phi: \CC \to \CM_1 ,$$
from the definition, we have a natural map
\begin{equation}\label{e:funct of weighted limits}
F(\underset{\CC}{\on{lim}}^\CW\, \Phi) \to \underset{\CC}{\on{lim}}^\CW\, (F\circ \Phi)
\end{equation}

\begin{lem}
Suppose that in the above situation, the functor $F$ admits a left adjoint which is a
strict\footnote{Recall that a left adjoint of an $\CA$-module functor is automatically $\CA$ left-lax and being strict is then a condition, see \cite[Chapter 1, Lemma 3.2.4]{GR1}}
$\CA$-module functor.  Then $F$ preserves weighted limits; i.e. the map \Cref{e:funct of weighted limits} is an isomorphism.
\end{lem}
\begin{proof}
The existence of an $\CA$-module left adjoint $F^L$ gives rise to the following commutative diagram
$$
\xymatrix{
\CM_2 \ar[r]^{F^L}\ar[d]_{\CW \otimes \on{id}_{\CM_2}} & \CM_1 \ar[d]^{\CW \otimes \on{id}_{\CM_1}} \\
\on{Funct}^{\CA}(\CC, \CA) \underset{\CA}\otimes \CM_2 \ar[r]^-{\on{id}\otimes F^L} & \on{Funct}^{\CA}(\CC, \CA) \underset{\CA}\otimes \CM_1
}
$$
of $\CA$-module categories.
Evaluating the corresponding commutative diagram obtained by passing to right adjoints on
$$ \Phi \in \on{Funct}^{\CA}(\CC, \CA) \underset{\CA}\otimes \CM_1 \simeq \on{Funct}^{\CA}(\CC, \CM_1) $$
gives the desired result.
\end{proof}

\sssec{}

We have the following basic result decomposing weighted limits into simpler ones.

\begin{prop} \label{p:weighted limits decomp}
An $\CA$-module functor $F: \CM \to \CM'$ preserves weighted limits if and only if it preserves all limits and cotensors.
\end{prop}

\begin{proof}

The ``only if" direction is obvious as limits and cotensors are special cases of weighted limits.

\medskip

Let $\CC$ be an $\CA$-enriched category, and
$$\CW \in \on{Funct}^{\CA}(\CC, \CA) \simeq \mathbf{P}^{\CA^{\on{rev}}}(\CC^{\on{op}})$$
be a weight.  By Corollary \ref{c:colimit pres of enr pshv}, $\CW$ is the colimit of objects in $\mathbf{P}^{\CA^{\on{rev}}}(\CC^{\on{op}})$ of the form
$$a \otimes \on{Yon}^{\CA^{\on{rev}}}_{\CC^{\on{op}}}(c)$$
for $a\in \CA$ and $c\in \CC$.

\medskip

Hence, by Lemma \ref{l:lim weight colim}, it suffices to consider $\CW$ of this form. Further, by \eqref{e:weighted lim tensor weight},
it suffices to consider $\CW$ of the form $\on{Yon}^{\CA^{\on{rev}}}_{\CC^{\on{op}}}(c)$.

\medskip

Now the assertion follows from Proposition \ref{p:ev weighted lim}.

\end{proof}

\sssec{}

We can make the decomposition of weighted limits into ordinary limits and cotensors explicit using Sect. \ref{sss:BK} and
and Remark \ref{r:BK}:

\medskip

Let
$$\CW \in \on{Funct}^{\CA}(\CC, \CA) \simeq \mathbf{P}^{\CA^{\on{rev}}}(\CC^{\on{op}})$$
be a weight functor.  By Sect. \ref{sss:BK}, $\CW$ is the geometric realization of the simplicial object $\on{BK}_\bullet(\CW)$,
\begin{multline*}
\on{BK}_n(\CW) \simeq \bi((\tau(R_\CC))^{\otimes n}\otimes \be(\CW))\simeq \\
\simeq \underset{(s_0,...,s_n)\in \CS^{n+1}}{\on{colim}}\, \Hom^\CA_\CC(s_0,s_1)\otimes...\otimes
\Hom^\CA_\CC(s_{n-1},s_n)\otimes \CW(s_0)\otimes \on{Yon}^{\CA^{\on{rev}}}_{\CC^{\on{op}}}(s_n).
\end{multline*}

Hence,
$$\underset{\CC}{\on{lim}}^\CW\, \Phi \simeq \on{Tot}(\underset{\CC}{\on{lim}}^{\on{BK}_\bullet(\CW)}\, \Phi),$$
where
$$\underset{\CC}{\on{lim}}^{\on{BK}_n(\CW)}\, \Phi \simeq
\underset{(s_0,...,s_n)\in \CS^{n+1}}{\on{lim}}\,
\Phi(s_n)^{\Hom^\CA_\CC(s_0,s_1)\otimes...\otimes \Hom^\CA_\CC(s_{n-1},s_n)\otimes \CW(s_0)}.$$

\ssec{Weighted colimits}  \label{ss:wt colimits}

The notion of a weighted colimit is dual to that of a weighted limit.
However, because our notion of $\CA$-module categories is built on the symmetric monoidal category of
presentable categories, there is some asymmetry in the discussion that follows.

\sssec{}

Let $\CC$ be $\CA$-enriched category and
$$\Phi: \CC \to \CM $$
be an $\CA$-enriched functor, and an $\CA^{\on{rev}}$-functor $\CW: \CC^{\on{op}} \to \CA^{\on{rev}}$, a ``weight.''
We can think of $\CW$ as an object of $\bP^\CA(\CC)$.

\medskip

By the universal property of the Yoneda embedding,
the functor $\Phi$ induces a functor of $\CA$-module categories
$$ \wt\Phi: \mathbf{P}^{\CA}(\CC) \to \CM .$$
Following \cite[Sect. 6]{Hinich-universal}, we define
$$ \underset{\CC}{\on{colim}}^\CW \, \Phi := \wt\Phi(\CW) .$$

\sssec{}

Let us describe the above construction more explicitly.  Namely, we claim:

\begin{prop} \label{p:weighted colims}
We have a canonical isomorphism
$$\Maps_\CM(\underset{\CC}{\on{colim}}^\CW \, \Phi, m) \simeq
\Maps_{\on{Funct}^{\CA^{\on{rev}}}(\CC^{\on{op}},\CA^{\on{rev}})}(\CW,\uHom_{\CM,\CA}(\Phi(-),m)).$$
\end{prop}

\begin{proof}

The statement of the proposition is equivalent to the assertion that the functor right adjoint to
$\wt\Phi$ sends $m\in \CM$ to the object
$$\uHom_{\CM,\CA}(\Phi(-),m)\in \on{Funct}^{\CA^{\on{rev}}}(\CC^{\on{op}},\CA^{\on{rev}})=\bP^\CA(\CC).$$

\medskip

By the (Yoneda) \lemref{l:enriched Yoneda},
given $F\in \bP^\CA(\CC)$, the $\CA$-functor $\CC^{\on{op}}\to \CA^{\on{rev}}$ corresponding to it is given by
$$c\mapsto \uHom_{\bP^\CA(\CC),\CA}(\on{Yon}^\CA_\CC(c),F).$$

Hence, for $F:=(\wt\Phi)^R(m)$, the corresponding $\CA$-functor $\CC^{\on{op}}\to \CA^{\on{rev}}$ is
$$\uHom_{\bP^\CA(\CC),\CA}(\on{Yon}^\CA_\CC(c),(\wt\Phi)^R(m)),$$
which by adjunction identifies with
$$\uHom_{\CM,\CA}(\wt\Phi(\on{Yon}^\CA_\CC(c)),m).$$

However, by construction
$$\wt\Phi(\on{Yon}^\CA_\CC(-))=\Phi(-),$$
and the assertion follows.

\end{proof}

\sssec{}

By definition, weighted colimits commute with colimits in the weight.  Namely, if we have a diagram
$$\CW_I: I \to \mathbf{P}^{\CA}(\CC), \quad i\mapsto \CW_i $$
of weights, for any functor $\Phi: \CC \to \CM$, we have a natural isomorphism
$$ \underset{i\in I}{\on{colim}} \,\, \underset{\CC}{\on{colim}}^{\CW_i}\, \Phi \simeq \underset{\CC}{\on{colim}}^{\CW}\, \Phi $$
where $\CW= \underset{i\in I}{\on{colim}}\,  \CW_i$.

\sssec{}

In the case that $\CC = *$ with $\CW$ given by $a\in \CA$ and $\Phi$ by $m \in \CM$, we have:
$$\underset{\CC}{\on{colim}}^\CW \, \Phi \simeq a \otimes m .$$

\medskip

Moreover, for a general $\CC$ and $\CW = \on{Yon}^\CA_\CC(c)$, $c\in \CC$, we have
$$\underset{\CC}{\on{colim}}^\CW \, \Phi \simeq \Phi(c).$$

\sssec{}
As in the case of weighted limits, we can describe weighted colimits explicitly in terms of ordinary limits and tensors
using Sect. \ref{sss:BK} as a geometric realization:
$$ \underset{\CC}{\on{colim}}^\CW \, \Phi \simeq | \underset{\CC}{\on{colim}}^{\on{BK}_\bullet(\CW)}\, \Phi |,$$
where $\underset{\CC}{\on{colim}}^{\on{BK}_\bullet(\CW)}\, \Phi$ is the simplicial object in $\CM$ with terms given by
$$\underset{(c_0,...,c_n)\in \CS^{n+1}}{\on{colim}}\, \Hom^{\CA}_\CC(c_0,c_1)\otimes...\otimes
\Hom^{\CA}_\CC(c_{n-1},c_n)\otimes \CW(c_n)\otimes \Phi(c_0).$$

\sssec{}
It follows immediately from the definition of weighted colimits that all $\CA$-module functors
$$ \CM_1 \to \CM_2 $$
preserve weighted colimits.

\ssec{Weighted colimits and enriched presheaves}\label{ss:weighted colim and presheaves}

\sssec{}\label{sss:Yoneda generates weighted}

An key feature of weighted colimits is that every object in $\mathbf{P}^{\CA}(\CC)$ is canonically a weighted colimit.
Namely, by definition, we have that for $F \in \mathbf{P}^{\CA}(\CC)$,
\begin{equation}\label{e:Yoneda weighted colim}
F \simeq \underset{\CC}{\on{colim}}^F\, \on{Yon}^\CA_\CC,
\end{equation}
where $\on{Yon}^\CA_\CC: \CC \to \mathbf{P}^{\CA}(\CC)$ is the enriched Yoneda embedding.

\sssec{}

We will say that an object $m\in \CM$ is \emph{totally $\CA$-compact} if the functor
$$\uHom_{\CM,\CA}(m,-): \CM \to \CA $$
preserves all weighted colimits.

\sssec{}

We have the following evident characterization of totally $\CA$-compact objects:

\begin{lem} \label{c:crit totally compact}
Let $\CM$ be an $\CA$-module category and $m\in \CM$.  The following are equivalent:

\smallskip

\noindent{\em(i)} $m$ is totally $\CA$-compact;

\smallskip

\noindent{\em(ii)} The functor $ (-) \otimes m: \CA \to \CM$
admits an $\CA$-module right adjoint\footnote{I.e., we need that the right adjoint as a plain functor
preserves colimits, and that the natural stricture on it of right-lax compatibility with the action of $\CA$
be strict.};

\smallskip

\noindent{\em(iii)} $\uHom_{\CM,\CA}(m,-): \CM \to \CA$
is an $\CA$-module functor.

\end{lem}

\sssec{} \label{sss:yoneda im cpct}

Note that the essential image of the enriched Yoneda functor
$$\on{Yon}^\CA_\CC: \CC \to \mathbf{P}^{\CA}(\CC) $$
consists of totally $\CA$-compact objects.

\medskip

Indeed, this follows from (i) $\Leftrightarrow$ (iii) in Lemma \ref{c:crit totally compact} and
Lemma \ref{l:enriched Yoneda}.

\sssec{}

The following assertion is useful:

\begin{prop}\label{p:fully faithful from enriched pshv}
Let $\CC$ be an $\CA$-enriched category and let
$$\wt\Phi: \mathbf{P}^{\CA}(\CC) \to \CM $$
be an $\CA$-module functor such that:

\smallskip

\noindent{\em(i)} $\wt\Phi$ is fully faithful in the enriched sense\footnote{I.e., it induces an isomorphism on $\uHom_{-,\CA}(-,-)$.}
on the essential image of the enriched Yoneda functor;

\smallskip

\noindent{\em(ii)} For every $c\in \CC$, the object $\Phi(c)\in \CM$ is totally compact, where
$\Phi:=\wt\Phi\circ \on{Yon}^\CA_\CC$.

\medskip

Then $\wt\Phi$ is fully faithful.
\end{prop}

\begin{proof}

Since $\mathbf{P}^{\CA}(\CC)$ is generated under weighted colimits by the objects $\on{Yon}^\CA_\CC(c)$
for $c\in \CC$, and $\wt\Phi$ preserves weighted colimits, by Proposition \ref{p:weighted colims},
it suffices to show that
for $c \in \CC$ and $F\in \mathbf{P}^{\CA}(\CC)$, the map
\begin{equation}\label{e:ff iso}
\uHom_{\mathbf{P}^{\CA}(\CC),\CA}(\on{Yon}^\CA_\CC(c), F) \to 
\uHom_{\CM,\CA}(\wt\Phi\circ \on{Yon}^\CA_\CC(c), \wt\Phi(F))
\end{equation}
is an isomorphism.

\medskip

Now, by \Cref{e:Yoneda weighted colim}, we have
$$F \simeq \underset{c'\in \CC}{\on{colim}}^{F}\, \on{Yon}^\CA_\CC(c').$$

Since both sides in \eqref{e:ff iso} preserve weighted colimits in $F$, we are reduced to showing that
$$\uHom_{\mathbf{P}^{\CA}(\CC),\CA}(\on{Yon}^\CA_\CC(c),\on{Yon}^\CA_\CC(c')) \to
\uHom_{\CM,\CA}(\wt\Phi\circ \on{Yon}^\CA_\CC(c),\wt\Phi\circ \on{Yon}^\CA_\CC(c'))$$
is an isomorphism, which follows from the assumption.

\end{proof}

\begin{cor}\label{c:enriched pshv fully faithful}
Let $\CC_1 \to \CC_2$ be a functor of $\CA$-enriched categories that is fully faithful in
the enriched sense\footnote{I.e., it induces an isomorphism on $\Hom^\CA(-,-)$.}.
Then the corresponding functor
$$ \mathbf{P}^{\CA}(\CC_1) \to \mathbf{P}^{\CA}(\CC_2) $$
is fully faithful.
\end{cor}

\begin{cor}\label{c:equiv to enriched pshv}
Let $\CC$ be an $\CA$-enriched category.  An $\CA$-module functor
$$\wt\Phi: \mathbf{P}^{\CA}(\CC) \to \CM $$
is an equivalence if and only if

\smallskip

\noindent{\em(i)} $\wt\Phi$ is fully faithful in the enriched sense on essential image of the enriched Yoneda functor;

\smallskip

\noindent{\em(ii)} For every $c\in \CC$, the object $\Phi(c)\in \CM$ is totally compact, where
$\Phi:=\wt\Phi\circ \on{Yon}^\CA_\CC$;

\smallskip

\noindent{\em(iii)} The image of $\CC$ generates $\CM$ under weighted colimits.

\end{cor}

\section{Adjunctions} \label{s:adj}

\ssec{Adjunctions in 2-categories}

\sssec{}
Let $\bS$ be a 2-category.  Recall that a 1-morphism $f: x \to y$ in $\bS$ admits a right (resp. left) adjoint if there exists another 1-morphism
$$ f^R: y \to x \mbox{ (resp. } f^L: y \to x \mbox{)}$$
and 2-morphisms
$$ f^L\circ f \to \on{id}_x \ \ \mbox{(resp. } f\circ f^R \to \on{id}_y \mbox{), and} $$
$$ \on{id}_y \to f \circ f^L\ \  \mbox{(resp. } \on{id}_x \to f^R \circ f \mbox{)} $$
satisfying the standard relations.

\sssec{}

Let $\on{Adj}$ denote the universal adjunction, i.e., this is a 2-category such that the \emph{space} of functors
$$ \on{Adj} \to \bS$$
for any 2-category $\bS$ is the space of adjunctions in $\bS$, i.e. the subspace of 1-morphisms of $\bS$ that admit a right adjoint.
There are natural functors
\begin{equation}\label{e:incl of adjoint}
\rho, \lambda: [1] \to \on{Adj}
\end{equation}
selecting the right and left adjoint, repectively.

\medskip
For a 2-category $\bS$, we will be interested in a description of the 2-category of functors
$$ \on{Funct}(\on{Adj}, \bS) .$$

\sssec{}\label{sss:BC conditions}
Recall the following notion.  Suppose that
\begin{equation}\label{e:square for BC}
\xymatrix{
x \ar[r]^f\ar[d]_p & y \ar[d]^q \\
z \ar[r]^g & w
}
\end{equation}
is a commutative square in $\bS$.  We say that \Cref{e:square for BC} satisfies the right (resp. left) Beck-Chevalley condition if the 1-morphisms $f$ and $g$ admit right (resp. left) adjoints and the corresponding 2-morphism
$$g^L \circ q \to p \circ f^L \mbox{ (resp. } p \circ f^R \to g^R \circ q \mbox{)} $$
is an isomorphism.

\sssec{}
We have the following basic fact (see e.g. \cite[Lemma B.5.9]{AMR} and \cite[Theorem D]{AGH}):

\begin{prop}\label{p:basic bc}
Let $\bS$ be a 2-category and let
\begin{equation}\label{e:inclusion of adjoint}
\on{Funct}(\on{Adj}, \bS) \to \on{Funct}([1], \bS)
\end{equation}
denote the restriction along the right (resp. left) adjoint \Cref{e:incl of adjoint}.  The functor \Cref{e:inclusion of adjoint}
is 1-full:

\smallskip

\noindent{\em(a)} Its essential image consists of objects $x\to y$ that admit a right (resp. left) adjoint.

\smallskip

\noindent{\em(b)} A 1-morphism
\begin{equation}\label{e:morph in arrows}
\xymatrix{
x \ar[r]\ar[d] & y \ar[d] \\
x' \ar[r] & y'
}
\end{equation}
between two such objects $x\to y$ and $x'\to y'$ lies in the image iff the square \Cref{e:morph in arrows} satisfies the right (resp. left) Beck-Chevalley condition.
\end{prop}

\ssec{Adjunctions in enriched presheaves} \label{ss:presheaves}

\sssec{}\label{sss:enriched 2cat}

Let $\CA$ be a presentable monoidal category.

\medskip

We will be interested in the case when $\CA$ has the structure of 2-category (e.g. $\CA = \DGCat_\kappa$) and the resulting
2-categorical aspects of $\CA$-module categories.

\medskip

Namely, suppose that $\CA$ is itself tensored over $\Cat^{\on{pres}}_\kappa$ (for some regular cardinal $\kappa$), i.e., $\CA$ is a $\Cat^{\on{pres}}_\kappa$-algebra
in categories $\Cat^{\on{pres}}$. In particular, this endows every $\CA$-module category with the structure of
2-category.\footnote{Since $\Cat^{\on{pres}}_\kappa$ is presentable, it follows that every $\Cat^{\on{pres}}_\kappa$-module category in $\Cat^{\on{pres}}$
has $\Cat^{\on{pres}}_\kappa$-valued homs.}


\begin{prop}\label{p:adjoints in V-functors}
Let $\CC$ be an $\CA$-enriched category and let $\CM$ be an $\CA$-module category.  A 1-morphism $f: \Phi_1 \to \Phi_2$ in $\on{Funct}^{\CA}(\CC, \CM)$
admits a right (resp. left) adjoint if and only if the following two conditions are satisfied:

\smallskip

\noindent{\em(i)}
For every $c\in \CC$, the 1-morphism
$$ f(c): \Phi_1(c) \to \Phi_2(c) $$
in $\CM$ admits a right (resp. left) adjoint;

\smallskip

\noindent{\em(ii)}
For all $c', c'' \in \CC$, the square
$$
\xymatrix{
\uHom_{\CC,\CA}(c',c'') \otimes \Phi_1(c') \ar[r]\ar[d] & \uHom_{\CC,\CA}(c',c'') \otimes \Phi_2(c'')\ar[d] \\
\Phi_1(c'') \ar[r] & \Phi_2(c'')
}
$$
satisfies the right (resp. left) Beck-Chevalley condition.

\end{prop}

\begin{proof}

We will prove the assertion for right adjoints (the case of left adjoints is identical).  The 1-morphism $f$ is given by a functor
$$ [1] \to \on{Funct}^{\CA}(\CC, \CM) \simeq \on{Funct}^{\CA}(\CC, \CA) \underset{\CA}\otimes  \CM .$$
By \Cref{p:duality of enriched presheaves}, such a functor is given by a functor of $\CA$-module categories
\begin{equation}\label{e:1-morph in pshv}
\mathbf{P}^{\CA}(\CC) \to \on{Funct}([1], \CM) .
\end{equation}
The 1-morphism $f$ admits a right adjoint iff the functor \Cref{e:1-morph in pshv} factors through the functor
\begin{equation}\label{e:incl of adj fun}
\on{Funct}(\on{Adj}, \CM) \to \on{Funct}([1], \CM),
\end{equation}
which is 1-full by \Cref{p:basic bc}.

\medskip

By Proposition \ref{p:env enh}, we have
$$ \on{Funct}_{\CA}(\mathbf{P}^{\CA}(\CC),\on{Funct}([1], \CM)) \simeq \on{Funct}^{\CA}(\CC, \on{Funct}([1], \CM)). $$
Now, the 1-morphism $f$ admits a right adjoint if and only if the corresponding object in the category $\on{Funct}^{\CA}(\CC, \on{Funct}([1], \CM))$ lies in
the subcategory $\on{Funct}^{\CA}(\CC, \on{Funct}(\on{Adj}, \CM))$.
By \Cref{p:1-full functors} and \Cref{p:basic bc}, the functor
$$\on{Funct}^{\CA}(\CC, \on{Funct}(\on{Adj}, \CM)) \to \on{Funct}^{\CA}(\CC, \on{Funct}([1], \CM))$$
is 1-full and the essential image is as claimed.

\end{proof}

\end{document}